\documentclass{c0PFH}

\begin{document}

\title{Proof of the simplicity conjecture}
\date{\today}
\author{Daniel Cristofaro-Gardiner, Vincent Humili\`ere, Sobhan Seyfaddini}
\maketitle

\begin{abstract}
In the 1970s, Fathi, having proven that the group of compactly supported volume-preserving homeomorphisms of the $n$-ball is simple for $n \ge 3$, asked if the same statement holds in dimension $2$.   We show that the group of compactly supported area-preserving homeomorphisms of the two-disc is not simple.  This settles what is known as the ``simplicity conjecture'' in the affirmative.  In fact, we prove the a priori stronger statement that this group is not perfect.  

An important step in our proof involves verifying for certain smooth 
twist maps a conjecture of Hutchings concerning
recovering the Calabi invariant from the asymptotics of spectral
invariants defined using periodic Floer homology.   Another key step,
 which builds on recent advances in continuous symplectic
  topology, involves proving that these spectral invariants extend continuously to area-preserving homeomorphisms of the disc.   These two properties of PFH spectral invariants  are potentially of independent interest.  
 
 Our general strategy is partially inspired by suggestions of Fathi and the approach of Oh
towards the simplicity question.  In particular, we show that infinite twist maps,  studied by Oh, are not finite energy homeomorphisms, which resolves the ``infinite twist conjecture" in the affirmative; these twist maps are now the first examples of Hamiltonian homeomorphisms which can be said to have infinite energy.  Another consequence of our work is that various forms of fragmentation for volume preserving homeomorphisms which hold for higher dimensional balls fail in dimension two.
\end{abstract}

\tableofcontents

\section{Introduction}

\subsection{The main theorem} \label{sec:intro_main_theo}

Let $(S,\omega)$ be a surface equipped with an area form.  An {\bf area-preserving homeomorphism} is a homeomorphism which preserves the measure induced by $\omega$. 
Let $\Homeo_c(\D, \omega)$ denote the group of area-preserving
homeomorphisms of the two-disc which are the identity near the
boundary.  Recall that a group is {\bf simple} if it does not have a
non-trivial proper normal subgroup.  The following fundamental
question was raised\footnote{ The history of when  this question was raised seems complicated.  It is asked by Fathi in the paper \cite{fathi}.  However, it has been suggested to us by Ghys  that the question dates back to considerably earlier, possibly to Mather or Thurston. } in the 1970s:   

\begin{question}\label{que:disc}
Is the group $\Homeo_c(\D, \omega)$ simple?
 \end{question}  

Indeed, the algebraic structure of the group of volume-preserving homeomorphisms has been well-understood in dimension at least three since the work of Fathi \cite{fathi} from the 70s; but, the case of surfaces, and in particular Question~\ref{que:disc}, has long remained mysterious.

Question~\ref{que:disc} has been the subject of wide interest.  For
example, it is highlighted in the plenary ICM address of Ghys
\cite[Sec.\ 2.2]{Ghys_ICM}; it appears on McDuff and Salamon's list of
open problems \cite[Sec.\ 14.7]{mcduff-salamon}; it has been
  one of the main motivations behind the development of
  $C^0$-symplectic topology, which we will further discuss in Section
  \ref{sec:simplicity_high_dim}; for other examples, see \cite{Banyaga, fathi, Ghys_lecture, Bounemoura, LeRoux10, LeRoux-6Questions, EPP}.  It has generally been believed since the early $2000s$ that the group $\Homeo_c(\D, \omega)$ is not simple: McDuff and Salamon refer to this as the {\bf simplicity conjecture}.  Our main  theorem resolves this conjecture in the affirmative.

\begin{theo}
\label{thm:main}
The group $\Homeo_c(\D, \omega)$ is not simple.  
\end{theo}

In fact, we can obtain an a priori stronger result.  Recall that a group $G$ is called {\bf perfect} if its commutator subgroup  $[G,G]$ satisfies $[G, G] = G$.  The commutator subgroup $[G, G]$ is a normal subgroup of $G$.  Thus, every non-abelian simple group is perfect. However, in the case of certain transformation groups, such as $\Homeo_c(\D, \omega)$, a general argument due to Epstein and Higman \cite{Epstein, Higman} implies that perfectness and simplicity are equivalent;  see Proposition \ref{prop:commutators}. Hence, we obtain the following corollary.

\begin{corol}\label{corol:not_perfect}
The group $\Homeo_c(\D, \omega)$ is not perfect.
\end{corol}

We remark that in higher dimensions, the analogue of Theorem~\ref{thm:main} contrasts our main result: by \cite{fathi}, the group $\Homeo_c(\D^n, \mathrm{Vol})$ of compactly supported volume-preserving homeomorphisms of the $n$-ball is simple for $n \ge 3$.  It also seems that the structure of $\Homeo_c(\D,\omega)$ is radically different from that of the group $\Diff_c(\D,\omega)$ of compactly supported area-preserving diffeomorphisms, as we will review below.

\subsubsection{Background} %

To place our main result in its appropriate context, we begin by reviewing the long and interesting history surrounding the question of simplicity for various subgroups of homeomorphism groups of manifolds.  
Our focus will be on compactly supported homeomorphisms/diffeomorphisms of manifolds without boundary in the component of the identity\footnote{The simplicity question is interesting only for compactly supported maps in the identity component, because this is a normal subgroup of the larger group.  The group $\Homeo_c(\D, \omega)$ coincides with its identity component.}.   

In the 1930s, in the ``Scottish Book", Ulam asked if the identity component of the group of homeomorphisms of the $n$--dimensional sphere is simple.  In 1947,  Ulam and von Neumann announced in an abstract \cite{Ulam-vonNeumann}  a solution to the question in the Scottish Book in the case $n=2$.    %

In the 50s, 60s, and 70s, there was a flurry of activity on this question and related ones.   

First, the works of Anderson \cite{Anderson}, Fisher \cite{Fisher}, Chernavski, Edwards and Kirby \cite{EK} led to the proof of simplicity of the identity component in the group of compactly supported homeomorphisms of any manifold.  These developments led Smale to ask if the identity component in the group of  compactly supported diffeomorphisms of any manifold is simple \cite{Epstein}.  This question was answered affirmatively by Epstein \cite{Epstein},  Herman \cite{Herman}, Mather \cite{ MatherI, Mather74, MatherII}, and Thurston \cite{thurston}\footnote{More precisely, Epstein, Herman and Thurston settled the question in the case of smooth diffeomorphisms, while Mather resolved the case of $C^r$ diffeomorphisms for $r < \infty$ and $r \neq \mathrm{dim(M)}+1.$ The case of $r = \mathrm{dim(M)}+1$ remains open.}.

The connected component of the identity in volume-preserving, and symplectic,  diffeomorphisms admits a homomorphism, called  {\bf flux},  to a certain abelian group.  Hence, these groups are not simple when this homomorphism is non-trivial.  Thurston proved, however, that the kernel of flux is simple in the volume-preserving setting for any manifold of dimension at least three; see \cite[Chapter 5]{Banyaga_book}.  In the symplectic setting,   Banyaga \cite{Banyaga} then proved that this group is simple when the symplectic manifold is closed;  otherwise, it is not simple as it admits a non-trivial homomorphism, called {\bf Calabi}, and Banyaga showed that the kernel of Calabi 
is always simple.  We will recall the definition of Calabi in the case of the disc in Section \ref{sec:prop-norm-subgr}. 

As mentioned above, the simplicity of the identity component in compactly supported volume-preserving homeomorphisms is well-understood in dimensions greater than two,\footnote{The group is trivial in dimension $1$.} thanks to the article \cite{fathi}, in which Fathi shows that, in all dimensions, the group admits a homomorphism, called ``mass-flow";  moreover, the kernel of mass-flow is simple in dimensions greater than two.  On simply connected manifolds, the mass-flow homomorphism is trivial, and so the group is indeed simple in dimensions greater than two ---  in particular, $\Homeo_c(\D^n, \mathrm{Vol})$, the higher dimensional analogue of the group in Question \ref{que:disc}, is simple for $n>2$, as stated above.   

Thus, the following rather simple picture emerges from the above %
cases of the simplicity question.  In the non-conservative setting, the connected component of the identity is simple.  In the conservative setting, there always exists a natural homomorphism (flux, Calabi, mass-flow) which obstructs the simplicity of the group.  However, the kernel of the homomorphism is always simple.  

\subsubsection{The case of surfaces and our case of the disc}

Despite the clear picture above, established by the end of the 70s, the case of area-preserving homeomorphisms of surfaces has remained unsettled  --- the simplicity question has remained open for the disc and more generally for the kernel of the mass-flow homomorphism\footnote{We review the mass-flow homomorphism and discuss possible ideas that build on the present work for establishing simplicity of its kernel in Section \ref{sec:questions_simplicity}.}  
 --- underscoring the importance of answering Question~\ref{que:disc}.  

In fact, the case of area-preserving homeomorphisms of the disc does seem drastically different.   For example, the natural homomorphisms flux, Calabi, and mass-flow mentioned above that obstruct simplicity are all continuous with respect to a natural topology on the group; however, we will show in Corollary \ref{corol:no_continuous_homomorphism}  that there can not exist a continuous homomorphism out of $\Homeo_c(\D, \omega)$ 
with a proper non-trivial kernel, when $\Homeo_c(\D, \omega)$ is equipped with the $C^0$-topology; we will review the $C^0$-topology in \ref{sec:c0}.

Of course,  $\Homeo_c(\D, \omega)$ could still admit a discontinuous homomorphism. For example,  Ghys asks, in his %
ICM address  \cite{Ghys_ICM}, whether the Calabi invariant extends to $\Homeo_c(\D, \omega)$;  this, and related works of Fathi and Oh, are further discussed in Section \ref{sec:extending_Calabi}.   However, as far as we know, even if one does not demand that the homomorphism is continuous, there is still no natural geometrically constructed homomorphism out of $\Homeo_c(\D, \omega)$ that would in any sense be an analogue of the flux, Calabi, or mass-flow; rather, as we will see, our proof of non-simplicity goes by explicitly constructing a non-trivial normal subgroup and showing that it is proper.   It seems likely to us that $\Homeo_c(\D, \omega)$ indeed has a more complicated algebraic structure.  

We will now see in the discussion of Le Roux's work below 
another way in which the  
case of area-preserving homeomorphisms of the disc is quite different.

\subsubsection{``Lots" of normal subgroups and the failure of fragmentation} %
\label{sec:lots}

\sloppy Le Roux \cite{LeRoux10} has previously studied the simplicity question for $\Homeo_c(\D, \omega)$, and this provides useful context for our work; it is also valuable to combine his conclusions with our Theorem~\ref{thm:main}.  

As Le Roux explains, Fathi's proof of simplicity in higher-dimensions does not work in dimension $2$, because it relies on the following ``fragmentation'' result: any element of $\Homeo_c(\D^n, \mathrm{Vol})$ can be written as the product of two others, each of which are supported on topological balls of three-fourths of the total volume; this is not true on the disc. 

  Building on this, Le Roux constructs a whole family $P_\rho$, for $0 < \rho \le 1$, of quantitative fragmentation properties for $\Homeo_c(\D, \omega)$:  the property $P_\rho$ asserts that there is some $\rho' < \rho$ and a positive integer $N$, such that any group element supported on a topological disc of area $\rho$ can be written as a product of $N$ elements, each supported on discs of area at most $\rho'$; we refer the reader to \cite{LeRoux10} for more details. 
For example, if Fathi's fragmentation result held in dimension $2$, then $ P_1$ would hold.

Le Roux then establishes the following alternative: if any one of these fragmentation properties holds, then $\Homeo_c(\D, \omega)$ is simple; otherwise, there is a huge number of proper normal subgroups, constructed in terms of ``fragmentation norms''.  Thus, in view of our Theorem~\ref{thm:main}, we have not just one proper normal subgroup but ``lots" of them; for example, combining Le Roux's work  \cite[Cor.\ 7.1]{LeRoux10}  with our Theorem~\ref{thm:main} yields the following.

\begin{corol} 
\label{cor:leroux}
Every compact\footnote{As above, we are working in the $C^0$-topology.} subset of $\Homeo_c(\D,\omega)$ is contained in a proper normal subgroup.
\end{corol}

As Le Roux explains \cite[Section 7]{LeRoux10}, this is ``radically'' different from the situation for the group $\Diff_c(\mathbb{D},\omega)$ of compactly supported area-preserving diffeomorphisms of the disc with its usual topology, where the normal closure of many one-parameter subgroups is the entire group; a similar argument can be used to show that Corollary~\ref{cor:leroux} is false for the other cases of the simplicity question reviewed above.

In view of our Theorem~\ref{thm:main}, we also have the following result about the failure\footnote{One should stress that a weaker version of fragmentation does hold, even in dimension $2$, for example if we remove the requirement in the definition of $P_\rho$ that the product has $N$ elements then the corresponding property holds for all $\rho$ by \cite{fathi}.}  of fragmentation in dimension $2$, which follows by invoking \cite[Thm.\ 0.2]{LeRoux10}.  

\begin{corol}
\label{cor:leroux2}
None of the fragmentation properties $ P_\rho$ for $0 < \rho \le 1$ hold in $\Homeo_c(\D, \omega)$.
\end{corol}
This generalizes a result of Entov-Polterovich-Py \cite[Sec.\ 5.1]{EPP}, who established Corollary~\ref{cor:leroux2} for $1/2 \le \rho \le 1.$

\subsection{Idea of the proof}

Our goal here is to give an outline of the proof of Theorem~\ref{thm:main}.  Along the way, we state several other results of potentially independent interest, including a resolution of the ``Infinite Twist Conjecture".

\subsubsection{A proper normal subgroup of $\Homeo_c(\D, \omega)$}\label{sec:prop-norm-subgr}

To prove Theorem~\ref{thm:main}, we will define below a normal subgroup of $\Homeo_c(\D, \omega)$ which is a variation on the construction of Oh-M\"uller \cite{muller-oh}.  We will show that this normal subgroup is proper, and in fact, contains the commutator subgroup of $\Homeo_c(\D, \omega)$.  

Denote by $C^{\infty}_c(\S^1 \times \D)$ the set of time-dependent functions, also referred to as  {\bf Hamiltonians}, of the disc whose supports are contained in the interior of $\D$.  As will be recalled in Section \ref{sec:area-pres-diffeos}, one can associate to every $H \in C^{\infty}_c(\S^1 \times \D)$ a {\bf Hamiltonian flow} $\varphi^t_H$. Furthermore, every area-preserving diffeomorphism of the disc $\phi \in \Diff_c(\D, \omega)$ is a Hamiltonian diffeomorphism in the sense that there exists $H \in C^{\infty}_c(\S^1 \times \D) $ such that $\phi = \varphi^1_H$.  

The energy, or the {\bf Hofer norm}, of a Hamiltonian $H \in C^{\infty}_c(\S^1 \times \D)$ is defined by the quantity \[ \| H \|_{(1, \infty)} = \int_0^1 \left( \max_{x \in \D} H(t,\cdot) - \min_{x \in \D} H(t, \cdot)\right) dt.\]

\begin{definition}\label{def:finite_energy_homeo}
An element $\phi \in \Homeo_c(\D, \omega)$ is a {\bf finite-energy homeomorphism} if there exists a sequence of smooth Hamiltonians $H_i \in C^{\infty}_c(\S^1 \times \D)$ such that the sequence $\| H_i \|_{(1, \infty)}$ is bounded, {\it i.e.}\  $\exists \, C \in \R$ such that $\| H_i \|_{(1, \infty)} \leq C$, and the Hamiltonian diffeomorphisms $\varphi^1_{H_i}$ converge uniformly to $\phi$.  We will denote the set of all finite-energy homeomorphisms by $\FHomeo_c(\D, \omega)$.
\end{definition}
We clarify the topology for convergence in the definition, namely the $C^0$ topology, in Section \ref{sec:c0}.

We will show in Proposition \ref{prop:FHomeo_subgp} and Corollary \ref{corol:fhomeo_commutators} that $\FHomeo_c(\D, \omega)$ is a normal subgroup of $\Homeo_c(\D, \omega)$ that contains the commutator subgroup of   $\Homeo_c(\D, \omega)$.  Thus, Theorem~\ref{thm:main} will follow from the following result that we show:
 
 \begin{theo}\label{theo:FHomeo_properness}
 $\FHomeo_c(\D, \omega)$ is a proper normal subgroup of $\Homeo_c(\D, \omega)$.
 \end{theo}
 
\begin{remark}\label{rem:Hameo}
 In defining $\FHomeo_c(\D, \omega)$ as above, we were inspired by the article of Oh and M\"uller \cite{muller-oh}, who defined a normal subgroup of $\Homeo_c(\D, \omega),$ denoted by $\Hameo_c(\D, \omega)$, which is usually referred as the group of {\bf hameomorphisms}; see Section \ref{sec:extending_Calabi} for its definition.   It has been conjectured that $\Hameo_c(\D, \omega)$ is a proper normal subgroup of $\Homeo_c(\D, \omega)$; see for example  \cite[Question 4.3]{muller-oh}.

It can easily be verified  that $\Hameo_c(\D, \omega) \subset \FHomeo_c(\D, \omega)$.  Hence, it follows from the above theorem that   $\Hameo_c(\D, \omega)$ is a proper normal subgroup of  $\Homeo_c(\D, \omega)$.   
 \end{remark}
 
 \subsubsection{Infinite energy homeomorphisms}
 \label{sec:infinitedistance}

  Theorem \ref{theo:FHomeo_properness} gives an
   affirmative answer to  \cite[Question 1]{LeRoux-6Questions} where
   Le Roux asks if there exist area-preserving homeomorphisms of the
   disc which are ``infinitely far in Hofer's distance" from
   area-preserving diffeomorphisms.

To elaborate, for any $\phi \in \Homeo_c(\D, \omega) \setminus \FHomeo_c(\D,
\omega)$ and any sequence $H_i \in C^{\infty}_c(\S^1 \times \D)$ of
smooth Hamiltonians such that the Hamiltonian diffeomorphisms
$\varphi^1_{H_i} \to \phi$ uniformly, then, $\Vert H_i\Vert_{(1,
  \infty)} \to \infty$.   This implies in particular that for any diffeomorphisms $\varphi_i \in \Diff_c(\D,\omega)$ converging to $\phi$, the $\varphi_i$ become arbitrarily far from the identity in the Hofer metric on $\Diff_c(\D,\omega)$, hence Le Roux's formulation above; we will review the Hofer metric in \ref{sec:C0_continuity}.

We remark  that we can also regard such a $\phi$ as having ``infinite Hofer energy".  Prior to our work, it was not known whether or not such infinite energy maps existed. %
  
In the next section, we will see explicit examples of $\phi$ that we will show are  in $\Homeo_c(\D, \omega) \setminus \FHomeo_c(\D,\omega)$.

\subsubsection{The Calabi invariant and the infinite twist}\label{sec:calabi_inf_twist}
The hard part of Theorem~\ref{theo:FHomeo_properness} is to show properness.  Here we describe the key example of an area-preserving homeomorphism that we later show is not in $\FHomeo_c(\D, \omega)$. 

We first summarize some background that will motivate what follows.   As mentioned above, for smooth, area-preserving compactly supported two-disc diffeomorphisms, non-simplicity is known, via the Calabi invariant.    More precisely,  the {\bf Calabi invariant} of $\theta \in \Diff_c(\D, \omega)$ is defined as follows.  Pick any Hamiltonian $H \in C^{\infty}_c(\S^1 \times \D)$ such that $\theta = \varphi^1_H$.  Then, 

$$\Cal(\theta) := \int_{\S^1}\int_{\D} H  \, \omega \; dt.$$ 
It is well-known that the above integral does not depend on the choice of $H$ and so $\Cal(\theta)$ is well-defined; it is also known that $\Cal: \Diff_c(\D, \omega) \rightarrow \R$ is a non-trivial group homomorphism, {\it i.e.} $ \Cal(\theta_1 \theta_2) = \Cal(\theta_1) + \Cal(\theta_2)$.
For further details on the Calabi homomorphism see \cite{calabi,mcduff-salamon}.
 
We will need to know the value of the Calabi invariant for the following class of area-preserving diffeomorphisms.   Let $f :[0, 1]  \to \R$ be a smooth function vanishing near $1$ and define $\phi_f \in \Diff_c(\D, \omega)$  by $\phi_f(0) := 0$ and $
\phi_f(r, \theta) := (r, \theta + 2 \pi f(r)).$  If the function $f$ is taken to be  (positive/negative) monotone, then the map $\phi_f$ is referred to as a \textbf{ (positive/negative) monotone twist}. Since we will be working exclusively with positive monotone twists, we will assume  monotone twists are all positive, unless otherwise stated. %

Now suppose that $\omega = \frac{1}{2\pi} rdr \wedge d\theta$.   A simple computation (see our conventions in Section \ref{sec:prelim_symp}) shows that $\phi_f$ is the time--1 map of the flow of the Hamiltonian defined by 
\begin{equation}\label{eq:hamiltonian_radial_twist}
F(r, \theta) =  \int_r^1 s f(s) ds.
\end{equation}
From this we compute:
\begin{equation}\label{eq:calabi_radial_twist}
\Cal(\phi_f) =   \int_0^1 \int_r^1 sf(s) ds \;  rdr.
\end{equation}

We can now introduce the element that will not be in  $\FHomeo_c(\D, \omega)$. Let $f:(0,1] \rightarrow \R$ be a smooth function which vanishes near $1$, is decreasing, and satisfies $\displaystyle \lim_{r \to 0} f(r) = \infty$.  Define $\phi_f \in \Homeo_c(\D, \omega)$  by $\phi(0) := 0$ and
 \begin{equation}\label{eq:inf_twist}
\phi_f(r, \theta) := (r, \theta + 2 \pi f(r)).
\end{equation}
 It is not difficult to see that $\phi_f$ is indeed an element of  $\Homeo_c(\D, \omega)$  which is in fact smooth away from the origin.  We will refer to $\phi_f$ as an {\bf infinite twist.}
 
 We will show that if
 \begin{equation}  \label{eq:infinite_calabi}
\int_0^1 \int_r^1 sf(s) ds \; r dr = \infty,
 \end{equation}
 then 
 \begin{equation}\label{eqn:inftwist_notin_FHomeo}
 \phi_f \notin \FHomeo_c(\D, \omega). 
 \end{equation}

\subsubsection{The infinite twist conjecture}

The idea outlined in the above 
section is inspired by Fathi's
suggestion that the Calabi homomorphism could extend to the subgroup
$\Hameo_c(\D, \omega)$, the normal subgroup constructed by Oh and
M\"uller mentioned in Remark~\ref{rem:Hameo}; see Conjecture 6.8 in Oh's article
\cite{Oh10}; we discuss this further in Section \ref{sec:extending_Calabi}. Moreover, \cite[Thm.\ 7.2]{Oh10} shows that if the Calabi invariant extended to  $\Hameo_c(\D, \omega)$, then any infinite twist $\phi_f$, satisfying \eqref{eq:infinite_calabi}, would not be in $ \Hameo_c(\D, \omega)$.

Hence,  it seemed reasonable to conjecture %
that $\phi_f \notin \Hameo_c(\D, \omega)$ if \eqref{eq:infinite_calabi} holds.  Indeed, McDuff and Salamon refer to this\footnote{The actual formulation in \cite{mcduff-salamon} is slightly different than this, because it does not include the condition \eqref{eq:infinite_calabi}.  However, without this condition, one can produce infinite twists that lie in $\Hameo_c(\D, \omega)$, by the following argument. Pick a function $f$ with $\lim_{r\to 0}f(r)=\infty$ such that the function $F$ from (\ref{eq:hamiltonian_radial_twist}) is bounded. For such a function $f$, the map $\phi_f$ satisfies the assumptions given in \cite{mcduff-salamon}, but it can be approximated by smooth monotone twist maps in a standard way, showing that $\phi_f$ actually belongs to $\Hameo_c(\D,\omega)$.  The authors of \cite{mcduff-salamon} have confirmed in private communication with us that imposing condition \eqref{eq:infinite_calabi} is consistent with what they intended.}  as the {\bf Infinite Twist Conjecture} and it is Problem 43 on their list of open problems; see \cite[Section 14.7]{mcduff-salamon}.  Since, as stated in Remark~\ref{rem:Hameo}, $\Hameo_c(\D, \omega) \subset \FHomeo_c(\D, \omega)$, we obtain the following Corollary from \eqref{eqn:inftwist_notin_FHomeo}.

\begin{corol}[``Infinite Twist Conjecture"]
Any infinite twist $\phi_f$ satisfying \eqref{eq:infinite_calabi} %
is not in $\Hameo_c(\D, \omega)$.
\end{corol}

Infinite twists can be defined on any symplectic manifold, and we discuss them further in \ref{sec:fhomeoq} in the context of future open questions.  

\subsubsection{Spectral invariants from periodic Floer homology}\label{sec:PFH_spectralinvariants_intro} 
To show that the infinite twist is not in $\FHomeo_c(\D, \omega)$, we use the theory of {\bf periodic Floer homology} (PFH), reviewed in Section \ref{sec:pfhs2}.  PFH is a version of Floer homology for area-preserving diffeomorphisms which was introduced by Hutchings \cite{Hutchings-index, Hutchings-Sullivan-Dehntwist}.  As with ordinary Floer homology, PFH can be used to define a sequence of functions $c_d : \Diff_c(\D, \omega) \rightarrow \R$, where $d\in \N$, called {\bf spectral invariants}, which satisfy various useful properties.   We give the definition of $c_d$ in Section \ref{sec:PFH_spec_initial_properties}, see in particular Remark \ref{rem:PFHspec-disc}.

The definition of PFH spectral invariants is due to Michael Hutchings \cite{Hutchings_unpublished}, but very few properties have been established about these.  We will prove in  Theorem \ref{thm:PFHspec_initial_properties} that the PFH spectral invariants satisfy the following properties:
\begin{enumerate}
\item Normalization: $c_d(\id) = 0$, 
\item Monotonicity: Suppose that $H \leq G$ where $H, G \in C^{\infty}_c(\S^1 \times \D)$.  Then, $c_d(\varphi^1_H) \leq c_d(\varphi^1_G)$ for all $d \in \N$,

\item Hofer Continuity: $|c_d(\varphi^1_H) - c_d(\varphi^1_G) | \leq d \Vert H-G \Vert_{(1, \infty)}$,

\item Spectrality: $c_{d}(\varphi^1_H) \in \Spec_d(H)$ for any $H \in \mathcal{H}$, where $\Spec_d(H)$ is the {\bf order $d$ spectrum} of $H$ 	and is defined in Section \ref{sec:action_spectra}.
\end{enumerate}

A key property, which allows us to use the PFH spectral invariants for studying homeomorphisms (as opposed to diffeomorphisms) is the following theorem, which we prove in Section \ref{sec:C0_continuity} via the methods of
continuous symplectic topology. 

\begin{theo}\label{theo:C0_continuity}
The spectral invariant $c_d : \Diff_c(\D, \omega) \rightarrow \R$ is continuous with respect to the $C^{0}$ topology on  $\Diff_c(\D, \omega)$.  Furthermore, it extends continuously to $\Homeo_c(\D, \omega)$.
\end{theo}

Another key property is the following  which was originally conjectured in greater generality by Hutchings \cite{Hutchings_unpublished}:

\begin{theo}\label{theo:PFHspec_Calabi_property} 
The PFH spectral invariants $c_d : \Diff_c(\D, \omega) \rightarrow \R$ satisfy the {\bf Calabi property}
\begin{equation}
\label{eqn:calabi}
\lim_{d\to \infty} \frac{c_d(\varphi)}{d} = \Cal(\varphi)
\end{equation}
if $\varphi$ is a monotone twist map of the disc.
\end{theo}

The property \eqref{eqn:calabi} can be thought of as a kind of analogue of the ``Volume Property" for ECH spectral invariants proved in \cite{CGHR}.  ECH has many similarities to PFH, which is part of the motivation for conjecturing that something like \eqref{eqn:calabi} might be possible.

\begin{remark}\label{rem:hutchings_conjecture}
In fact, Hutchings \cite{Hutchings_unpublished} has conjectured that the Calabi property in Theorem \ref{theo:PFHspec_Calabi_property} holds more generally for all $\varphi \in \Diff_c(\D, \omega)$.  The point is that we verify this conjecture for monotone twists; and, this is sufficient for our purposes.  Hutchings' conjecture, if true in full generality, would have further implications for our understanding of the algebraic structure of $\Homeo_c(\D,\omega)$, which we further discuss in \ref{sec:fhomeoq}.  
 \end{remark}

\subsubsection{Proof of Theorem \ref{theo:FHomeo_properness}} \label{sec:proof_properness}

We will now give the proof of Theorem \ref{theo:FHomeo_properness}, assuming the results stated above ---  we will prove these later in the paper --- and the forthcoming Proposition~\ref{prop:FHomeo_subgp}, which states that $\text{FHomeo}_c$ is a normal subgroup.

We begin by explaining the basic idea.  As was already explained above, the challenge with our approach is to show that the infinite twist is not in $\FHomeo_c(\mathbb{D},\omega)$.  Here is how we do this.  Theorem \ref{theo:C0_continuity} allows to define the PFH spectral invariants for any $\psi \in \Homeo_c(\D, \omega)$.  We will show, by using the  Hofer Continuity property, that if $\psi$ is a finite-energy homeomorphism then the sequence of PFH spectral invariants $\{c_d(\psi)\}_{d\in \N}$ grows at most linearly.  On the other hand, in the case of  an infinite twist $\phi_f$, satisfying the condition in Equation \eqref{eq:infinite_calabi},  the sequence $\{c_d(\phi_f)\}_{d\in \N}$ has super-linear growth, as a consequence of the Calabi property \eqref{eqn:calabi}.  From this we can conclude that $\phi_f \notin  \FHomeo_c(\D, \omega)$, as desired.

The details are as follows.  We begin with the following lemma which tells us that for a finite-energy homeomorphism $\psi$ the sequence of PFH spectral invariants $\{c_d(\psi)\}_{d\in \N}$ grows at most linearly.

\begin{lemma}\label{lem:linear_growth}
Let  $\psi \in \FHomeo_c(\D, \omega)$ 
be a finite-energy homeomorphism.  Then, there exists a constant $C$, depending on $\psi$, such that $$\frac{c_d(\psi)}{d} \leq C, \, \forall d \in \N.$$
\end{lemma}
\begin{proof}
By definition, $\psi$ being a finite-energy homeomorphism implies that there exist  smooth Hamiltonians $H_i \in C^{\infty}_c(\S^1 \times \D)$ such that the sequence $\| H_i \|_{(1, \infty)}$ is bounded, {\it i.e.}\  $\exists \, C \in \R$ such that $\| H_i \|_{(1, \infty)} \leq C$, and the Hamiltonian diffeomorphisms $\varphi^1_{H_i}$ converge uniformly to $\psi$.   

The Hofer continuity property  and the fact that $c_d(\id) = 0$ imply that 
 $$c_d(\varphi^1_{H_i}) \leq d \| H_i \|_{(1, \infty)} \leq d C,$$ for each $ d \in \N$.

  On the other hand, by Theorem \ref{theo:C0_continuity} $c_d(\psi) = \lim_{i \to \infty} c_d(\varphi^1_{H_i})$.  We conclude from the above inequality that $c_d(\psi) \leq dC$ for each $ d \in \N$.
\end{proof}

We now turn our attention to showing that the PFH spectral invariants
of an infinite twist $\phi_f$, which satisfies Equation
\eqref{eq:infinite_calabi},  violate the inequality from the above lemma.  We will need the following.

\begin{lemma}\label{lem:approx_monotone_twists}
  There exists a sequence of smooth monotone twists $\phi_{f_i}\in \Diff_c(\D, \omega)$ satisfying the following properties:   
\begin{enumerate}
\item The sequence $\phi_{f_i}$ converges in the $C^0$ topology to $\phi_f$,
\item There exist Hamiltonians  $F_i$, compactly supported in the interior of the disc $\D$, such that $\varphi^1_{F_i}= \phi_{f_i}$ and $F_i \leq F_{i+1}$,
\item $\displaystyle \lim_{i\to \infty} \Cal(\phi_{f_i}) = \infty$.
\end{enumerate}
\end{lemma} 
\begin{proof}
Recall that $f$ is a decreasing function of $r$ which vanishes near $1$ and satisfies $\displaystyle \lim_{r \to 0} f(r)  = \infty$.   It is not difficult to see that we can pick smooth functions $f_i: [0,1] \rightarrow \R$ satisfying the following properties:
\begin{enumerate}
\item $f_i = f$ on $[\frac{1}{i}, 1]$,
\item $f_i \leq f_{i+1}$.
\end{enumerate}
Let us check that the monotone twists $\phi_{f_i}$ satisfy the requirements of the lemma.  To see that they converge to $\phi_f$, observe that $\phi_f$ and $\phi_{f_i}$ coincide outside the disc of radius $\frac{1}{i}$.  Hence, $\phi_f^{-1} \phi_{f_i}$ converges uniformly to $\id$ because it is supported in the disc of radius $\frac{1}{i}$. 
Next, note that by Formula \eqref{eq:hamiltonian_radial_twist}, $\phi_{f_i}$ is the time--1 map of the Hamiltonian flow of $ F_i(r, \theta) =  \int_r^1 s f_i(s) ds$.   Clearly, $F_i \leq F_{i+1}$ because $f_i \leq f_{i+1}$.  Finally,  by Formula \eqref{eq:calabi_radial_twist} we have 
$$ \Cal(\phi_{f_i}) = \int_0^1 \int_r^1 sf_i(s) ds \; rdr \geq   \int_{\frac{1}{i}}^1 \int_r^1 sf_i(s) ds \; rdr  =  \int_{\frac{1}{i}}^1 \int_r^1 sf(s) ds \; r dr.$$
Recall that $f$ has been picked such that $\int_0^1 \int_r^1 sf(s) ds \; rdr = \infty$; see Equation \eqref{eq:infinite_calabi}.  We conclude that $\displaystyle \lim_{i \to \infty}  \Cal(\phi_{f_i}) = \infty$. 
\end{proof}

We will now use Lemma \ref{lem:approx_monotone_twists} to complete the proof of Theorem \ref{theo:FHomeo_properness}. By  the upcoming Proposition~\ref{prop:FHomeo_subgp}, $\text{FHomeo}_c$ is a normal subgroup; it is certainly non-trivial, so it remains to show it is proper.  

By the Monotonicity property, we have  $c_d(\phi_{f_i}) \leq c_d(\phi_{f_{i+1}})$ for each $d\in \N$. Since $\phi_{f_i}$ converges in $C^0$ topology to $\phi_f$, we conclude from Theorem \ref{theo:C0_continuity} that $c_d(\phi_{f}) = \lim_{i \to \infty} c_d(\phi_{f_i})$.  Combining the previous two lines we obtain the following inequality:
$$
c_d(\phi_{f_i}) \leq c_d(\phi_{f}), \forall d,i \in \N.
$$

  Now the Calabi property of Theorem \ref{theo:PFHspec_Calabi_property} tells us that 
$\lim_{d\to \infty} \frac{c_d(\phi_{f_i})}d = \Cal(\phi_{f_i})$.  Combining this with the previous  inequality we get $ \Cal(\phi_{f_i}) \leq \lim_{ d\to \infty} \frac{c_d(\phi_f)}{d}$ for all $i$.  Hence, by the third item in Lemma~\ref{lem:approx_monotone_twists} 
$$\lim_{ d\to \infty} \frac{c_d(\phi_f)}{d} = \infty,$$ and so by Lemma~\ref{lem:linear_growth} $\phi_f$ is not in $\FHomeo_c(\D, \omega)$.  

\begin{remark}\label{rem:only_need_inequality}
The proof outlined above does not use the full force of Theorem \ref{theo:PFHspec_Calabi_property}; it only uses the fact that $ \lim_{d\to \infty} \frac{c_d(\varphi)}{d} \geq  \Cal(\varphi) $. \end{remark}

\subsubsection*{Organization of the paper}
In Section \ref{sec:prelim_symp}, we review some of the necessary background from symplectic geometry, especially the case of surfaces.  In Section \ref{sec:pfhs2} of the paper we review the construction of periodic Floer homology and the associated spectral invariants.  Some of the properties of PFH spectral invariants, such as Hofer continuity and spectrality, are proven in Section \ref{sec:PFH_spec_initial_properties}.  We prove Theorem \ref{theo:C0_continuity}, on $C^0$ continuity of PFH spectral invariants, in Section \ref{sec:C0_continuity}.  Section \ref{sec:PFH_monotone_twist} of the article begins with a more precise version of Theorem \ref{theo:PFHspec_Calabi_property}.  The rest of Section \ref{sec:PFH_monotone_twist} is dedicated to the development of a combinatorial model of the periodic Floer homology of monotone twists.  
 In Section \ref{sec:comuting_spec_invariant}, we first use the aforementioned combinatorial model of PFH to compute the PFH spectral invariants for monotone twists.  We then use this computation to prove that the PFH spectral invariants for monotone twists satisfy the Calabi property of Theorem~\ref{theo:PFHspec_Calabi_property}; this will be carried out in Section \ref{sec:proof_Calabi_property}.

\subsubsection*{Acknowledgments} 
We are grateful to Michael Hutchings for many useful conversations, in particular for explaining the theory of PFH spectral invariants to us and for explaining his conjecture on the Calabi property of PFH spectral invariants.  We warmly thank Lev Buhovsky for helpful conversations  which helped us realize that the Calabi property of PFH spectral invariants could not possibly hold  for spectral invariants satisfying the ``standard" spectrality axiom on the disc.  We are grateful to Fr\'ed\'eric Le Roux for many helpful discussions on fragmentations.  We thank Albert Fathi for sharing a copy of his unpublished manuscript \cite{Fathi-Calabi} with us and for helpful comments.   We also thank Abed Bounemoura, Igor Frenkel, \'Etienne Ghys, Helmut Hofer, Dusa McDuff, Yong-Geun Oh, Dietmar Salamon, Felix Schlenck, Takashi Tsuboi, and  Chris Wendl  for helpful communications. 

DCG is partially supported by NSF grant DMS 1711976 and the Minerva Research Foundation.  DCG is extremely grateful to the Institute for Advanced Study for providing a wonderful atmosphere in which much of this research was completed.  The research leading to this article began in June, 2018 when DCG was a ``FSMP Distinguished  Professor" at the Institut Math\'ematiques de Jussieu-Paris Rive Gauche (IMJ-PRG).  DCG is grateful to the Fondation Sciences Math\'ematiques de Paris (FSMP) and IMJ-PRG for their support and hospitality.  DCG also thanks Kei Irie, Peter Kronheimer, and Dan Pomerleano for helpful discussions.

VH thanks Fr\'ed\'eric Le Roux for hours of discussions related to the problem studied here, which started in 2006.  VH is partially supported by the ANR project ``Microlocal'' ANR-15-CE40-0007.

SS: I am grateful to Lev Buhovsky, Fr\'ed\'eric Le Roux, Yong-Geun Oh and Claude Viterbo from whom I have learned a great deal about the various aspects of the simplicity question.   This material is partially based upon work supported by the NSF under Grant No. DMS-1440140 while I was in residence at the MSRI in Berkeley during the Fall, 2018 semester; this support at MSRI also allowed for valuable in-person discussions with DCG.  This project has received funding from the European Research Council (ERC) under the European Union’s Horizon 2020 research and innovation program (grant agreement No. 851701).

\section{Preliminaries about the symplectic geometry of surfaces}\label{sec:prelim_symp} 
Here we collect some basic facts, and fix notation, concerning two-dimensional symplectic geometry and diffeomorphism groups.

\subsection{Symplectic form on the disc and sphere}
\label{sec:sphere}
Let $\S^2 := \{(x,y, z) \in \R^3 : x^2 + y^2 + z^2 = 1\} \subset \R^3$ and $\D:= \{(x,y)\in \R^2: x^2 + y^2 \leq 1\}$.  
  We equip the sphere $\S^2$ with the symplectic form $\omega:= \frac{1}{4 \pi}d\theta \wedge dz$, where $(\theta,z)$ are cylindrical coordinates on $\mathbb{R}^3$.  Note that with this form, $\S^2$ has area $1.$  Let 
\[ S^+ = \{ (x,y,z) \in \S^2 : z \geq 0 \},\]
be the northern hemisphere in $\S^2$.  In certain sections of the paper, we will need to identify the disc $\D$ with $S^+$.   To do this, we will take the embedding $\iota: \D \to \S^2$ given by the formula 
\begin{equation}\label{eq:identify_disc_north_hemisphere}
\iota(r, \theta) = (\theta,1-r^2),
\end{equation}
where $(r, \theta)$ denotes the standard polar coordinates on $\R^2$.  We will equip the disc with the area form given by the pullback of $\omega$ under $\iota$; explicitly, this is given by the formula $\frac{1}{2\pi} rdr\wedge d\theta$.  We will denote this form by $\omega$ as well.  Note that this gives the disc a total area of $\frac{1}{2}$.  

Any area form on $\S^2$ or $\D$ is equivalent to the above differential forms, up to multiplication by a constant. 

\subsection{The $C^0$ topology}     
\label{sec:c0}

Here we fix our conventions and notation concerning the $C^0$ topology.
 
Denote by $\Homeo(\S^2)$  the group of homeomorphisms of the sphere and by $\Homeo_c(\D)$ the group of homeomorphisms of the disc whose support is contained in the interior of $\D$.
 
 Let $d$ be a Riemannian distance on $\S^2$.  Given two maps $\phi, \psi :\S^2 \to \S^2,$ we denote
$$d_{C^0}(\phi,\psi)= \max_{x\in \S^2}d(\phi(x),\psi(x)).$$  We will say that a sequence of maps $\phi_i : \S^2 \rightarrow \S^2$  {\bf converges uniformly}, or $C^0$--converges, to $\phi$, if $d_{C^0}(\phi_i, \phi) \to 0$ as $ i \to \infty$. As is well known, the notion of $C^0$--convergence does not depend on the choice of the Riemannian metric.  The topology induced by $d_{C^0}$ on $\Homeo(\S^2)$ is referred to as the {\bf $C^0$  topology}.    

The $C^0$ topology on $\Homeo_c(\D)$ is defined analogously as the topology induced by the distance 
$$d_{C^0}(\phi,\psi)= \max_{x\in \D}d(\phi(x),\psi(x)).$$

  \subsection{Hamiltonian diffeomorphisms} \label{sec:area-pres-diffeos} 
  Let  
  \[ \Diff(\S^2, \omega):=\{\theta \in \Diff(\S^2):  \theta^*\omega =\omega\}\] 
  denote the group of area-preserving, in other words symplectic, diffeomorphisms of the sphere.   
  
Let  $C^{\infty}(\S^1 \times \S^2)$ 
denote the set of smooth time-dependent Hamiltonians on $\S^2$.  As alluded to in the introduction, a smooth Hamiltonian $H \in C^{\infty} (\S^1 \times \S^2)$  gives rise to a time-dependent vector field $X_H$, called the {\bf Hamiltonian vector field}, defined via the equation
\[ \omega(X_{H_t}, \cdot) = dH_t.\]  

The {\bf Hamiltonian flow} of $H$, denoted by  $\varphi^t_H$, is by definition the flow of $X_H$.  A {\bf Hamiltonian diffeomorphism} is a diffeomorphism which arises as the time-one map of a Hamiltonian flow.  It is easy to verify that every Hamiltonian diffeomorphism of $\S^2$ is area-preserving.  And, as is well-known,  every area-preserving diffeomorphism of the sphere is in fact a Hamiltonian diffeomorphism.

As for the disc, as mentioned in the introduction, every $\theta \in \Diff_c(\D, \omega)$ is Hamiltonian, in the sense that one can find $H \in C^{\infty}_c(\S^1 \times \D)$ such that $\theta = \varphi^1_H$, where the notation is as in the sphere case.  Here, $C^{\infty}_c(\S^1 \times \D)$ denotes the set of Hamiltonians of $\D$ whose support is compactly contained in $\D$. 

Note that $\Diff(\S^2, \omega) \subset \Homeo_0(\S^2, \omega)$ and $\Diff_c(\D, \omega) \subset \Homeo_c(\D, \omega)$.  It is well-known that $\Diff(\S^2, \omega)$ and $\Diff_c(\D, \omega)$  are dense, with respect to the $C^0$ topology, in $\Homeo_0(\S^2, \omega)$ and $\Homeo_c(\D, \omega)$, respectively.

\paragraph{Hamiltonians for monotone twists.} Recall the definition of monotone twists from Section \ref{sec:calabi_inf_twist}.  Every monotone twist $\varphi$, being an element of $\Diff_c(\D, \omega)$, is a Hamiltonian diffeomorphism of the disc and so can be written as the time--$1$  map $\varphi^1_H$ of the Hamiltonian flow of some $H \in C^\infty_c(\S^1 \times \D)$.  The Hamiltonian $H$ can be picked to be of the form 
\[ H = \frac{1}{2} h(z),\]
where $h: \S^2 \rightarrow \R$ is a function of $z$ satisfying \[ h' \geq 0, h'' \geq 0, h(-1) = 0, h'(-1) = 0.\]

\subsection{Rotation numbers}\label{sec:rotation_number}
Let $p$ be a fixed point of  $\varphi \in \Diff(\S^2, \omega)$.  One can find a  Hamiltonian isotopy $\{\varphi^t\}_{t \in [0,1]}$ such that $\varphi^0 = \id, \varphi^1 = \varphi$ and $\varphi^t(p) = p$ for all $t \in [0, 1]$.   In this section, we briefly review the definitions and some of the properties of  $\rot(\{\varphi^t\}, p)$, the rotation number of the isotopy  $\{\varphi^t\}_{t \in [0,1]}$ at $p$, and  $\rot(\varphi, p),$ the rotation number of $\varphi$ at $p$.  Our conventions are such that  $\rot(\varphi, p)$ will be a real number in the interval $(-\frac12, \frac12]$.  For further details on this subject, we refer the reader to \cite{katok-hasselblatt}.  

  The derivative of the isotopy $D_p\varphi^t: T_p \S^2 \rightarrow T_p\S^2$, viewed as a linear isotopy of $\R^2$, induces an isotopy of the circle $\{f_t\}_{t\in[0,1]}$ with $f_0 = \id$ and $f_1 = D_p \varphi$.  We define $\rot(\{\varphi^t \}, p)$, the {\bf rotation number of the isotopy $(\varphi^t)_{t \in [0,1]}$ at $p$},  to be the Poincar\'e rotation number of the circle isotopy $\{f_t\}$.   

The number $\rot(\{\varphi^t \}, p)$ depends only on $\{D_p \varphi^t\}_{t\in [0,1]}$ and it satisfies the following properties.  Let $\{\psi^t\}_{t\in[0,1]}$ be another Hamiltonian isotopy such that $\psi^0 = \id, \psi^1 = \varphi^1$, and $\psi^t(p) = p$ for $t\in[0,1]$.  Then, 
\begin{enumerate}
\item $\rot(\{\varphi^t \}, p) - \rot(\{\psi^t \}, p) \in \Z$,
\item $\rot(\{\varphi^t \}, p) = \rot(\{\psi^t \}, p)$ if the two isotopies are homotopic rel. endpoints among isotopies fixing the point $p$.
\end{enumerate} 
The above facts may be deduced from the standard properties of the Poincar\'e rotation number; see, for example,  \cite{katok-hasselblatt}.

Lastly, we define  $\rot(\varphi, p)$, the {\bf rotation number of $\varphi$ at $p$}, to be the unique number in $(-\frac12, \frac12]$ which coincides with  $\rot(\{\varphi^t\}, p)$ mod $\Z$.  It follows from the discussion in the previous paragraph that it depends only on $D_p\varphi$. 

\medskip

\subsection{The action functional and its spectrum}\label{sec:action_spectra}
Spectral invariants take values in the ``action spectrum".  We now explain what this spectrum is.

Denote by $\Omega:= \{z: \S^1 \rightarrow \S^2 \}$ the space of all loops in $\S^2$.  By a {\bf capping} of a loop $z: \S^1 \rightarrow \S^2,$ we mean a map 
\[ u : D^2 \rightarrow \S^2,\] 
such that  $u|_{\partial D^2} = z$.  We say two cappings $u, u'$ for a loop $z$ are {\bf equivalent} if $u, u'$ are homotopic rel $z$.  Henceforth, we will only consider cappings up to this equivalence relation.  Note that given a capping $u$ of a loop $z$, all other cappings of $z$ are of the form $u\# A$ where $A \in \pi_2(\S^2)$ and $\# $ denotes the operation of connected sum.  
 
A {\bf capped loop} is a pair $(z,u)$ where $z$ is a loop and $u$ is a capping for $z$.  We will denote by $\tilde{\Omega}$ the space of all capped loops in the sphere. 
    
  Let $H \in C^{\infty} (\S^1 \times \S^2)$ denote a smooth Hamiltonian in $\S^2$.   Recall that  $\mathcal{A}_H: \tilde{\Omega} \rightarrow \mathbb{R}$, the {\bf action functional} associated to the Hamiltonian $H$ is defined by
 \begin{equation}
 \label{eqn:actiondefn}
 \mathcal{A}_H(z,u) =  \int_{0}^{1} H(t,z(t))dt \text{ } +  \int_{D^2} u^*\omega.
 \end{equation}
Note that $\mathcal{A}_H(z,u \# A ) = \mathcal{A}_H(z,u ) + \omega(A),$  for every $A \in \pi_2(\S^2)$.

The set of critical points of $\mathcal{A}_H$, denoted by $\Crit(\mathcal{A}_H)$, consists of capped loops  $(z,u) \in \tilde{\Omega}$ such that $z$ is a $1$--periodic orbit of the Hamiltonian flow $\varphi^t_H$.  We will often refer to such  $(z,u)$ as a capped $1$--periodic orbit of $\varphi^t_H$.

The {\bf action spectrum} of $H$, denoted by $\Spec(H)$, is the set of critical values of $\mathcal{A}_H$; it has Lebesgue measure zero.   It turns out that the action spectrum $\Spec(H)$ is independent of $H$ in the following sense: If  $H'$ is another Hamiltonian such that $\varphi^1_H = \varphi^1_{H'}$, then  there exists a constant $C \in \R$ such that  
$$
\Spec(H) = \Spec(H') + C,$$
 where $\Spec(H') + C$ is the set obtained from $\Spec(H')$ by adding the value $C$ to every element of $\Spec(H')$.   \cite[Lemma 3.3]{Schwarz} proves this in the case where $\omega$ vanishes on $\pi_2(M)$ and the proof generalizes readily to general symplectic manifolds.  %
 Moreover, it follows from the proof of \cite[Lemma 3.3]{Schwarz} that if $H, H'$ are supported in the northern hemisphere $ S^+ \subset \S^2$, then the above constant $C$ is zero and hence 
\begin{equation}\label{eqn:lemma3.3_schwarz}
\Spec(H) = \Spec(H').
\end{equation}

The PFH spectral invariants will take values in a more general set,
which we call the {\bf higher order action spectrum}
To define it,
let $H, G$ be two Hamiltonians.  The {\bf composition} of $H$ and $G$
is the Hamiltonian $H\#G(t,x) : = H(t,x) + G(t, (\phi^t_H)^{-1}(x))$.
It is known that $\phi^t_{H\#G} = \phi^t_H \circ \phi^t_G$; see
\cite[Sec. 5.1, Prop. 1]{hofer-zehnder}, for example. 
Denote by $H^k$ the $k$-times composition of $H$ with itself.  For any $d>0$, we now define the {\bf order $d$ spectrum} of $H$ by $$\displaystyle \Spec_d(H) := \cup_{k_1+ \ldots + k_j = d} \;\; \Spec(H^{k_1}) + \ldots + \Spec(H^{k_j}).$$  Note that $\displaystyle \Spec_d(H)$ may equivalently be described as follows: For every value $a \in   \Spec_d(H)$ there exist  capped periodic orbits $(z_1, u_1), \ldots, (z_k, u_k)$ of $H$ the sum of whose periods is $d$ and such that
\[
a = \sum \mathcal{A}_{H^{k_i}}(u_i, z_i).
\]
 
 We can use the above to define the action spectrum for compactly supported disc maps.  Recall from Section \ref{sec:sphere} our convention to identify the northern hemisphere of $\mathbb{S}^2$ with the disc; we will use this  to define the action spectrum in the case of the disc.

More precisely, if $H, H'$ are supported in the northern hemisphere $ S^+ \subset \S^2$,  and generate the same time-1 map $\phi$, then we in fact have $\Spec_d(H) = \Spec_d(H')$ for all $d>0$.  
Indeed, as an immediate consequence of Equation \eqref{eqn:lemma3.3_schwarz} we have $\Spec(H^k) = \Spec(H'^k)$ for all $k \in \N$ and so it follows from the definition that $\Spec_d(H) = \Spec_d(H')$ for all $d>0$.   Hence, if $\phi \in \DiffS$, then we can
define the action spectra of $\phi$ without any ambiguity by setting
\begin{equation} \label{eqn:spec_time1}
\Spec_d(\phi) = \Spec_d(H),
\end{equation}
 where $H$ is any Hamiltonian in $C^{\infty}_c(\S^1 \times  S^+)$ such that $\phi = \varphi^1_H$.

\subsection{Finite energy homeomorphisms}
  Recall that we defined finite-energy homeomorphisms in Definition \ref{def:finite_energy_homeo}.  It is not hard to show that $\FHomeo_c(\D, \omega)$ is a normal subgroup of $\Homeo_c(\D, \omega)$ and $\FHomeo_c(\D, \omega)$ contains the commutator subgroup of  $\Homeo_c(\D, \omega)$.  In this section, we show that $\FHomeo_c(\D, \omega)$ is a  normal subgroup.  
  
  \begin{prop}\label{prop:FHomeo_subgp}
  $\FHomeo_c(\D, \omega)$ is a normal subgroup of $\Homeo_c(\D, \omega)$. 
  \end{prop}
  \begin{proof}
  Consider smooth Hamiltonians $H, G \in C^{\infty}_c(\S^1 \times
  \D)$.  As was partly mentioned in Section \ref{sec:action_spectra},
  it is well-known (and proved for example in \cite[Sec. 5.1, Prop. 1]{hofer-zehnder})
  that the Hamiltonians 
  \begin{equation}
  \label{eqn:somehams}
  H\#G(t, x) := H(t,x) + G(t, (\varphi^t_H)^{-1}(x)), \quad \bar{H}(t,x) := - H(t, \varphi^t_H(x)),
  \end{equation}
 generate the Hamiltonian flows $ \varphi^t_H \phi^t_G$ and 
 $(\varphi^t_H)^{-1}$ respectively.   Furthermore, given $\psi \in \Diff_c(\D, \omega)$, the Hamiltonian
 \[ H\circ \psi(t,x):= H(t, \psi(x))\]
 generates the flow $\psi^{-1} \varphi^t_H \psi$.  
  
We now show that $\FHomeo_c$ is closed under conjugation.  Take $\phi
\in \FHomeo_c(\D, \omega)$ and let $H_i$ and $C$ be as in Definition
\ref{def:finite_energy_homeo}. %
Let $\psi \in \Homeo_c(\D, \omega)$ and take a sequence $\psi_i \in \Diff_c(\D, \omega)$ which converges uniformly to $\psi$.  Consider the Hamiltonians $H_i \circ \psi_i$.  The corresponding Hamiltonian diffeomorphisms are the conjugations $\psi_i^{-1} \varphi^1_{H_i} \psi_i$ which converge uniformly to $\psi^{-1} \phi \psi$.  Furthermore, \[\| H_i \circ \psi_i \|_{(1, \infty)} = \| H_i \|_{(1, \infty)}  \leq C,\] where the inequality follows from the definition of $\FHomeo_c(\D, \omega)$.

We will next check that $\FHomeo_c$  is a group.  Take $\phi, \psi \in \FHomeo_c$ and let $H_i, G_i  \in C^{\infty}_c(\S^1 \times \D)$ be two sequences of Hamiltonians such that  $\varphi^1_{H_i}, \varphi^1_{G_i}$ converge uniformly to $\phi, \psi$, respectively, and $ \| H_i \|_{(1, \infty)},  \| G_i \|_{(1, \infty)} \leq C $ for some constant $C$.  Then, the sequence $\varphi^{-1}_{H_i} \circ \varphi^1_{G_i}$ converges uniformly to $\phi^{-1} \circ \psi$.  Moreover,  by the above formulas, we have  $\varphi^{-1}_{H_i} \circ \varphi^1_{G_i} =  \varphi^1_{\overline{H}_i \#G_i}$.   
Since $ \| \overline H_i \# G_i \|_{(1, \infty)} \leq  \| H_i \|_{(1, \infty)} +  \| G_i \|_{(1, \infty)} \leq 2C,$ this proves that $\phi^{-1} \circ \psi \in \FHomeo_c$ which completes the proof that $\FHomeo_c$ is a group. 
\end{proof}
  
\subsection{Equivalence of perfectness and simplicity}
The goal of this section is to show that in the case of $\Homeo_c(\D, \omega)$ perfectness and simplicity are equivalent. 
\begin{prop}\label{prop:commutators}
Any non-trivial normal subgroup $H$ of  $\Homeo_c(\D, \omega)$ contains the  commutator subgroup of $\Homeo_c(\D, \omega)$.   Hence, $\Homeo_c(\D, \omega)$ is perfect if and only if it is simple.
\end{prop}

Before proving  the above proposition, we state two of its corollaries.

\begin{corol}\label{corol:fhomeo_commutators}The commutator subgroup of $\Homeo_c(\D, \omega)$ is contained in $\FHomeo_c(\D, \omega)$.
\end{corol}
\begin{remark}\label{rem:inf_twist_not_commutator}   As a consequence, proving that the infinite twist $\phi_f$, introduced in Section \ref{sec:calabi_inf_twist}, is not in  $\FHomeo_c(\D, \omega)$ proves also that $\phi_f$ cannot be written as a product of commutators.
\end{remark}

As promised in the introduction, we prove in the next corollary that $\Homeo_c(\D, \omega)$ admits no non-trivial continuous homomorphisms.  This fact seems to be well-known to the experts, however, we do not know of a published reference for it.

\begin{corol}\label{corol:no_continuous_homomorphism}
 The group $\Homeo_c(\D, \omega)$ admits no non-trivial homomorphism which is continuous with respect to the $C^0$ topology.
\end{corol}
\begin{proof}
Let $H$ be a non-trivial normal subgroup of  $\Homeo_c(\D, \omega)$.   We will show that $H$ is dense with respect to the $C^0$ topology; this proves the corollary because the kernel of a continuous homomorphism is closed.

By Proposition \ref{prop:commutators}, we know that $H$ contains the  commutator subgroup of $\Diff_c(\D, \omega)$.  Consequently, $H$ contains the kernel of the Calabi homomorphism as the commutator subgroup of $\Diff_c(\D, \omega)$  coincides with the kernel of the Calabi invariant \cite{Banyaga}.  

We claim that the kernel of the Calabi invariant is dense in $\Diff_c(\D, \omega)$; hence, it is dense in $\Homeo_c(\D, \omega)$.  Indeed, take any $\psi \in \Diff_c(\D, \omega)$ and denote $a= \Cal(\psi)$.  Pick Hamiltonians $H_n$ such that 
\begin{itemize}
\item $H_n$ is supported in a disc of diameter $\frac{1}{n}$,
\item $\int_\D H_n = -a$.  Thus, $\Cal(\varphi^1_{H_n}) = -a.$
\end{itemize}  
Then, $\Cal(\varphi^1_{H_n} \circ \psi) = 0$ and $\varphi^1_{H_n} \circ \psi \xrightarrow{C^0} \psi$.
\end{proof}

The proof of  Proposition \ref{prop:commutators} relies  on  a general argument, due to Epstein \cite{Epstein} and Higman \cite{Higman}, which essentially shows that perfectness implies simplicity for transformation groups satisfying certain assumptions. We will present a version of this argument, which we learned in \cite{fathi}, in our context. 
\begin{proof}[Proof of Proposition \ref{prop:commutators}]
 Pick $h \in H$ such that $ h \neq \mathrm{Id}$.  We can find a closed  {\it topological disc}, that is a set which is homeomorphic to a standard Euclidean disc $\D' \subset \D$ such that $h(\D') \cap \D' = \emptyset$.
Denote by $\Homeo_c(\D', \omega)$ the subset of $ \Homeo_c(\D, \omega)$ consisting of area-preserving homeomorphisms whose supports are contained in the interior of $\D'$.  We will first prove the following lemma.

\begin{lemma}
\label{lem:comh}
The commutator subgroup of $\Homeo_c(\D', \omega)$ is contained in $H$. 
\end{lemma}
\begin{proof}
We must  show that for any $f, g \in \Homeo_c(\D', \omega)$, the commutator $[ f,g]  := fgf^{-1} g^{-1}$ is an element of $H$.

First, observe that for any $f \in \Homeo_c(\D', \omega)$ we have 
\begin{equation}\label{eq:commutator1}
[f, r] \in H
\end{equation}
for any $r \in H$.   Indeed, by normality, $ frf^{-1}  \in H$ and hence   $ frf^{-1} r^{-1} \in  H$.   Next, we claim that for any $f,g \in \Homeo_c(\D', \omega)$
 \begin{equation}\label{eq:commutator2}
[ f,g] [g, hfh^{-1}] = f [g, [f^{-1}, h] \, ] f^{-1}.
\end{equation}
Postponing the  proof of this identity for the moment, we will first show that it implies the lemma.    Note that $g$ and $hfh^{-1}$  are, respectively, supported in $\D'$ and $h(\D')$ which are disjoint.  Thus, $[g, hfh^{-1}] = \id$.  Hence, Identity \eqref{eq:commutator2} yields  
$[ f,g] = f [g, [f^{-1}, h] \, ] f^{-1}$.  Now, \eqref{eq:commutator1} implies that $ [g, [f^{-1}, h] \, ]  \in H$ which, by normality of $H$, implies that $f [g, [f^{-1}, h] \, ] f^{-1} \in H$.  This gives us the conclusion of the lemma. We complete the proof by proving  Identity \eqref{eq:commutator2}:
\begin{align*}
    & [ f,g] [g, hfh^{-1}] = ( fgf^{-1}g^{-1}) \,  (ghfh^{-1}g^{-1}hf^{-1}h^{-1}) \\
 = & fg ( f^{-1}hfh^{-1}) g^{-1}hf^{-1}h^{-1} = fg [f^{-1},h] g^{-1}hf^{-1}h^{-1} \\
 =  & fg [f^{-1},h] g^{-1}hf^{-1}h^{-1} f f^{-1} =  fg [f^{-1},h] g^{-1} [h, f^{-1}]  f^{-1} \\
 =  & f[g, [f^{-1}, h] \, ] f^{-1}.
\end{align*} \end{proof}

We continue with the proof of Proposition \ref{prop:commutators}.  Fix
a small $\varepsilon >0$ and let $\mathcal{E}$ 
 be the set consisting of all  $g \in \Homeo_c(\D, \omega)$ whose supports are contained in some topological disc of area $\varepsilon$.  It is a well-known fact that the set $\mathcal{E}$ generates the group $\Homeo_c(\D , \omega)$.  This is usually referred to as the \emph{fragmentation property} and it was proven by Fathi; see  Theorems 6.6, A.6.2, and A.6.5 in \cite{fathi}.

We claim that $[f, g] \in H$  for any $f,g \in \mathcal{E}$. Indeed,  assuming $\varepsilon$ is small enough, we can find a topological disc $U$ which contains the supports of $f$ and $g$ and whose area is less than the area of $\D'$.  There exists $r \in \Homeo_c(\D , \omega)$ such that $r(U) \subset \D'$.  As a consequence, $rfr^{-1}, r g r^{-1}$ are both supported in $\D'$ and hence, by Lemma~\ref{lem:comh}, $[rfr^{-1}, r g r^{-1}] \in H$.
Since $H$ is a normal subgroup of $ \Homeo_c(\D , \omega)$, and  $[rfr^{-1}, r g r^{-1}] = r [f,g] r^{-1}$, we conclude that $[f, g] \in H$.  

Now, the set $\mathcal{E}$ generates  $\Homeo_c(\D , \omega)$ and $[f,g] \in H$  for any $f,g \in \mathcal{E}$.   Hence,  the quotient group $\Homeo_c(\D , \omega)/ H$ is abelian. Thus,  $H$ contains the commutator subgroup of $\Homeo_c(\D , \omega)$.
\end{proof}

\section{Periodic Floer Homology and basic properties of the PFH spectral invariants}
\label{sec:pfhs2}

In this section, we recall the definition of Periodic Floer Homology (PFH), due to Hutchings \cite{Hutchings-index, Hutchings-Sullivan-Dehntwist}, and the construction of the spectral invariants which arise from this theory, also due to Hutchings \cite{Hutchings_unpublished}.  We will then prove that PFH spectral invariants satisfy the Monotonicity, Hofer Continuity, and Spectrality properties which we mentioned in the introduction.  The spectral invariants appearing in Section \ref{sec:PFH_spectralinvariants_intro}  are defined by identifying  area-preserving maps of the disc, $\Diff_c(\D, \omega)$, with area-preserving maps of the sphere, which are supported in the northern hemisphere $S^+$, and using the PFH of $\S^2$.   Thus, the three aforementioned properties will follow from  related properties about PFH spectral invariants on $\S^2$; see Theorem \ref{thm:PFHspec_initial_properties} below.

\subsection{Preliminaries on $J$--holomorphic curves and stable Hamiltonian structures}\label{sec:prelim_Jcurves}

A stable Hamiltonian structure (SHS) on a closed three-manifold $Y$ is a pair $(\alpha, \Omega)$, consisting of a $1$--form $\alpha$ and  a closed  two-form $\Omega$, such that 
\begin{enumerate}
\item $\alpha \wedge \Omega $ is a volume form on $Y$, 
\item  $\ker(\Omega) \subset \ker(d\alpha)$.
\end{enumerate}

Observe that the first condition implies that $\Omega$ is non-vanishing, and as a consequence, the second  condition is equivalent to $d\alpha = g\Omega$, where $g:Y \rightarrow \R$ is a smooth function. 

  A stable Hamiltonian structure determines a plane field $\xi : = \ker(\alpha)$ and a {\bf Reeb} vector field $R$ on $Y$ given by 
  $$R \in \ker(\Omega), \; \alpha(R) =1.$$
 Closed integrals curves of $R$ are called {\bf Reeb orbits}; we regard Reeb orbits as equivalent if they are equivalent as currents.  %
  
  Stable Hamiltonian structures were introduced in \cite{BEHWZ, CiMo} as a setting in which one can obtain general Gromov-type compactness results, such as the SFT compactness theorem, for pseudo-holomorphic curves in $\R \times Y$.  Here are two examples of stable Hamiltonian structures which are relevant to our story.
  
  \begin{example}\label{ex:contact_structure}
  A {\bf contact} form on $Y$ is a $1$-form $\lambda$ such that $\lambda \wedge d\lambda$ is a volume form. The pair $(\alpha, \Omega) := (\lambda,  d \lambda)$ gives a stable Hamiltonian structure  with $g \equiv 1$.  The plane field $\xi$ is the associated contact structure and the Reeb vector field as defined above gives the usual Reeb vector field of a contact form. 
  
  The {\bf contact symplectization} of $Y$ is
\[ X := \mathbb{R} \times Y_\varphi , \]
which has a standard symplectic form, defined by
\begin{equation}
\label{eqn:symp_form_symplectization}
\Gamma = d(e^s \lambda),
\end{equation}
where $s$ denotes the coordinate on $\R$.
  \end{example}
  
   \begin{example}\label{ex:mapping_torus}
 Let $(S,\omega_S)$ be a closed surface and denote by $\varphi$ a smooth area-preserving diffeomorphism of $S$. Define the mapping torus 
 $$Y_\varphi := \frac{S \times [0,1]}{(x,1) \sim (\varphi(x), 0) }.$$
 Let $r$ be the coordinate on $[0,1]$.  Now, $Y_\varphi$ carries a stable Hamiltonian structure  $(\alpha, \Omega) := (dr, \omega_\varphi)$, where $\omega_\varphi$ is the canonical closed two form on $Y_\varphi$ induced by $\omega_S$.  Note that the plane field $\xi$ is given by the vertical tangent space of the fibration $\pi: Y_\varphi \rightarrow \S^1$ and the Reeb vector field is given by $R =\partial_r$.  Here, $g \equiv 0$.   Observe that the Reeb orbits here	 are in correspondence with the periodic orbits 
 of $\varphi$.
 
 We define the {\bf symplectization} of $Y_\varphi$ 
\[ X := \mathbb{R} \times Y_\varphi ,\]
which has a standard symplectic form, defined by
\begin{equation}
\label{eqn:standard2}
\Gamma = ds \wedge dr + \omega_{\varphi},
\end{equation}
where $s$ denotes the coordinate on $\R$.
\end{example}

We say an almost complex structure $J$ on $X = \R \times Y$ is {\bf admissible}, for a given SHS $(\alpha, \Omega)$, if the following conditions are satisfied:
\begin{enumerate}
\item $J$ is invariant under translation  in the $\R$-direction of $\R\times Y$,
\item $J \partial_s =  R$, where $s$ denotes the coordinate on the $\R$-factor of $\R \times Y$,  
\item  $ J \xi = \xi$, where $\xi := \ker(\alpha)$, and $\Omega(v, Jv) > 0$ for all nonzero $v \in \xi$.
\end{enumerate} 
We will denote by $\mathcal{J}(\alpha, \Omega)$ the set of almost complex structures which are admissible for $(\alpha, \Omega)$.   The space $\mathcal{J}(\alpha, \Omega)$ equipped with the $C^{\infty}$ topology is path connected, and even contractible. 

  Define a {\bf $J$-holomorphic map} to be a smooth map
\[ u: (\Sigma,j)  \to (X,J),\]
satisfying the equation
\begin{equation}
\label{eqn:jcurveequation}
du \circ j = J \circ du,
\end{equation}
where $(\Sigma,j)$ is a closed Riemann surface (possibly disconnected), minus a finite number of punctures.   As is common in the literature on ECH,  we will sometimes have to consider $J$-holomorphic maps up to equivalence of currents, and we call such an equivalence class a {\bf $J$-holomorphic current}; see  \cite{Hutchings-Notes} for the precise definition of this equivalence relation.   An equivalence class of $J$-holomorphic maps under the relation of biholomorphisms of the domain will be called a {\bf $J$-holomorphic curve};  this relation might be more familiar to the reader, but is not sufficient for our needs.

For future reference, we will call a $J$-holomorphic curve or current {\bf irreducible} when its domain is connected.  A $J$-holomorphic map $u:  (\Sigma,j)  \to (X,J)$ is called {\bf somewhere injective} if there exists a point $ z \in \Sigma$ such that $u^{-1}(u(z))$ = \{z\} and $du: T_z\Sigma \rightarrow T_{u(z)} X$ is injective.   We say that $u$ is a {\bf multiple cover} if there exists a branched cover $\phi : (\Sigma, j) \rightarrow  (\Sigma', j')$ of degree greater than $1$ and a $J$--holomorphic map $u': (\Sigma',j')  \to (X,J)$ such that $u = u' \circ \phi$. Every nonconstant irreducible $J$--holomorphic curve is either somewhere injective or multiply covered; and, two somewhere injective maps are equivalent as curves if and only if they are equivalent as currents.  %

In the lemma below we state a standard property of $J$-holomorphic curves which plays a key role in our arguments.  %
For a proof see the argument in \cite[Lem. 9.9]{Wendl-Notes}, for example.

\begin{lemma}
\label{lem:pointwisenonnegative} Suppose $J \in \mathcal{J}(\alpha, \Omega)$ where $(\alpha, \Omega)$ is a stable Hamiltonian structure on $Y$.  If  $C$ is a $J$-holomorphic curve in $\R \times Y$, then $\Omega$ is pointwise nonnegative on $C$.  Furthermore, $\Omega$ vanishes at a point on $C$ only if $C$ is tangent to the span of $\partial_s$ and $R$.  %
\end{lemma}

\subsubsection{Weakly admissible almost complex structures on mapping cylinders} \label{sec:admissible_vs_weakly}
In this article, we will be almost exclusively considering $J$--holomorphic curves and currents in the symplectization $X= \R \times Y_\varphi$, introduced above in Example \ref{ex:mapping_torus}.  Now, the usual SHS on $\R \times Y_\varphi$ is $(dr, \omega_\varphi)$, defined above, and hence, we will be mostly considering almost complex structures $J$ which are admissible for this SHS, i.e.\  $J \in \mathcal{J}(dr, \omega_\varphi)$.  However, in Section \ref{sec:PFH_monotone_twist} we will need the added flexibility of working with almost complex structures on $\R \times Y_\varphi$ which are admissible for some SHS of the form $(\alpha, \Omega)$ with $\Omega=\omega_\varphi$ and $(\alpha,\Omega)$ inducing the same orientation as $(dr,\omega_{\varphi})$ 
we will refer to such almost complex structures as {\bf weakly admissible}.  Clearly, an admissible almost complex structure is weakly admissible.   %

We will now list several observations about weakly admissible almost complex structures which will be helpful in Section \ref{sec:PFH_monotone_twist}.  Let $(\alpha_0, \omega_\varphi)$ and $(\alpha_1, \omega_\varphi)$ be SHSs inducing the same orientation and denote their Reeb vector fields by $R_0, R_1$.

\begin{enumerate}
\item There exists a positive function $\eta : Y_\varphi \rightarrow \R$ such that $R_0 = \eta R_1$.  This is because  $R_0, R_1 \in \ker(\omega_\varphi)$ and $(\alpha_0, \omega_\varphi)$ and $(\alpha_1, \omega_\varphi)$ induce the same orientation.   In particular, $R_0, R_1$ have the same Reeb orbits.

\item Define $\alpha_t = (1-t)\alpha_0 + t \alpha_1$.  Then, $(\alpha_t, \omega_\varphi)$ is a SHS for all $t\in [0,1]$.  In other words, the space of all SHS of the above form is convex, hence contractible.   
\item As a consequence of the previous observation, 
we see that the space of weakly admissible almost complex structures is path connected. Indeed, it is even contractible as it forms a fibration, over the space of SHS with $\Omega =\omega_\varphi$, whose base and fibres are contractible. 
\item Lastly, the conclusions of Lemma \ref{lem:pointwisenonnegative} hold for $J$-holomorphic curves, for weakly admissible $J$. 
\end{enumerate}

\subsection{Definition of periodic Floer homology}
\label{sec:pfhdefn}
Periodic Floer homology (PFH) is a version of Floer homology, defined by Hutchings \cite{Hutchings-index,   Hutchings-Sullivan-Dehntwist}, for area-preserving maps of surfaces. The cons\-truction  of PFH is closely related to the better-known embedded contact homology (ECH) and, in fact, predates the construction of ECH.  We now review the definition of PFH; for further details on the subject we refer the reader to \cite{Hutchings-index, Hutchings-Sullivan-Dehntwist}.

Let $(S,\omega_S)$ be a closed\footnote{PFH can still be defined if $S$ is not closed, but we will not need this here.} surface with an area form, and $\varphi$ a {\bf nondegenerate} smooth area-preserving diffeomorphism.  Non-degeneracy is defined as follows: A periodic point $p$ of $\phi$, with period $k$, is said to be non-degenerate if the derivative of $\varphi^k$ at the point $p$ does not have $1$ as an eigenvalue.  We say $\varphi$ is {\bf $d$--nondegenerate} 
if all of its periodic points  of period at most $d$ are  nondegenerate; if $\varphi$ is $d$--nondegenerate for all $d$, then we say it is non-degenerate.  A $C^\infty$-generic area-preserving diffeomorphism is nondegenerate. 

Recall the definition of $Y_\varphi$ from Example \ref{ex:mapping_torus} and take $ 0 \neq h \in H_1(Y_\varphi)$.  If $\varphi$ is nondegenerate and satisfies a certain ``monotonicity" assumption,\footnote{If the monotonicity assumption does not hold, we can still define PFH, but we need a different choice of coefficients; this is beyond the scope of the present work.  } which we do not need to discuss here as it automatically holds when $S= \S^2$, the {\bf periodic Floer homology} %
$PFH(\varphi,h)$ is defined; it is the homology of a chain complex $PFC(\varphi, h)$ which we define below.  

\begin{remark}\label{rem:ECH}
If we carry out the construction outlined below, nearly verbatim, for a contact SHS  $(\lambda, d\lambda)$, rather than the SHS $(dr, \omega_\varphi)$, then we would obtain the {\bf embedded contact homology} ECH; see \cite{Hutchings-Notes} for further details. 
\end{remark}

\subsubsection{PFH generators}
The chain complex $PFC$ is freely generated over\footnote{We could also define PFH over $\mathbb{Z}$, but we do not need this here.} $\mathbb{Z}_2$, by certain finite orbit sets $\alpha = \lbrace (\alpha_i, m_i) \rbrace$ called {\bf PFH generators}.  Specifically, we require that each $\alpha_i$ is an embedded Reeb orbit,  
the $\alpha_i$ are distinct, the $m_i$ are positive integers, $m_i = 1$ whenever $\alpha_i$ is \emph{hyperbolic},\footnote{Being hyperbolic means that the eigenvalues at the corresponding periodic point of $\varphi$ are real.   Otherwise, the orbit is called {\it elliptic}. } and $\sum m_i [\alpha_i] = h$.

\subsubsection{The ECH index}
  The $\mathbb{Z}_2$ vector space $PFC(\varphi, h)$ has a relative $\mathbb{Z}$ grading which we now explain.  Let $\alpha =\{(\alpha_i, m_i) \rbrace, \beta= \lbrace (\beta_j, n_j) \rbrace$ be two $PFH$ generators in $PFC(\varphi, h)$. Define $H_2(Y_\varphi, \alpha, \beta)$  to be  the set of equivalence classes of $2$--chains $Z$ in $Y_\varphi$ satisfying $\partial Z = \sum m_i \alpha_i - \sum n_i \beta_i$.  Note that $H_2(Y_\varphi, \alpha, \beta)$ is an affine space over $H_2(Y_\varphi)$.
  
  We define the {\bf ECH index}  
\begin{equation}
\label{eqn:ECH_index}
I(\alpha,\beta, Z) = c_{\tau} (Z) + Q_{\tau}(Z) + \sum_i \sum^{m_i}_{k  = 1} CZ_{\tau}(\alpha^k_i) -  \sum_j \sum^{n_j}_{k  = 1} CZ_{\tau}(\beta^k_j),
\end{equation}
where $\tau$ is (the homotopy class of a trivialization) of the plane field $\xi$ over all Reeb orbits, $c_{\tau}(Z)$ denotes the relative first Chern class of $\xi$ over $Z$,   $Q_{\tau}(Z)$ denotes the relative intersection pairing, and $CZ_{\tau}(\gamma^k)$ denotes the Conley-Zehnder index of the $k^{th}$ iterate of $\gamma$; all of these quantities are computed using the trivialization $\tau$.  We will review the definitions of $c_{\tau}$, $CZ_\tau$,  and $Q_{\tau}$ in Section \ref{sec:index}.

  It is proven in \cite{Hutchings-index} that  although the individual terms in the above definition do depend on the choice of $\tau$,  the ECH index itself does not depend on $\tau$.  According to \cite[Prop 1.6]{Hutchings-index}, the change in index caused by changing the relative homology class $Z$ to another  $Z' \in H_2(Y_\varphi, \alpha, \beta)$  is given by the formula 
  
 \begin{equation}\label{eqn:index_ambiguity}
 I(\alpha, \beta, Z) - I(\alpha, \beta, Z') = \langle c_1(\xi) + 2 PD(h), Z-Z' \rangle.
 \end{equation}
 
 \subsubsection{The differential}
Let  $J \in \mathcal{J}(dr, \omega_\varphi)$ be an almost complex structure on $X = \R \times Y_\varphi$ which is admissible for the SHS $(dr, \omega_\varphi)$ and define 
\[ \mathcal{M}^{I = 1}_{J}(\alpha,\beta)\]
to be  the space of $J$-holomorphic currents $C$ in X, modulo translation in the $\mathbb{R}$ direction, with ECH index $I(\alpha, \beta, [C]) =1$, which are
asymptotic to $\alpha$ as $s \to + \infty$ and $\beta$ as $s \to - \infty$; we refer the reader to \cite{Hutchings-Sullivan-Dehntwist}, page 307, for the precise definition of asymptotic in this context.

Assume now and below for simplicity that $S = \S^2$.  (For other surfaces, a similar story holds, but we will not need this.)     
Then, for generic $J$, $\mathcal{M}^{I = 1}_{J}(\alpha,\beta)$ is a compact $0$-dimensional manifold and we can define the PFH differential by the rule 
\begin{equation}
\label{eqn:diff}
\langle \partial \alpha, \beta \rangle = \# \mathcal{M}^{I=1}_J(\alpha,\beta),
\end{equation}
where $\#$ denotes mod $2$ cardinality.  It is shown in \cite{Hutchings-TaubesI,Hutchings-TaubesII} that\footnote{More precisely, \cite{Hutchings-TaubesII} 
proves that the differential in embedded contact homology squares to zero.  As pointed out in \cite{Hutchings-TaubesI} and \cite{Lee-Taubes} this proof carries over, nearly verbatim, to our setting.} $\partial^2 = 0$, hence the homology $PFH$ is defined.

Lee and Taubes \cite{Lee-Taubes} proved that the homology does not depend on the choice of $J$; in fact, \cite[Corollary 1.1]{Lee-Taubes} states that for any surface, the homology depends only  on the Hamiltonian isotopy class of $\varphi$ and the choice of $h \in H_1(Y_\varphi)$.   In our case, where $S = \S^2$, all orientation preserving
area-preserving diffeomorphisms are Hamiltonian isotopic, and so we obtain a well-defined invariant which we denote by $PFH(Y_{\varphi},h).$  For future motivation, we note that the Lee-Taubes invariance results discussed here come from an isomorphism of PFH and a version of the Seiberg-Witten Floer theory from \cite{Kronheimer-Mrowka}.

Importantly, for the applications to this paper, we can relax the assumption that $\varphi$ is nondegenerate to requiring only that $\varphi$ is $d$-nondegenerate, where $d$, called the {\bf degree}, is the positive integer determined by the intersection of $h$ with the fiber class of the map $\pi: Y_\varphi \rightarrow \S^1$ ; note that any orbit set $\alpha$ with $[\alpha] = h$ must correspond to periodic points with period no more than $d$. 

\subsubsection{The structure of PFH curves}
\label{sec:structure}

We now explain part of the motivation for PFH, and for the ECH index $I$; we will use some of the results below later in the paper as well.  For additional details on the account here, we refer the reader to \cite{Hutchings-Notes}, for example.

Let $C$ be a $J$--holomorphic curve which is asymptotic to orbit sets  $\alpha, \beta$  and denote $I([C]):= I(\alpha, \beta, [C])$.  A key fact which powers the definition of PFH is the {\bf index inequality}
\begin{equation}
\label{eqn:indexinequality}
\text{ind}(C) \leq I([C]) - 2 \delta(C), 
\end{equation}
valid for any somewhere injective curve $C$, where $\delta(C) \geq 0$ is a count of singularities of $C$; see \cite{Hutchings-Notes} for the precise definition of $\delta(C)$; 
here $\text{ind}(C)$ refers to the {\bf Fredholm index}
\begin{equation}
\label{eqn:indexformula_Fredholm}
\text{ind}(C) = - \chi(C) + 2 c_{\tau}(C) + CZ^{ind}_{\tau}(C),
\end{equation}   
where $\chi(C)$ denotes the Euler characteristic of $C$, and $CZ_\tau^{ind}$ is another combination of Conley-Zehnder terms that we will review in Section \ref{sec:index}.  The Fredholm index is the formal dimension of the moduli space of curves
near $C$, and \eqref{eqn:indexinequality} therefore says that the ECH index bounds this dimension from above.  In the case relevant here, when $I([C])= 1$, we therefore get an important structure theorem for ECH index $1$ curves.  To simplify the exposition, we continue to assume here and below that $S = \mathbb{S}^2.$

\begin{prop} [Lem.\ 9.5, \cite{Hutchings-index}, Cor.\ 2.2 \cite{Hutchings-Sullivan-Dehntwist}]
\label{prop:structure_Jcurves}
Assume that $J$ is admissible and generic and that $\varphi$ is $d$-nondegenerate\footnote{In \cite{Hutchings-index}, one also wants to assume a ``local linearity" condition around periodic points; however, this condition is not necessary --- as was stated in \cite{Hutchings-index}, this condition was added to simplify the analysis, and it can be dropped, as in \cite{Lee-Taubes}, using work of Siefring \cite{Siefring-relative,Siefring-intersection}.}.  Let $C$ be a $J$-holomorphic current in $X$, asymptotic to orbit sets of degree at most $d$. 
Then $I([C]) \geq 0$, and if $I([C]) = 1$, then $C$ has exactly one embedded component $C'$ with $I(C') = 1$, and all other components, if they exist, are multiple covers of $\mathbb{R}$-invariant cylinders that do not intersect $C'$.      
\end{prop}

By an $\R$--invariant cylinder, we mean a $J$--holomorphic curve, or current, $C$ given by a map $u : \R \times \S^1 \rightarrow X$ of the form $u(s,t) =  (s,  \alpha(t))$ where $\alpha(t)$ is an embedded Reeb orbit; these are also referred to as {\bf trivial cylinders}.

\medskip

We end this section by mentioning that for any SHS $(\alpha, \Omega)$ on $Y_\varphi$, the formula \eqref{eqn:indexformula_Fredholm} still gives the formal dimension of the moduli space of curves near $C$ when
$J \in \mathcal{J}(\alpha, \Omega);$ 
see \cite[Sections 7 and 8]{Wendl-Notes}; later, we will make use of this.

\subsection{Twisted PFH and the action filtration}
\label{sec:specdefn}
One can define a twisted version of PFH, where we  keep track of the relative homology classes of $J$--holomorphic 
currents, that we will need to define spectral invariants. It has the same invariance properties of ordinary PFH, e.g. by \cite[Corollary 6.7]{Lee-Taubes}, where it is shown to agree with an appropriate version of Seiberg-Witten Floer cohomology.
For the benefit of the reader, we provide a brief explanation of how the twisted version works in relationship with Seiberg-Witten theory in Remark~\ref{rmk:twisted} below.  We refer the reader to \cite[Section 11.2.1]{Hutchings-Sullivan-T3}, \cite[Section 1.a.2]{Taubes-ECH=SW-V}.

The main reason we want to use twisted PFH is because 
while PFH does not have a natural action filtration, the twisted PFH does.   As above, we are continuing to assume $S = \S^2$.

First, note that in this case, $Y_\varphi$ is diffeomorphic to $ \S^2 \times \S^1$ and so $H_1(Y_{\varphi}) = \mathbb{Z}$.  A class $h \in H_1(Y_{\varphi})$ is then determined by its intersection with the homology class of a fiber of the map $\pi: Y_{\varphi} \to \S^1$, which we  defined above to be the {\bf degree} and denote by the integer $d$; from now on we will write the integer $d$ in place of $h$, since these two quantities determine each other.

Choose a reference cycle  $\gamma_0$ in $Y_{\varphi}$ such that $\pi|_{\gamma_0}: \gamma_0 \to \S^1$ is an orientation preserving diffeomorphism and fix a trivialization $\tau_0$ of $\xi$ over $\gamma_0$.   We can now define the $\widetilde{PFH}$ chain complex $\widetilde{PFC}(\varphi,d)$.  A generator of $\widetilde{PFC}(\varphi,d)$ is a pair $(\alpha,Z)$, where $\alpha$ is a PFH generator of degree $d$, and $Z$ is a relative homology class in $H_2(Y_{\varphi},\alpha, d \gamma_0)$.  The $\mathbb{Z}_2$ vector space $\widetilde{PFC}(\varphi,d)$ has a canonical $\mathbb{Z}$-grading $I$ given by
\begin{equation}
\label{eqn:twistedgrading}
I(\alpha,Z) = c_{\tau} (Z) + Q_{\tau}(Z) + \sum_i \sum^{m_i}_{k  = 1} CZ_{\tau}(\alpha^k_i).
\end{equation}
The terms in the above equation are defined as in the definition of the ECH index given by Equation \eqref{eqn:ECH_index}. Note that the above index depends on the choice of the reference cycle $\gamma_0$ and the trivialization $\tau_0$ of $\xi$ over $\gamma_0$.

The index defined here is closely related to the ECH index of Equation \eqref{eqn:ECH_index}:  Let $(\alpha, Z)$ and $(\beta, Z')$ be two generators of $\widetilde{PFC}(\varphi,d)$.  Note that $ Z - Z'$ is a relative homology class in $H_2(Y_\varphi, \alpha, \beta)$.  Then, it follows from\cite[Prop 1.6]{Hutchings-index} that 
$$I(\alpha, \beta, Z-Z') = I(\alpha, Z) - I(\beta, Z').$$

As a consequence, we see that the index difference  $ I(\alpha, Z) - I(\beta, Z')$ does not depend on the choices involved in the definition of the index.

We now define the differential on $\widetilde{PFC}(\varphi,d)$.  We say that $C$ is a $J$--holomorphic curve, or current,  from $(\alpha, Z)$ to $(\beta, Z')$ if it is asymptotic to $\alpha$ as $s \to + \infty$ and $\beta$ as $s \to -\infty$, and moreover satisfies 
\begin{equation}
\label{eqn:importantequation10}
Z' + [C] = Z,
\end{equation}
as elements of $H_2(Y_\varphi, \alpha, d\gamma_0).$

Suppose that $I(\alpha, Z) - I(\beta, Z') =1$ and let $J \in \mathcal{J}(dr, \omega_\varphi)$.  We define \[ \mathcal{M}_J( (\alpha,Z), (\beta,Z') )\] 
to be the moduli space of $J$-holomorphic currents in $X = \R \times Y_\varphi$, modulo  translation in the $\R$ direction,  from $(\alpha,Z)$ to $(\beta,Z')$.  As before, for generic  
$J \in \mathcal{J}(dr, \omega_\varphi)$, the above moduli space is a compact $0$--dimensional manifold and we define the differential by the rule
$$\langle \partial (\alpha, Z), (\beta, Z') \rangle = \# \mathcal{M}_J( (\alpha,Z), (\beta,Z') ),$$
where $\#$ denotes mod $2$ cardinality.  As before,  $\partial^2 = 0$ by \cite{Hutchings-TaubesI, Hutchings-TaubesII}, and so the homology $\widetilde{PFH}$ is well-defined; by \cite{Lee-Taubes} it depends only on the degree $d$; we write more about this in Remark~\ref{rmk:twisted} below.  We will write it as $\widetilde{PFH}(Y_{\varphi},d).$ 

By a direct computation in the case where $\varphi$ is an irrational rotation of the sphere, i.e. $\varphi(z, \theta) = (z, \theta + \alpha)$ with $\alpha$ being irrational, we obtain 
\begin{equation}\label{eq:PFH_sphere}
\widetilde{PFH}_*(Y_\varphi,d) = 
\begin{cases}
  \mathbb{Z}_2, & \text{if } *=d \text{ mod } 2, \\
  0 & \text{otherwise.}
\end{cases}
\end{equation}
Here is a brief outline of the computation leading to the above identity.  The Reeb vector field in $Y_\varphi$ has two simple Reeb orbits  $\gamma_+, \gamma_-$ corresponding to the north and the south poles.  Both of these orbits are elliptic and so the orbit sets of $\widetilde{PFC}(\varphi, d)$ consist entirely of elliptic Reeb orbits.  This implies that the difference in index between any two generators of $\widetilde{PFC}(\varphi, d)$ chain complex is an even integer; see \cite[Proposition 1.6.d]{Hutchings-index}.   Thus, the PFH differential vanishes.  Now, the above identity follows from the fact that for each index $k$, satisfying  $k = d $ mod $2$, there exists a unique generator of index $k$ in $\widetilde{PFC}(\varphi, d)$; this fact may be deduced from the index computations  carried out in Section \ref{sec:index}.

\medskip

The vector space $\widetilde{PFC}(\varphi,d)$ carries a filtration, called the {\bf action filtration}\footnote{The relation between the quantity $\cal A(\alpha, Z)$ and the Hamiltonian action functional discussed in Section \ref{sec:action_spectra} will be clarified in Lemma \ref{lemma:action-action}.}, defined by 
\[ \mathcal{A}(\alpha,Z) = \int_Z \omega_{\varphi} .\]
We define  $\widetilde{PFC}^L(\varphi,d)$ to be the $\Z_2$ vector space spanned  by generators $(\alpha,Z)$ with $\mathcal{A}(\alpha, Z) \leq L$.   

By Lemma \ref{lem:pointwisenonnegative}, $\omega_{\varphi}$ is pointwise nonnegative along any $J$-holomorphic curve $C$, and so $\int_C \omega_{\varphi} \geq 0$.   This implies  that the differential does not increase the action filtration, i.e. $$\partial( \widetilde{PFC}^L(\varphi,d) ) \subset \widetilde{PFC}^L(\varphi,d). $$ 
Hence,  it makes sense to define $\widetilde{PFH}^L(\varphi,d)$ to be the homology of the subcomplex $\widetilde{PFC}^L(\varphi,d)$.  

We are now in position to define the PFH spectral invariants.  There is an inclusion induced map
\begin{equation}
\label{eqn:iind}
\widetilde{PFH}^L(\varphi,d) \to \widetilde{PFH}(Y_\varphi,d).
\end{equation}
If $0 \ne \sigma \in \widetilde{PFH}(\varphi,d)$ is any nonzero class, then we define the {\bf PFH spectral invariant}
\[ c_{\sigma}(\varphi)\] 
to be the infimum, over $L$, such that $\sigma$ is in the image of the inclusion induced map \eqref{eqn:iind} above.  The number $c_{\sigma}(\varphi)$ is finite, because $\varphi$ is non-degenerate and so there are only finitely many Reeb orbit sets of degree $d$, and hence only finitely many pairs $(\alpha,Z)$ of a fixed grading.   We remark that  $c_{\sigma}(\varphi)$ is given by the action of some $(\alpha,Z)$.   
Indeed, this can be deduced from the following two observations:   
\begin{enumerate}
\item  If $L < L'$ are such that there exists no $(\alpha,Z)$ with $L \leq \mathcal{A}(\alpha, Z) \leq L'$, then the two vector spaces   $\widetilde{PFC}^L(\varphi,d)$ and  $\widetilde{PFC}^{L'}(\varphi,d)$ coincide and so  $\widetilde{PFH}^L(\varphi,d) \to \widetilde{PFH}(Y_\varphi,d)$ and $\widetilde{PFH}^{L'}(\varphi,d) \to \widetilde{PFH}(Y_\varphi,d)$ have the same image.  

\item The set of action values $\{\mathcal{A}(\alpha, Z) : (\alpha, Z) \in \widetilde{PFC}^L(\varphi,d) \}$ forms a discrete subset of $\R$.  This is a consequence of the fact that, as stated above, 
there are only finitely many Reeb orbit sets of degree $d$.  
\end{enumerate}

In Remark~\ref{rmk:nojdependence} below we show that this does not depend on the choice of the admissible almost complex structure $J$.  Note, however, that   $c_{\sigma}(\varphi)$ does depend on the choice of the reference cycle $\gamma_0$.

\subsection{Initial properties of PFH spectral invariants } \label{sec:PFH_spec_initial_properties}
   Let $p_ -= (0,0,-1) \in \S^2$.   We denote $$\mathcal{S} := \{ \varphi \in \Diff(\S^2, \omega) : \varphi(p_-) = p_-, \,  -\tfrac{1}{4} < \rot(\varphi, p_-) < \tfrac{1}{4} \},$$
 where  $\rot(\varphi, p_-)$ denotes the rotation number of $\varphi$ at $p_-$ as defined in Section \ref{sec:rotation_number}. We remark that our choice of the constant $\frac14$ is arbitrary; any other constant in $(0,\frac12)$ would be suitable for us; we just need to slightly enlarge the class of diffeomorphisms arising from $\Diff_c(\mathbb{D}^2,\omega)$, so as to facilitate computations.

Recall from the previous section that the spectral invariant $c_{\sigma}$ depends on the choice of reference cycle $\gamma_0 \in Y_{\varphi}.$  For $\varphi \in S,$ there is a unique embedded Reeb orbit through $p_-$, and we set this to be the reference cycle $\gamma_0$.  

The grading on $\widetilde{PFH}$ depends on the choice of trivialization 
$\tau_0$ over $\gamma_0$; our convention in this paper is that we always choose $\tau_0$ such that the rotation number $\theta$ of the linearized Reeb\footnote{Following \cite[Section 3.2]{Hutchings-Notes}, we define the rotation number $\theta$  as follows: Let $\{\psi_t\}_{t\in \R}$ denote the $1$--parameter group of diffeomorphisms of $Y_\varphi$ given by the flow of the Reeb vector field.  Then, $D\psi_t : T_{\gamma_0(0)}Y_\varphi \rightarrow T_{\gamma_0(t)}Y_\varphi$  induces a symplectic linear map $\phi_t :\xi_{\gamma_0(0)} \rightarrow \xi_{\gamma_0(t)}  $, which using the trivialization $\tau_0$ we regard as a symplectic linear transformation of $\R^2$.  We define $\theta$ to be the rotation number of the isotopy $\{\phi_t\}_{t\in [0,1]}$ as defined in Section \ref{sec:rotation_number}. } flow along $\gamma_0$ with respect to $\tau_0$ satisfies $- \frac14 < \theta < \frac14 $: this determines $\tau_0$ uniquely.

We will want to single out some particular spectral invariants for $\varphi \in \mathcal{S}$, and show that they have various convenient properties; we will use these to define the spectral invariants for $\varphi \in \Diff_c(\D, \omega)$.

Having set the above conventions, we do this as follows.   Suppose that $\varphi \in \mathcal{S}$ is non-degenerate.  According to Equation \eqref{eq:PFH_sphere}, for every pair $(d,k)$ with $k = d$ mod $2$, we have a distinguished nonzero class $\sigma_{d,k}$ with degree $d$ and grading $k$, and so we can define $$c_{d,k}(\varphi) := c_{\sigma_{d,k}}(\varphi).$$  Lastly, we also define\footnote{ Alternatively, one may define $c_d(\varphi) := c_{d,k}(\varphi)$ for any $-d \le k \le d$ satisfying $k=d$ mod $2$.  These alternative definitions are all suitable for our purposes in this article.}
$$c_d(\varphi) := c_{d, -d} (\varphi).$$

We will see in the proof of Theorem~\ref{thm:PFHspec_initial_properties} that the $c_{d,k}(\varphi)$ for nondegenerate $\varphi$ determine $c_{d,k}(\varphi)$ for all $\varphi$ by continuity.

 To prepare for what is coming, we identify a class of Hamiltonians $\mathcal{H}$ with the key property, among others, that $\mathcal{S} = \{\varphi^1_H: H \in \mathcal{H} \}$.    Denote  
\begin{align*}
\mathcal{H} := \{H \in C^{\infty}(\S^1 \times \S^2) : &\  \varphi^t_H(p_-) = p_- , H(t, p_-) = 0,  \forall t \in [0,1],
\\  & -\tfrac14 < \rot( \{\varphi^t_H\}, p_- ) < \tfrac14  \},
\end{align*}
where $\rot( \{\varphi^t_H\} , p_-)$ is the rotation number of the isotopy $\{\varphi^t_H\}_{t \in [0,1] }$ at $p_-$; see Section \ref{sec:rotation_number}.    
Observe that %
$\mathcal{S} = \{\varphi^1_H: H \in \mathcal{H} \}.$

\medskip

  The theorem below, which is the main result of this section, establishes some of  the key properties of the PFH spectral invariants and furthermore allows us to extend the definition of these invariants to all, possibly degenerate, $\varphi \in \mathcal{S}$. 
  In the statement below $|| \cdot ||_{(1,\infty)}$ denotes the energy, or the {\bf Hofer norm}, on $C^{\infty}(\S^1 \times \S^2)$ which is defined as follows
  \[ \| H \|_{(1, \infty)} = \int_0^1 \left( \max_{x \in \S^2} H(t,x) - \min_{x \in \S^2} H(t, x)\right) dt.\]  
  
\begin{theo}
\label{thm:PFHspec_initial_properties}
The PFH spectral invariants $c_{d,k}(\varphi)$ admit a unique extension
to all $\varphi \in \mathcal{S}$ satisfying  the following properties: 
\begin{enumerate}
\item Monotonicity: Suppose that $H \leq G$, where $H, G \in \mathcal{H}$. Then,
\[ c_{d, k}(\varphi^1_H) \leq c_{d, k} (\varphi^1_G).\]
\item Hofer Continuity:  For any $H, G \in \mathcal H$, we have  $$| c_{d,k} (\varphi^1_H) - c_{d, k} (\varphi^1_G) | \leq d || H - G ||_{(1,\infty)}.$$ 
\item Spectrality: $c_{d, k}(\varphi^1_H) \in \Spec_d(H)$ for any $H \in \mathcal{H}$.

\item %
Normalization: $c_{d,-d}(\id) =0$.
\end{enumerate}
\end{theo}

\begin{remark}\label{rem:PFHspec-disc}
To define the PFH spectral invariant $c_{d, k}$ for $\varphi \in \Diff_c(\D, \omega) $, we use Equation \eqref{eq:identify_disc_north_hemisphere} to  identify $ \Diff_c(\D, \omega)$ with area-preserving diffeomorphisms of the sphere which are supported in the interior of the northern hemisphere $S^+$.
   
We similarly define $c_d:  \Diff_c(\D, \omega) \rightarrow \R$
which was introduced in Section \ref{sec:PFH_spectralinvariants_intro}.  It follows from Theorem~\ref{thm:PFHspec_initial_properties}   that $c_d : \Diff_c(\D, \omega) \rightarrow \R$ satisfies the properties 1-4 in Section \ref{sec:PFH_spectralinvariants_intro}. 
\end{remark}

The rest of this section is dedicated to the proof of the above theorem.  The proof requires certain preliminaries. First, it will be convenient to explicitly identify $Y_{\varphi}$ with $\S^1 \times \S^2$. To do so pick $ H \in \mathcal{H}$ such that $\varphi = \varphi^1_H$.\footnote{We remark that the choice of $H\in \mathcal H$  such that $\varphi = \varphi^1_H$ is unique up to homotopy of Hamiltonian isotopies rel endpoints.  This fact, which is not used in our arguments, may be deduced from properties of the rotation number.}
  We define 
\begin{equation} \label{eq:trivialization}
\begin{split}
&   \S^1 \times \S^2 \rightarrow Y_{\varphi} \\
 (t,x) & \mapsto  \left( (\varphi^{t}_H)^{-1} (x), t \right),
\end{split}
\end{equation}
where $t$ denotes the variable on $\S^1$.   For future reference, note that this identifies the Reeb vector field on $Y_\varphi$ with the vector field 
\begin{equation}
\label{eqn:trivializedreeb}
\partial_t + X_H
\end{equation} on $\S^1 \times \S^2$.  The $2$-form $\omega_{\varphi}$ pulls back under this map to the form 
\[ \omega + dH \wedge dt \]
where $\omega$ is the area form on $\S^2$.   

The Reeb orbit $\gamma_0$ 
maps under \ref{eq:trivialization} to the preimage of $p_-$ under the map $\S^1 \times \S^2 \to \S^2$; 
we will continue to denote it by $\gamma_0$.  
Moreover, the trivialization $\tau_0$ from above agrees (up to homotopy) under this identification with the trivialization over $\gamma_0$ given by pulling back a fixed frame of $T_{p_-} \S^2$ under the map  $\S^1 \times \S^2 \to \S^2$.   

The map \ref{eq:trivialization} allows us to identify  $\mathbb{R} \times \S^1 \times \S^2$ with the symplectization $X$ via
\begin{align*}
\mathbb{R} \times \S^1 \times \S^2 &\rightarrow X \\
(s,t,x) &\mapsto (s,(\varphi_{H}^t)^{-1}(x), t).
\end{align*}
 The symplectic form $\Gamma$ on $X$ then pulls back to
\begin{equation}
\label{eqn:trivsymplectic} 
\omega_H = ds \wedge dt + \omega + dH \wedge dt.
\end{equation}

Let $H, K$ be two Hamiltonians in $\mathcal H$.  As mentioned earlier, $\widetilde{PFH}(\varphi_H^1, d)$ is isomorphic to $\widetilde{PFH}(\varphi_K^1,d)$.  The proof of this uses Seiberg-Witten theory, and is carried out in \cite[Corollary 6.1]{Lee-Taubes}; this isomorphism is canonical with a choice of reference cycle in $H_2(\S^1 \times \S^2,\gamma_0, \gamma_0)$; we say more about this in Remark~\ref{rmk:twisted} below. 
We take this reference cycle to be the constant cycle\footnote{This is the projection $\gamma_0 \times I \to \gamma_0$.} over $\gamma_0$.  In this case, we will see below that the canonical isomorphism 
\begin{equation}
\label{eqn:choice}
\widetilde{PFH}(\varphi^1_H,d) \to \widetilde{PFH}(\varphi^1_K,d),
\end{equation}
preserves the $\mathbb{Z}$-grading.

 As is generally the case with related invariants, one might expect this isomorphism to be induced by a chain map counting certain $J$-holomorphic curves. In fact, it is not currently known how to define the map \eqref{eqn:choice} this way; the construction uses Seiberg-Witten theory.   Nevertheless, the map in \eqref{eqn:choice} does satisfy a ``holomorphic curve" axiom which was proven by Chen \cite{Chen} using  variants of Taubes' ``Seiberg-Witten to Gromov" arguments in \cite{Taubes-Gr=SW}.  A similar ``holomorphic curve" axiom was proven in the context of embedded contact homology by Hutchings-Taubes; we compare the Chen proof to the Hutchings-Taubes one in Remark~\ref{rmk:chen} below.

To state this holomorphic curve axiom in our context, take Hamiltonians $H, K \in \mathcal{H}$, and define for $s\in \R$
\[ G_s= K + \beta(s) \cdot (H-K)\]
where $\beta: \R \rightarrow [0,1]$ is some non-decreasing function that is $0$ for $s$ sufficiently negative and $1$ for $s$ sufficiently positive.  Now consider the form
\[\omega_X = ds\wedge dt + \omega + dG \wedge dt,\]
where, as throughout this article, $dG$ denotes the derivatives in the $\S^2$ directions.  This is a symplectic form on $\mathbb{R} \times \S^1 \times \S^2$.  Observe that, for $s > > 0$, the form $\omega_X$ agrees with the symplectization form $\omega_H$, and for $s < <0$, it agrees with the symplectization  form $\omega_K$.  Let $J_X$ be any $\omega_X$-compatible\footnote{Recall that an almost complex structure $J$ is {\bf compatible} with a symplectic form $\omega$ if $g(u,v) := \omega(u,Jv)$ defines a Riemannian metric.} almost complex structure that agrees with a generic $(dt,\omega_H)$ admissible almost complex structure $J_+$ for $s >> 0$ and with a generic $(dt,\omega_K)$ admissible almost complex structure $J_-$ for $s << 0$.

Then, the holomorphic curve axiom says that \eqref{eqn:choice} is induced by a chain map 
\begin{equation}
\label{eqn:chainchoice}
\Psi_{H,K}: \widetilde{PFC}(\varphi^1_H,d,J_+) \to \widetilde{PFC}(\varphi^1_K,d,J_-),
\end{equation}
with the property that if $ \langle \Psi_{H,K} (\alpha,Z), (\beta,Z')
\rangle \ne 0$, then there is an ECH index $0$
$J_X$-{\bf holomorphic building}   
$C$ from $\alpha$ to $\beta$ such that
\begin{equation}
\label{eqn:importantequation}
Z' + [C] = Z,
\end{equation}
 as elements of $H_2(\S^1 \times \S^2, \alpha, d \gamma_0);$ we say more about this in Remark~\ref{rmk:twisted} below.   Here, by a $J_X$-holomorphic building from $\alpha$ to $\beta$, we mean a sequence of $J_i$-holomorphic curves 
\[ (C_0,\ldots,C_i, \ldots, C_k),\] 
such that the negative asymptotics of $C_i$ agree with the positive
asymptotics of $C_{i+1}$, the curve $C_0$ is asymptotic to $\alpha$ at
$+\infty$, and the curve $C_k$ is asymptotic to $\beta$ at $- \infty$;
we refer the reader to \cite[Section 5.3]{Hutchings-Notes} 
for more
details.  We remark for future reference that the $C_i$ are called
{\bf levels}, and each $J_i$ is either\footnote{More can be said, but
  we will not need this additional information} $J_X, J_+$ or $J_-$.
The condition that the ECH index of the building is zero means that
the sum of the ECH indices of the levels add up to zero.  In
particular, this index condition, together with
\eqref{eqn:importantequation}, implies the earlier claim that the map
\eqref{eqn:choice} preserves the $\mathbb{Z}$-grading by additivity
of $c_{\tau}$ and $Q_{\tau}$, since the trivializations over $\gamma_0$ required to define the grading on $\widetilde{PFC}(\varphi^1_H)$ and $\widetilde{PFC}(\varphi^1_K)$ are the same.   

We will want to assume that $J_X$ is {\bf compatible with the fibration} $\R \times \S^1 \times \S^2 \rightarrow \mathbb{R} \times \S^1$ in the following sense:  Let $\mathbb{V}$ be the vertical tangent bundle of this fibration and denote by $\mathbb{H}$  the $\omega_X$-orthogonal complement of $\mathbb V$;  observe that $\mathbb{H}$ is spanned by the vector fields $ \partial_s$ and $ \partial_t + X_G $.  Then, we will want $J_X$ to preserve $\mathbb{V}$ and $\mathbb{H}$.  Given any admissible $J_{\pm}$ on the ends, we can achieve this as follows.  On the horizontal tangent bundle $\mathbb{H}$, we always demand that $J_X$ sends $\partial_s$ to $\partial_t + X_G$.  On the vertical tangent bundle, we observe that $\omega_X|_{\mathbb{V}} = \omega$, and in particular $\omega_X|_{\mathbb{V}}$ is independent of $s$; we can then connect $J_{+}|_{\mathbb{V}} $ to $J_{-}|_{\mathbb{V}}$ through a path of $\omega$-tamed almost complex structures on $\mathbb{V}$.

We can now prove Theorem~\ref{thm:PFHspec_initial_properties}.  
\begin{proof}[Proof of Theorem~\ref{thm:PFHspec_initial_properties}]
  We begin by first supposing that the monotonicity, Hofer continuity, and spectrality properties hold when $\varphi^1_K, \varphi^1_H$ are nondegenerate and explain how this implies the rest of the theorem.   To that end, let $H \in \mathcal H$, not necessarily nondegenerate, and take a sequence $H_i \in \mathcal H$ which $C^2$ converges to $H$ and such that $\varphi^1_{H_i}$ is nondegenerate.  Then, we define
$$ c_{d, k} (\varphi^1_H) = \lim_{i\to \infty} c_{d,k}(\varphi^1_{H_i}).$$  This limit exists thanks to the inequality $| c_{d,k} (\varphi^1_{H_i}) - c_{d, k} (\varphi^1_{H_j}) | \leq d || H_i - H_j ||_{(1,\infty)}.$  Moreover, the same inequality implies that the limit value does not depend on the choice of the sequence $H_i$ and so  $c_{d, k} (\varphi^1_H)$ is well-defined for all $H \in \mathcal H$.  Thus, we obtain a well-defined mapping $$ c_{d, k} : \mathcal{S} \rightarrow \R.$$
It can be seen that $c_{d, k}$ continues to satisfy the monotonicity
and Hofer continuity properties for degenerate $\varphi^1_K,
\varphi^1_H$.   The spectrality property is also satisfied; this is a
consequence of the Arzela-Ascoli theorem under the assumption that
spectrality is satisfied in the nondegenerate case; we will not
provide the details here.   %
 Moreover, note that, by the Hofer continuity property, the  mapping $ c_{d, k} : \mathcal{S} \rightarrow \R$ is uniquely determined by its restriction to the set of all non-degenerate $\varphi \in \mathcal S$.

To prove that $c_{d, -d}(\id) =0$, it is sufficient to show that $c_{d, -d}(\varphi) = 0$ in the case where $\varphi$ is a positive irrational rotation of the sphere, i.e. $\varphi(z, \theta) = (z, \theta + \alpha)$ with $\alpha $ being a small and  positive irrational number.  As in the explanation for Equation \eqref{eq:PFH_sphere}, the  chain complex $\widetilde{PFC}(\varphi, d)$ has a unique generator in indices $k$ such that $k =d$ mod $2$ and it is zero for other indices.  The unique generator of index $-d$ is of the form $(\alpha, Z)$ where $\alpha = \{(\gamma_0, d)\}$ and $Z$ is the trivial class in $H_2(Y_\varphi, d\gamma_0, d\gamma_0)$; 
this fact can be deduced from the forthcoming index computations of Section \ref{sec:index}. The action $\mathcal{A}(\alpha, Z)$ is zero.  This proves that $c_{d, -d} (\varphi) = 0 = c_{d, -d}(\id)$.\footnote{
With a similar argument one can prove that $c_{d, k}(\id) = 0$ for every $-d \le k \le d$ with $k=d$ mod $2$. }

\medskip

For the rest of the proof we will suppose that $\varphi^1_{H}, \varphi^1_K$ are nondegenerate.  We will now prove the monotonicity and Hofer continuity properties.  Let $(\alpha_1, Z_1) + \ldots + (\alpha_m, Z_m)$ be a cycle in $\widetilde{PFC}(\varphi^1_H,d)$  representing $\sigma_{d,k}$, with $$c_{\sigma_{d,k}}(\varphi^1_H) = \mathcal{A}(\alpha_1, Z_1)  \geq \ldots \geq \mathcal{A}(\alpha_m, Z_m).$$  Let $(\beta, Z')$ be a generator in  $ \widetilde{PFC}(\varphi^1_K,d)$ which has maximal action among generators   which appear with a non-zero coefficient in $$\Psi_{H,K}\left( (\alpha_1, Z_1) + \\ \ldots + (\alpha_m, Z_m) \right).$$

Then, by the holomorphic curve axiom there is a $J_X$-holomorphic building  from some $(\alpha_i, Z_i)$ to $(\beta, Z')$.   For the rest of the proof we will write  $(\alpha_i, Z_i) = (\alpha,Z)$  and will denote the $J_X$-holomorphic building  by $C$. 

For the arguments below, which only involve energy and index arguments,
we can assume that  $C$ consists of a single $J_X$-holomorphic level -- in other words, is an actual $J_X$-holomorphic curve, rather than a building -- so to simplify the notation,  we assume this.  

\medskip

 For the remainder of the proof we will need the following Lemma. 

\begin{lemma} \label{lemm:important_identity_inequality} The following identity holds: 
$$\mathcal{A}(\alpha, Z) - \mathcal{A}(\beta, Z') = \int_C \omega + dG\wedge dt + G' ds\wedge dt, $$
where  $G'$ denotes $\frac{\partial G}{\partial s}$. Furthermore, we have 
$$\int_C \omega + dG\wedge dt \geq 0.$$
\end{lemma}
\begin{proof}[Proof of Lemma \ref{lemm:important_identity_inequality}]
We will begin by proving that 
\begin{equation}
\label{eqn:actionequation}
\mathcal{A}(\alpha,Z) - \mathcal{A}(\beta,Z') =\int_C \omega + d(Gdt),
\end{equation}
which establishes the first item because $\omega+d(Gdt) = \omega + dG \wedge dt + G' ds \wedge dt.$  Note that we can write
\[ \mathcal{A}(\alpha,Z) = \int_Z \omega + d(Hdt), \quad \mathcal{A}(\beta,Z') = \int_{Z'} \omega + d(Kdt).\]
Hence,  Equation \eqref{eqn:actionequation} would follow if we show that
 \[ \int_C \omega =  \int_Z \omega -  \int_{Z'} \omega, \text{ and }  \int_C d(Gdt) = \int_Z d(Hdt) - \int_{Z'} d(Kdt) .\]
The first identity holds because all of these integrals are determined by the homology classes, and we have $[C]  = Z -  Z'.$  The second identity follows from the following chain of identities:
\begin{align*}
\int_C d(Gdt) &= \int_\alpha G dt - \int_\beta Gdt \\
                      &= \int_\alpha H dt - \int_\beta K dt \\
                      &= \int_Z d(Hdt) - \int_{Z'} d(Kdt),
\end{align*}
where the first equality holds by Stokes' theorem, the second follows from the definition of $G$, and the third is a consequence of Stokes' theorem combined with the fact that $H, K$ both belong to $\mathcal{H}$ and so vanish on $\gamma_0$. This completes the proof of the first item in the lemma.

Now, we will show that $\int_C (\omega + dG \wedge dt) \geq 0$ by showing that the form $\omega + dG \wedge dt$ is pointwise non-negative  along $C$.  Indeed, at any point $p \in X$, we can write any vector as $v + h$, where $v \in \mathbb{V}$ and $h \in \mathbb H$ are vertical and horizontal tangent vectors as described in the paragraph before the proof of Theorem \ref{thm:PFHspec_initial_properties}.  Since $C$ is $J_X$--holomorphic, it is sufficient to show that $(\omega + dG \wedge dt)(v + h, J_X v + J_X h) \geq 0.$  We will show that
\begin{equation} \label{eq:1}
(\omega + dG \wedge dt)(v + h, J_X v + J_X h) = \omega_X (v, J_Xv),
\end{equation}
which proves the inequality because $J_X$ is $\omega_X$-tame.  Now, to simplify our notation we will denote $\Omega = \omega + dG \wedge dt$ for the rest of the proof.  Expanding the left hand side of the above equation we get
 \[\Omega(v + h, J_Xv + J_Xh) =  \Omega(v, J_Xv) + \Omega(h, J_Xh) + \Omega(v, J_Xh) + \Omega(h, J_Xv). \]
We will now show that $\Omega(v, J_Xv) = \omega_X(v, J_Xv)$ and
$\Omega(h, J_Xh)= \Omega(v, J_Xh) = \Omega(h, J_Xv) =0$ which clearly
implies Equation \eqref{eq:1}.   
To see this, note that $v$ and $J_Xv$ are in the kernel of $ds \wedge dt$, hence 
\[ \Omega(v, J_Xv) = \omega_X(v, J_X v),\]
\[  \Omega(v, J_X h) = \omega_X(v,J_X h) = 0, \quad \Omega(h, J_X v) =\omega_X(h,J_X v) = 0.\]
It remains to show that $\Omega(h, J_X h) = 0$, that is $\Omega\vert_\mathbb{H} = 0$.   This follows from the fact that $\bb H$ is spanned by $\lbrace \partial_s, \partial_t + X_G \rbrace$ and $\partial_s$ is in the kernel of $\Omega$. Indeed, a $2$--form on a $2$--dimensional vector space with non-trivial kernel is identically zero. 
\end{proof}

\medskip 
Note that  $c_{d,k}(\varphi^1_H) \geq \mathcal{A}(\alpha, Z)$ and  $c_{d,k}(\varphi^1_K) \leq  \mathcal{A}(\beta,Z').$  Hence,
\begin{equation}\label{eqn:lower_bound_action}
 c_{d,k}(\varphi^1_H) - c_{d,k}(\varphi^1_K) \geq \mathcal{A}(\alpha,Z) - \mathcal{A}(\beta,Z').
\end{equation}
As a consequence of this inequality, Monotonicity would follow from proving that if $H \geq K$, then $\mathcal{A}(\alpha,Z) - \mathcal{A}(\beta,Z') \geq 0.$ By the above lemma we have \begin{equation}
\label{eqn:masterequation} 
\mathcal{A}(\alpha,Z) - \mathcal{A}(\beta,Z') \geq \int_C G'\,ds \wedge dt.
\end{equation} 
 If $H \geq K$, then $G' \geq 0$, and so $\int_C G' ds \wedge dt \geq 0$, which proves Monotonicity.
 
 As for Hofer Continuity, it is sufficient to show that 
 \begin{equation}\label{eq:estimate_hofer}
 \left|\int_C G'\,ds \wedge dt \right| \leq d \Vert H-K \Vert_{(1, \infty)}.
 \end{equation}  
 Indeed, this inequality combined with Inequalities \eqref{eqn:lower_bound_action} and \eqref{eqn:masterequation} implies that $c_{d,k}(\varphi^1_K) - c_{d,k}(\varphi^1_H) \leq d \Vert H-K \Vert_{(1, \infty)}.$  Similarly, by switching the role of $H$ and $K$,
one gets $c_{d,k}(\varphi^1_H) - c_{d,k}(\varphi^1_K) \leq d \Vert K -H \Vert_{(1, \infty)}$ which then implies Hofer Continuity.  
 
 It remains to prove Inequality  \eqref{eq:estimate_hofer}.  We know that the form $ds \wedge dt$ is pointwise nonnegative on $C$, because $ds \wedge dt|_{\bb V} = 0$, we saw in the proof of the previous lemma that $\Omega|_{\bb H} = 0$, and $J_X$ is $\omega_X$-tame.   %
Hence
\[ \left|\int_C G' ds \wedge dt \right| = \left|\int_C \beta'(s) (H-K) ds \wedge dt \right| \leq \int_C \beta'(s) |H-K| ds \wedge dt.\]  
Note that because $H, K$ both vanish at the point $p_-$, for all $t,x$ we have
\[ |H(t,x) - K(t,x) | \leq \max_{\S^2} \left (H_t -K_t\right) -\min_{\S^2} \left (H_t -K_t \right) .\]
Hence, we get 
\[ \left|\int_C G' ds \wedge dt \right| \leq \int_C \beta'(s) \left(\max_{\S^2} (H-K) - \min_{\S^2}(H - K)\right)  ds \wedge dt.\]
We can evaluate the second integral by projecting $C$ to the  $(s,t)$ plane; this projection has degree $d$, 
and since $\int_{-\infty}^{+\infty}\beta'=1$, the second integral evaluates to
\[ d ||H-K||_{(1,\infty)} .\]
This completes the proof of Hofer Continuity.

\medskip

It remains to prove Spectrality.  As stated in Section \ref{sec:specdefn}, the spectral invariant $c_{d,k}(\varphi^1_H)$ is the action of a PFH generator $(\alpha,Z)$  of degree $d$. Spectrality,  hence  Theorem \ref{thm:PFHspec_initial_properties}, is then a consequence of the following lemma.
\end{proof}

\begin{lemma}\label{lemma:action-action} Let $(\alpha, Z)$ be a PFH generator of degree $d$ for the mapping torus $Y_\varphi$ of $\varphi=\varphi_H^1$ with $H \in \mathcal{H}$. Then, $\cal A(\alpha, Z)$ belongs to  $\Spec_d(H)$, as defined in Section \ref{sec:action_spectra}.
\end{lemma}

\begin{proof} For every orbit set $\alpha$ we will construct a specific relative class $Z_\alpha \in H_2(Y_\varphi, \alpha, \gamma_0)$ and will show that $\mathcal{A}(\alpha, Z_\alpha) \in \Spec_d(H)$.  Any other $Z\in H_2(Y_\varphi, \alpha, \gamma_0)$ is of the form  $Z_\alpha + k [\S^2]$ where $k \in \Z$.   Hence, $\mathcal{A}(\alpha, Z) = \mathcal{A}(\alpha, Z_\alpha)  + k$  and so we get that $\mathcal{A}(\alpha, Z) \in \Spec_d(H)$ for all $Z$.

First, suppose that $d=1$.  Let $q \in \Fix(\varphi)$ and suppose that $\alpha$ is the Reeb orbit in the mapping cylinder corresponding to $q$.  The relative cycle $Z_{\alpha}$ will be of the form 
$Z_{\alpha}= Z_0 + Z_1 + Z_2$.  We begin by choosing a path $\eta$ in $\S^2 \times\{0\} \subset Y_{\varphi^1_H}$ such that  $\partial \eta = (q,0) - (p_-, 0)$.   We parametrize this curve with a variable $x \in [0,1]$. 
We define $Z_0$ to be the chain induced by the map 
\begin{align*}
\begin{split}
&[0,1]^2 \to Y_{\varphi^1_H}, \;\; (x,t)\mapsto (\eta(x),t). 
\end{split}
\end{align*}
Its boundary is given by $\partial Z_0= \alpha - \gamma_0 + (\eta, 0) - (\varphi(\eta), 0)$.
Note that  $ \int_{Z_0} \omega_{\varphi} =0.$

We define $Z_1$ to be the chain induced by the map
\begin{align*}
\begin{split}
&[0,1]^2 \to Y_{\varphi^1_H}, \;\; (t,x) \mapsto (\varphi^t_H( \eta(x)), 0).
\end{split}
\end{align*}
Then,  $\partial Z_1 = (\varphi(\eta), 0) - (\eta, 0) - (\varphi^t_H(q), 0)$. We can compute 
$\int_{Z_1} \omega_{\varphi}$ as follows:
\begin{align*}
\int_{Z_1} \omega_{\varphi} &= \int \int_{[0,1]^2} \omega \langle \partial_t \varphi^t_H(\eta(x) ) , \partial_x \varphi^t_H (\eta(x)) \rangle
\\  
 &= \int \int_{[0,1]^2} \omega \langle X_{H_t} (\varphi^t_H(\eta(x ))) , \partial_x \varphi^t_H (\eta(x)) \rangle 
\\  &= \int \int _{[0,1]^2} dH_t (  \partial_x \varphi^t_H( \eta(x)  ) ) 
 = \int \int _{[0,1]^2}  \partial_x H_t ( \varphi^t_H (\eta(x)) ) 
\\ &= \int^{1}_{0} H_t (\varphi^t_H(q)) - H_t(\varphi^t(p_-))dt = \int^1_0 H_t(\varphi^t_H(q)) dt.
\end{align*}

Next, we define $Z_2$ to be the chain induced by a
map $(u_\alpha, 0)$ where $u_\alpha:D^2 \rightarrow \S^2$  is such that  $u_\alpha|_{\partial D^2}$ is the Hamiltonian orbit $t\mapsto\varphi^t_H(q)$.  Then, $\partial Z_2= (\varphi^t_H(q), 0)$ and  $\int_{Z_2} \omega_{\varphi} = \int_{D^2} u_\alpha^* \omega$. We will furthermore require $u_\alpha$ to satisfy the following additional properties which will be used in Section \ref{sec:comb}:
\begin{enumerate}[(i)]
\item  If $\alpha = \gamma_0$, the Reeb orbit corresponding to $p_-$, then we take $u_\alpha$ to be the constant disc with image $p_-$.
\item  If $\alpha \neq \gamma_0$, then we take  $u_\alpha$ such that its image does not contain  $p_-$.
\end{enumerate}

Finally, we set $Z_\alpha = Z_0 + Z_1 + Z_2$.  Adding up the above quantities we obtain $\partial Z_\alpha=\partial Z_0 + \partial Z_1 + \partial Z_2 = \alpha - \gamma_0$ and 
\begin{equation}\label{eqn:action_Z_alpha}
\mathcal{A}(\alpha,Z_\alpha) =   \int_{Z_0} \omega_{\varphi} + \int_{Z_1} \omega_{\varphi} + \int_{Z_2} \omega_{\varphi}  = \int_{D^2} u_\alpha^* \omega +  \int^1_0 H_t(\varphi^t_H(q)) dt. 
\end{equation} 
Clearly, $\mathcal{A}(\alpha,Z_\alpha) \in \Spec(H)$.  We remark that $Z_{\alpha}$ does not depend on the choice of $u_{\alpha}$ subject to the above conditions.  

Next, let $q$ be a periodic point of period $m \in \N$ and suppose that $\alpha$ is the Reeb orbit in the mapping cylinder corresponding to $q$.  Then, $q$ is a fixed point of $\varphi_H^m$.  Consider the mapping torus $Y_{\varphi_H^m}$.  There is a map $c:Y_{\varphi_H^m} \to Y_{\varphi_H^1}$, pulling back $\omega_{\varphi^1_H}$ to $ \omega_{\varphi^m_H}$, given by mapping each interval $\S^2 \times [\frac{k}{m}, \frac{k+1}{m}]$ onto $Y_{\varphi}$ via the map
\[ (x,t) \mapsto (\varphi_H^{k}(x) , m \cdot t - k).\]
Now repeat the construction from above 
to produce a relative cycle $Z'$ in $Y_{\varphi^m}$. 
Define $Z_\alpha$ to be the pushforward of $Z'$  under the map $c$.  Then 
\[ \int_{Z_\alpha} \omega_{\varphi_H^1} = \int_{Z'} \omega_{\varphi_H^m}.\] 
The map $\varphi_H^m$ is generated by the Hamiltonian $H^m$, so by the argument in the $d = 1$ case, $\int_{Z'} \omega_{\varphi_H^m} \in \Spec(H^m)$, hence the same holds for $\int_{Z_{\alpha}} \omega_{\varphi^1_H}.$  We note for future reference that as \eqref{eqn:action_Z_alpha} holds for $\mathcal{A}(\alpha,Z')$ with respect to $\varphi^m_H$, and $\varphi^m_H$ can be viewed as the time $1$-map of the Hamiltonian $F_t(x) = m H_{mt}(x)$, we have
\begin{equation}\label{eqn:action_Z_alpha_2}
\mathcal{A}(\alpha,Z_\alpha)  = \int_{D^2} u_\alpha^* \omega +  \int^m_0 H_t(\varphi^t_H(q)) dt,
\end{equation}
where $u_\alpha$ is a disc whose boundary is the Hamiltonian orbit $t \mapsto \varphi^t_H(q), t\in [0, m]$, which satisfies the analogues of properties (i) and (ii) above.

Now let $\alpha = \lbrace (\alpha_i, m_i) \rbrace$, where the $\alpha_i$ are simple closed orbits of $\partial_t$.  So, each $(\alpha_i,m_i) $ corresponds to a (not necessarily simple) orbit of a periodic point $q_i$ of $\varphi^1_H$.  By using the construction in the previous paragraph, we can associate a relative cycle to each $(\alpha_i, m_i)$; the sum, over $i$, of all of these cycles gives a relative cycle from $\alpha$ to $d \gamma_0$, where $d$ is the sum of the periods of the periodic points $q_i$.  The arguments in the previous paragraphs show that the action of this cycle $\mathcal{A}(\alpha, Z)$ is in $\Spec_d(H)$.
\end{proof}

\begin{remark}
\label{rmk:nojdependence}
In the special case where $H = K$, but the two $J_i$ are different, the Monotonicity argument above, applied first to $H \geq K$ and next to $K \geq H$, gives that the spectral invariant does not depend on $J$.
\end{remark}

\begin{remark} \label{rmk:chen} For the benefit of the reader, 
we provide a brief comparison between the Chen proof and the  
proof of the holomorphic curve axiom in the ECH context by Hutchings-Taubes; we also provide a sketch of the Chen proof.  For brevity, we assume in the remark that the reader is familiar with some of the relevant background, referring
to \cite{Hutchings-Taubes-ChordII} for relevant terminology.

In the Hutchings-Taubes context, one starts with an exact symplectic cobordism between contact manifolds, and attaches symplectization-like ends, to get a non-compact symplectic manifold $(\overline{X},\omega).$   Then, Hutchings-Taubes prove a ``holomorphic curve axiom" for any $\omega$-compatible almost complex structure $J$ that agrees with symplectization-admissible almost complex structures $J_{\pm}$ on the symplectization ends.  Their proof goes via adapting a ``Seiberg-Witten to Gromov" type degeneration that was needed by Taubes to show the equivalence of ECH and HM, to this context.  Essentially, just as Taubes' proof that ECH = HM, they write down a family of perturbations to the four-dimensional Seiberg-Witten equations, so that when there is a solution counted by the Seiberg-Witten cobordism map, it degenerates after perturbation to a holomorphic building.  Although in the Hutchings-Taubes case 
the symplectic form is assumed exact,  it was observed by Hutchings (see e.g. \cite{Hutchings-blog}) that this is not essential.

In the Chen context, one starts with a fibered symplectic cobordism between (fiberwise symplectic) mapping tori, and again attaches symplectization-like ends, analogously to the Hutchings-Taubes context, to get a non-compact symplectic manifold $(\overline{X},\omega)$; the main difference here is that the symplectic form on the ends in the Hutchings-Taubes context is the symplectization form for a contact form, whereas in the Chen context, the symplectic form is the symplectization form for the mapping torus.  Just like in Hutchings-Taubes, Chen proves a ``holomorphic curve axiom" for any $\omega$-compatible almost complex structure $J$ that agrees with symplectization-admissible almost complex structures $J_{\pm}$ on the ends; this also essentially goes via a ``Seiberg-Witten to Gromov" type degeneration 
\end{remark}

\begin{remark}
\label{rmk:twisted}
On the Seiberg-Witten side, the twisted theory corresponds to a version of the Floer homology, where, instead of taking the quotient of solutions by the full gauge group $\mathcal{G} = C^{\infty}(M,\S^1)$, one only takes the quotient by the subgroup $\mathcal{G}^0 \subset \mathcal{G}$ of gauge transformations in the connected component of the identity.   This has an $H^1(Y)$ action, induced by the action via gauge transformations, which corresponds to the $H_2(Y)$ action on twisted PFH given by adding a homology class.

As mentioned above, it was remarked by Taubes \cite{Taubes-ECH=SW-V}, Sec.\ 1, that the twisted invariant on the PFH/ECH side is canonical only up to a choice of element of $H_2(Y,\rho,\rho')$, where $\rho, \rho'$ are two reference cycles.   
Implicit in this assertion is that after a choice of reference cycle $R$, the isomorphism \eqref{eqn:chainchoice} satisfies a holomorphic curve axiom, for buildings $C$ satisfying
\begin{equation}
\label{eqn:cobordismequation}
Z + R = [C] + Z'.
\end{equation}
This is the best way to think about \eqref{eqn:importantequation}: this corresponds to the case where our reference cycle is constant over $\rho$.  

Since \eqref{eqn:cobordismequation} is only implicit in Taubes, we also should emphasize that the full force of \eqref{eqn:cobordismequation} is not needed for our purposes.  Indeed, all we need to know to prove Theorem~\ref{thm:PFHspec_initial_properties} is that the map \eqref{eqn:importantequation} satisfies the weaker axiom that if there is a nonzero coefficient of $\Psi_{H,K}(\alpha,Z)$ on $(\beta,Z')$, then there is an $I=0$ holomorphic building $C$ from $\alpha$ to $\beta$; this is proved in \cite{Chen}.   With this weaker axiom, we can no longer assume  \eqref{eqn:importantequation}.  Still, however, most of the proof of Theorem~\ref{thm:PFHspec_initial_properties} goes through unchanged, except for two exceptions, which cancel each other out.  The first difference is that  \eqref{eqn:actionequation} in Lemma~\ref{lemm:important_identity_inequality} requires that
\[ \int_C \omega = \int_{Z - Z'} \omega.\]
The above equation does not hold unless we know that $[C] = Z - Z'$; and, without \eqref{eqn:importantequation}, we only know that $[C] = Z - Z' + x [\S^2]$ for some $x$.  Thus, \eqref{eqn:actionequation} must be modified by subtracting $x$ from the right hand side, and so the identity asserted by Lemma~\ref{lemm:important_identity_inequality} must also be modified by subtracting $x$ from its right hand side.  However, this change is balanced by the fact that as $I([C]) = 0$, we must also have $I(Z) - I(Z') = -x(2d+2)$, and so the left hand side of \eqref{eqn:lower_bound_action} must also be modified to $c_{d,k - x(2d+2)}(\varphi^1_H) - c_{d,k}(\varphi^1_K)$; since $c_{d,k - x(2d+2)}(\varphi^1_H) = c_{d,k}(\varphi^1_H) - x$, this precisely balances the required modification of the statement of Lemma~\ref{lemm:important_identity_inequality}.

For more about the connection between the twisted theory and the relevant Seiberg-Witten Floer homology, we refer the reader to \cite{Taubes-ECH=SW-V} Sections 1 - 2.  In  \cite{Taubes-ECH=SW-V} Sections 1 - 2, Taubes is writing about twisted ECH; we have adapted what is written there to the PFH context, as suggested by  \cite{Lee-Taubes}, Cor. 6.1.
\end{remark}

\section{$C^0$ continuity}\label{sec:C0_continuity}

Here we prove Theorem~\ref{theo:C0_continuity}, using Theorem \ref{thm:PFHspec_initial_properties}  from Section  \ref{sec:pfhs2}.

The central objects of Theorem ~\ref{theo:C0_continuity} are the maps
$c_d:\Diff_c(\D^2, \omega)\to\R$. Remember from Section \ref{sec:PFH_spec_initial_properties} and Remark \ref{rem:PFHspec-disc} that these maps are defined from the spectral invariants $c_d:\cal S\to\R$, by identifying $\Diff_c(\D^2, \omega)$ with the group $\DiffS$ consisting of symplectic diffeomorphisms of $\S^2$, which are supported in the interior of the northern hemisphere $S^+$. In the present section, we directly work in the group $\DiffS$.

More generally, given an open subset $U\subset \S^2$, we will denote by $\DiffU$ the set of all Hamiltonian diffeomorphisms compactly supported in an open subset $U$.

Our proof is inspired by the proof of the $C^0$-continuity of barcodes (hence, of spectral invariants) arising from Hamiltonian Floer theory
presented in \cite{LSV}.
Other existing proofs of $C^0$-continuity of spectral invariants make use\footnote{The product is usually used to deduce continuity everywhere from continuity at $\id$. Without a product, we need another argument to prove continuity in the complement of the identity.}  of the product structure on Hamiltonian Floer homology.  It might be possible to define a ``quantum product" on PFH, see \cite{Hutchings-Sullivan-Dehntwist}, however, at the time of the writing of this article such structures do not exist.

Let $d$ be a positive integer. As in \cite{LSV}, we treat separately the $C^0$-continuity of $c_d$ at the identity and elsewhere. Theorem \ref{theo:C0_continuity} will be a consequence of the following two propositions. 

\begin{prop}\label{prop:C0-continuity-at-id} The map $c_d:\DiffS\to \R$ is continuous at $\id$ with respect to the $C^0$-topology on $\DiffS$.  
\end{prop}

\begin{prop}\label{prop:C0-continuity-elsewhere} Every area-preserving homeomorphism $\eta\in\HomeoS$ with $\eta\neq \id$ admits a $C^0$-neighborhood $\cal V$ such that the restriction of $c_d$ to $\cal{V}\cap\DiffS$ is uniformly continuous with respect to the $C^0$-distance. 
\end{prop}

This last proposition readily implies that any $\eta\in \HomeoS\setminus\{\id\}$ admits a $C^0$-neighborhood $\cal V$ to which $c_d$ extends continuously. In particular, it extends continuously at $\eta$. Since this holds for any such homeomorphism $\eta$, this shows together with Proposition \ref{prop:C0-continuity-at-id} that $c_d$ extends to a map $\HomeoS\to\R$ continuous with respect to $C^0$-topology, hence Theorem \ref{theo:C0_continuity}. 

Proposition \ref{prop:C0-continuity-elsewhere} can be rephrased as follows. Any homeomorphism $\eta\in\HomeoS$, $\eta\neq \id$, admits a neighborhood $\cal{V}$ in $\HomeoS$ such that for all $\eps>0$, there exists $\delta>0$ satisfying:  
\begin{align}
\forall \phi,\psi\in \cal{V} \cap{\DiffS}, \text{ if } d_{C^0}(\phi,\psi)<\delta \text{ then }
|c_d(\phi)-c_d(\psi)|<\eps. \label{eq:uniform-continuity}
\end{align}

\paragraph{The Hofer norm.}
Our proofs  will make intensive use of the Hofer norm for Hamiltonian diffeomorphisms. We now recall its definition and basic properties. We refer the reader to \cite{Polterovich2001} and the references therein for a general introduction to the material presented here.

We have seen earlier in the paper the definition of the Hofer norm of a Hamiltonian on the sphere and the disc.  On a general symplectic manifold, the {\bf Hofer norm} of a compactly supported Hamiltonian diffeomorphism $\phi$ is defined as
\[\|\phi\|=\inf\{\|H\|_{(1,\infty)}\},\]
where the infimum runs over all compactly supported Hamiltonians $H$ whose time-$1$ map is $\phi$.   It satisfies a triangle inequality 
\[ \|\phi\circ\psi\|\leq\|\phi\|+\|\psi\|,\]
for all Hamiltonian diffeomorphisms $\phi, \psi$, it is conjugation invariant and moreover, we have $\|\phi^{-1}\|=\|\phi\|$ for all Hamiltonian diffeomorphisms $\phi$. 

We remark, in connection with Section  \ref{sec:infinitedistance}, that we can also define a metric, called {\bf Hofer's metric}, by $d_{\mathrm{Hofer}}(\phi,\psi) = \| \phi \psi^{-1} \|$; this is a nondegenerate, bi-invariant metric on the group of compactly supported Hamiltonian diffeomorphisms, but we will not say more about it here because we do not need it for the results in this paper; for more details, we refer the reader to \cite{Polterovich2001}.

The displacement energy of a subset $A$ of the ambient symplectic manifold is by definition the quantity
\[e(A) :=\inf\{\|\phi\|: \phi(\overline{A})\cap \overline{A}=\emptyset\}.\]
On a surface, it is known that for a disjoint union of closed discs, with each disc having area $a$, and whose union covers less than half the area of the surface, the displacement energy is $a$.

\medskip

\noindent\textbf{Important note:} We will use the Hofer norm on the symplectic manifold  $\S^2\setminus\{p_-\}$. Thus, all the Hamiltonians considered in this section will be compactly supported in the complement of the south pole $p_-$; in particular, they belong to $\mathcal{H}$. 

\medskip
Note that the second item of Theorem \ref{thm:PFHspec_initial_properties} can be reformulated as
\begin{equation}
  \label{eq:hofer-bound}
  |c_d(\psi)-c_d(\phi)|\leq d\cdot\|\psi^{-1}\circ\phi\|,
\end{equation}
for all Hamiltonian diffeomorphisms $\phi, \psi\in\Diff_{\S^2\setminus \{p_-\}}(\S^2,\omega)$. To see this, let $K$ be a Hamiltonian such that $\psi=\varphi_K^1$. For any Hamiltonian $H$ such that $\psi^{-1}\circ\phi=\varphi_H^1$, then $\phi=\varphi_{K\# H}^1$ by (\ref{eqn:somehams}), thus by  Theorem \ref{thm:PFHspec_initial_properties} we obtain
\begin{align*}
|c_d(\psi)-c_d(\phi)|&=|c_d(\varphi_K^1)-c_d(\varphi_{K\# H}^1)|\\ &\leq d\cdot \|K-K\# H\|_{(1,\infty)}=d\cdot\|H\|_{(1,\infty)}.
\end{align*}
Inequality (\ref{eq:hofer-bound}) follows.

\subsection{Continuity at the identity}
We first prove Proposition \ref{prop:C0-continuity-at-id}. The case $d=1$ can be proved with the same proof as \cite{Sey13}, using the so-called $\eps$-shift technique. We will generalize this idea to make the proof work for all $d\geq 1$. 

Let us start our proof with a lemma.  

\begin{lemma}\label{lemma:eps-shift} Let $d\geq 1$ and let $F$ be a time-independent Hamiltonian, compactly supported in  $\S^2\setminus\{p_-\}$
  and $f=\phi_F^1$ the time-one map it generates. Assume that the next two conditions are satisfied:
  \begin{enumerate}[(a)]
  \item for all $k\in\{1,\dots, d\}$, the $k$-periodic points of $f$ are precisely the critical points of $F$;   
  \item none of the critical points of $F$ are in the closure of $S^+$.  
  \end{enumerate}
Then, there exists $\delta>0$ such that $c_d(\phi\circ f)=c_d(f)$, for any $\phi\in\DiffS$ with $d_{C^0}(\phi,\id)<\delta$.   
\end{lemma}

Postponing the proof of this lemma, we now explain how it implies Proposition \ref{prop:C0-continuity-at-id}.

\begin{proof}[Proof of Proposition \ref{prop:C0-continuity-at-id}]
  Let $\eps>0$. Let $F$ be a function on $\S^2$ satisfying the assumptions of Lemma \ref{lemma:eps-shift} and assume furthermore that 
  \[ \max F-\min F\leq \frac\eps{2d}.\]
  For instance a $C^2$-small function supported in the complement of  $p_-$
 all of whose critical points are in the southern hemisphere is appropriate. Let then $\delta$ be as provided by Lemma \ref{lemma:eps-shift}.

Let $\phi\in\DiffS$ be such that $d_{C^0}(\phi,\id)<\delta$. Then, we have
 \[c_d(\phi\circ f)=c_d(f).\]
 Using Inequality (\ref{eq:hofer-bound}) and the fact that $c_d(\id)=0$, we obtain
 \[|c_d(\phi\circ f)|=|c_d(f)-c_d(\id)|\leq d\,\|f\|.\]
We then deduce from Inequality (\ref{eq:hofer-bound}):
 \begin{align*}
   |c_d(\phi)|\leq |c_d(\phi\circ f)| + d\|f\|\leq 2d\|f\|.
 \end{align*}
 Now, by definition of the Hofer norm, $\|f\|\leq \max F-\min F$. Thus we get
\[c_d(\phi)\leq \eps.\]
This show the $C^0$-continuity of $c_d$ at $\id$.
\end{proof}

We will now prove the lemma.

\begin{proof}[Proof of Lemma \ref{lemma:eps-shift}] Let $F$ be as in the statement of the lemma and $f=\phi_F^1$. We want to prove that $c_d(f)$ remains unchanged when we $C^0$-perturb $f$ with a Hamiltonian diffeomorphism supported 
in the northern hemisphere.  To obtain this, we will first prove that the entire spectrum remains unchanged under such perturbations.

Let $k\in\{1,\dots, d\}$; we begin by showing that the set of $k$-periodic points is unchanged by these perturbations.    By assumption, there exists $c>0$ such that \[d(f^k(x),x)>c\] for all $x$ in the closure of $S^+$. Now note that the diffeomorphism $(\phi f)^k$ converges to $f^k$ uniformly when $\phi$ tends uniformly to $\id$. Thus, there exists $\delta>0$ such for $d_{C^0}(\phi, \id)<\delta$, the inequality $d((\phi f)^k(x),x)>0$ holds for all $x$ in the closure of $S^+$. In other words, $\phi\circ f$ has no $k$-periodic points in the closure of $S^+$. Since $\phi$ coincides with the identity outside $S^+$, this implies that $\phi\circ f$ and $f$ have the same $k$-periodic points, which are in turn the critical points of $F$. For the rest of the proof we pick $\delta$ such that the above holds for all $k\in\{1,\dots, d\}$, and $\phi$ such that $d_{C^0}(\phi,\id)<\delta$.

We next show that the actions of these $k$-periodic points, i.e. the critical points of $F$,  agree when computed with respect to $f$ and $\phi\circ f$.  

To compute these actions, let $H$ be a Hamiltonian supported in $S^+$ such that $\varphi_H^1=\phi$. By (\ref{eqn:somehams}), the isotopy $\varphi_H^t\varphi_F^t$ (whose time one map is $\phi\circ f$) is generated by the Hamiltonian $H\# F(t,x)=H(t,x)+F((\varphi_H^t)^{-1}(x))$. 

Let $y$ be a critical point of $F$. Then, $\varphi_F^t(y)=y$ for all $t\in[0,1]$, and since $y\notin S^+$, we also have $\varphi_H^t(y)=y$ for all $t\in[0,1]$. Thus $y$ remains fixed along the whole isotopy. A capping of such an orbit is a trivial capping to which is attached $\ell[\S^2]$ for some $\ell\in\Z$. Also note that since $H$ is supported in $S^+$,
\[H\# F(t,\varphi_H^t\varphi_F^t(y))=H(t,y)+F(y)=F(y).\]
Applying Formula (\ref{eqn:actiondefn}) we obtain
\begin{align*}
  \mathcal{A}_{H\# F}(y,\ell[\S^2])&=\int_0^1 H\# F(t,\varphi_H^t\varphi_F^t(y))dt +\ell \mathrm{Area}(\S^2).\\
                      &= F(y)+\ell\\
  &= \mathcal{A}_F(y,\ell[\S^2]).  
\end{align*}
This shows that $\spec(H\#F)=\spec(F)$. A similar argument shows that $\spec((H\#F)^k)=\spec(F^k)$ for all $k\in\{1,\dots, d\}$, hence $\spec_d(H\# F)=\spec_d(F)$.

By (\ref{eqn:spec_time1}), we have proved
\begin{equation}
\spec_d(\phi\circ f)=\spec_d(f)\label{eq:eps-shift}
\end{equation}
for all $\phi\in\DiffS$ such that $d_{C^0}(\phi, \id)<\delta$.

\medskip
There remains to deduce $c_d(\phi\circ f)=c_d(f)$ from this equality of spectrums. 

Given $\phi\in\DiffS$ such that $d_{C^0}(\phi, \id)<\delta$, one can construct, using the Alexander isotopy, a Hamiltonian isotopy $(\varphi_K^t)_{t\in[0,1]}$ in $\DiffS$, such that $d_{C^0}(\varphi_K^s,\id)<\delta$ for all $s\in[0,1]$ and $\varphi_K^1=\phi$; we refer the reader to \cite[Lemma 3.2]{Sey13} for the details.   

Equation (\ref{eq:eps-shift}) then implies $\spec_d(\varphi_K^s\circ f)=\spec_d(f)$, for all $s\in [0,1]$.
Now, by Theorem~\ref{thm:PFHspec_initial_properties} 
the function $s\mapsto c_d(\varphi_K^s\circ f)$ is continuous and takes its values in the measure $0$ subset $\spec_d(f)\subset \R$.  As a consequence, it is constant. This shows $c_d(\phi\circ f)=c_d(f)$ and concludes our proof. 
\end{proof}

\subsection{Continuity away from the identity}
We now turn our attention to Proposition \ref{prop:C0-continuity-elsewhere}. We want to prove \eqref{eq:uniform-continuity}, or equivalently, that any $\eta\in\HomeoS$, $\eta\neq \id$, admits an open neighborhood $\mathcal{V}$ such that for all $\eps>0$, there exists $\delta>0$ satisfying
\begin{align}
  \forall f\in\mathcal{V}& \cap \DiffS, \forall g\in\DiffS, \label{eq:uniform-conitnuity2}\\
  &\text{ if }d_{C^0}(g, \id)<\delta \text{ then } |c_d(gf)-c_d(f)|<\eps.\nonumber
 \end{align}

Our proof will follow from three lemmas, which we now introduce. 

To state the first, let us introduce some terminology. We will say that a diffeomorphism $f$ \textbf{$d$-displaces} a subset $U$ if the subsets $U$, $f(U)$, $\dots$, $f^d(U)$ are pairwise disjoint.  Our first lemma states that for $g$ supported in an open subset $d$-displaced by $f$, an even stronger version of \eqref{eq:uniform-conitnuity2} holds.  It is inspired by \cite[Lemma 3.2]{usher10}, which can be regarded as the analogue for the $d = 1$ case.

\begin{lemma}\label{lemme:cd-under-displacement} Let $f\in\DiffS$ and $B$ be an open topological disc %
  whose closure is included in  $\S^2\setminus\{p_-\}$
  and which is $d$-displaced by $f$. Then, for all $\phi\in\DiffB$, we have $c_d(\phi\circ f)=c_d(f)$.
\end{lemma}

We will prove this lemma in Section \ref{sec:proof-of-lemmas}. 

To apply the previous lemma, we need there to exist an appropriate open disc $B$.  The next lemma is the key ingredient for this. %

\begin{lemma}\label{lemma:d-disjoint-iterates} Let $\eta\in \Homeo_c(\D,\omega)$ with $\eta\neq\id$. Then, there exists a point $x\in \D$ such that $x,\eta(x), \eta^2(x), \dots, \eta^d(x)$ are pairwise distinct points.  
\end{lemma}

In particular, by the lemma, there exists an open topological disc $B$ whose closure is $d$-displaced by $\eta$. 

If we then let $\cal V$ be the $C^0$ open neighborhood of $\eta$ given by the set of all $f\in\HomeoS$
which $d$-displace the closure of the disc $B$, then by Lemma \ref{lemme:cd-under-displacement}, we have $c_d(\phi\circ f)=c_d(f)$ for all $\phi\in\DiffB$ and $f\in \cal V\cap\DiffS$.  Now it turns out that every map $g$ which is sufficiently $C^0$ close to $\id$ is close in Hofer distance to an element in $\DiffB$. This is the content of the next lemma.  

\begin{lemma}\label{lemma:small-hofer} Let $B$ be an open topological disc whose closure is included in  $\S^2\setminus\{p_-\}$. For all $\eps>0$, there exists $\delta>0$, such that for all $g\in\DiffS$ with $d_{C^0}(g,\id)<\delta$, there is $\phi\in\DiffB$ such that $\|\phi^{-1}g\|\leq\eps$.
\end{lemma}

We will prove Lemma \ref{lemma:small-hofer} at the end of Section \ref{sec:proof-of-lemmas}. Assuming this, we can achieve the proof of \eqref{eq:uniform-conitnuity2} and hence of Proposition \ref{prop:C0-continuity-elsewhere}, as we now explain.  

\begin{proof}[Proof of Proposition \ref{prop:C0-continuity-elsewhere}.] 
  Let $\eps>0$ and let $\delta>0$ be as provided by Lemma  \ref{lemma:small-hofer}. Let also $f\in \cal V\cap\DiffS$. Then, for all $g$ satisfying $d_{C^0}(g,\id)<\delta$, there exists $\phi\in \DiffB$ such that $\|\phi^{-1}g\|<\eps$. Thus, using Lemma~\ref{lemme:cd-under-displacement}, Hofer continuity \eqref{eq:hofer-bound} and the conjugation invariance of the Hofer norm, we have:
\[|c_d(gf)-c_d(f)|=|c_d(g f)-c_d(\phi f)|\leq d\|f^{-1}\phi^{-1}g f\|=d\|\phi^{-1}g\|<d\eps.\]
Since $\eps$ is arbitrary, this concludes the proof of \eqref{eq:uniform-conitnuity2} and of Proposition \ref{prop:C0-continuity-elsewhere},  modulo the proofs of the lemmas \ref{lemme:cd-under-displacement}, \ref{lemma:d-disjoint-iterates} and \ref{lemma:small-hofer}.
\end{proof}

\subsection{Proofs of Lemmas \ref{lemme:cd-under-displacement}, \ref{lemma:d-disjoint-iterates} and \ref{lemma:small-hofer}.}\label{sec:proof-of-lemmas}

We now provide the proofs of the lemmas stated in the preceding section. We start with the simplest one.   

\begin{proof}[Proof of Lemma \ref{lemma:d-disjoint-iterates}] Let $\eta\in \Homeo_c(\D,\omega)$ with $\eta\neq\id$ and let $d$ be a positive integer. It is known (see \cite{Constantin-Kolev}) that for any positive integer $N$, $\eta^N\neq \id$. Thus, there exists a point $x\in\D$ such that $\eta^{d!}(x)\neq x$. For such a point,  $x,\eta(x), \eta^2(x),\dots, \eta^d(x)$ are pairwise distinct. Indeed, otherwise, there would be integers $0\leq k<\ell\leq d$ such that $\eta^k(x)=\eta^\ell(x)$, and we would get $\eta^{\ell-k}(x)=x$, in contradiction with $\eta^{d!}(x)\neq x$ since $d!$ is evenly divided by $\ell - k$.  
\end{proof}

We now turn our attention to Lemma \ref{lemme:cd-under-displacement}. 

\begin{proof}[Proof of Lemma \ref{lemme:cd-under-displacement}] Let $F$ be a Hamiltonian supported in $S^+$ with  
$\varphi_F^1 = f$ and let $G$ be a Hamiltonian supported in $B$ with $\varphi_G^1=\phi$. We will prove that for all $s\in[0,1]$, $\spec_d(\varphi_G^s f)=\spec_d(f)$.  This implies, as in the proof of Lemma \ref{lemma:eps-shift}, that the map $s\mapsto c_d(\varphi_G^s f)$ is constant, hence the lemma.

Let $s\in [0,1]$.  We will first verify that the diffeomorphism $\varphi_G^s f$ admits the same $k$-periodic points as $ f$, for all $k\in\{1,\dots, d\}$.

For all $\ell\in\{1,\dots, d\}$, we have $B\cap  f^\ell(B)=\emptyset$ and $B\cap  f^{-\ell}(B)=\emptyset$. It follows that $\varphi_G^s( f^{\ell}(B)) =  f^\ell(B),$ 
 for all  $\ell\in\{-d,\dots, d\}$, hence
  \[(\varphi_G^s f)^k( f^{-\ell}(B)) =   f^{k-\ell}(B), \quad \forall k\in\{1,\dots, d\}, \forall\ell\in\{0,\dots, d\}. \] 
Since $ f^{-\ell}(B)\cap  f^{k-\ell}(B)=\emptyset$ for such $k, \ell$, this implies that $\varphi_G^s f$ has no $k$-periodic points with $1\leq k\leq d$ in $\bigcup_{\ell=0}^d f^{-\ell}(B)$.

We now fix  $k\in\{1,\dots, d\}$. If $x\notin \bigcup_{\ell=0}^d f^{-\ell}(B)$, then $(\varphi_G^s\circ f)^k(x)= f^k(x)$. As a consequence, the $k$-periodic points of $\varphi_G^s\circ f$ are exactly those of $ f$. 

We will now prove that the corresponding action values coincide as well. For that purpose it is convenient to use an isotopy generating the $(\varphi_G^s\circ f)^k$ obtained by concatenation of isotopies rather than composition. 
Namely, the map $(\varphi_G^s\circ f)^k$ is the time-1 map of the isotopy $\psi^t$ which at time $t\in [\frac{\ell}{2k},\frac{\ell+1}{2k}]$ for $\ell\in\{0,\dots, 2k-1\}$ is given by %
\[\psi^t=
  \begin{cases}
\varphi_F^{\rho(2kt-\ell)}\circ(\varphi_G^s\circ f)^{\frac{\ell}{2}},&\quadif \ell\text{ is even,}\\
\varphi_G^{s\rho(2kt-\ell)}\circ f\circ(\varphi_G^s\circ f)^{\frac{\ell-1}{2}}   ,&\quadif \ell\text{ is odd.} 
  \end{cases}
\]
Here, $\rho:[0,1]\to[0,1]$ is a non-decreasing smooth function which is equal to $0$ near $0$ and equal to $1$ near $1$. The role of the time-reparametrization $\rho$ is simply to make the isotopy smooth at the gluing times.

This isotopy is generated by the Hamiltonian $K$ given by the formula
   \[K_t(x)=
     \begin{cases}
       2k\rho'(2kt-\ell)F_{\rho(2kt-\ell)}(x),&\quadif \ell\text{ is even,}\\
       2ks\rho'(2kt-\ell)G_{s\rho(2kt-\ell)}(x),&\quadif \ell\text{ is odd,}
     \end{cases}
   \]
   for $\ell\in\{0,\dots, 2k-1\}$ and $t\in [\frac{\ell}{2k},\frac{\ell+1}{2k}]$.

   We will compute the spectrum of $\varphi_G^s\circ f$ with the help of this particular Hamiltonian. The action of a capped 1-periodic orbit $(z,u)$ of $K$, with $z=\varphi_K^t(x)$, is given by
\begin{align*}
  \cal A_K(z,u) &= \int_{\D^2}u^*\omega + \int_{0}^{1}K_{t}(\psi^t(x))\,dt \\
                & \int_{\D^2}u^*\omega +\sum_{\ell=0}^{2k-1}\int_{\frac{\ell}{2k}}^{\frac{\ell+1}{2k}}K_{t}(\psi^t(x))\,dt.
\end{align*}
After suitable change of variables, 
we obtain:
\begin{align*}
  \cal A_K(z,u) = \int_{\D^2}u^*\omega &+\sum_{j=0}^{k-1}\left( \int_{0}^{1}F_{t}(\varphi_F^{t}\circ  (\varphi_G^s\circ f)^j(x))\,dt \right.\\
                    &  +\left.\int_{0}^{1} sG_{st}(\varphi_G^{st}\circ f\circ (\varphi_G^s\circ f)^j(x))\,dt
                                     \right).
\end{align*}
Since we showed above that $\varphi^s_G \circ  f$ has no $k$-periodic points in 
$ f^{- 1}(B)$, we know that $  f \circ (\varphi_G^s \circ  f)^j(x)$ does not belong to $B$, hence to the support of $G$. It follows that the integrand for the third integral above has to vanish and the integrand for the second integral above can be simplified, so that we get
\begin{align*}
  \cal A_K(z,u)=\int_{\D^2}u^*\omega 
                 +\sum_{j=0}^{k-1}\int_{0}^{1} F_t(\varphi_F^t\circ f^j(x))\,dt.
\end{align*}
We see that this action does not depend on $s$. As a consequence, we get that $\spec_d(\varphi_G^s   f)=\spec_d( f)$ for all $s\in[0,1]$. 
\end{proof}

There remains to prove Lemma \ref{lemma:small-hofer}. Its proof will consist of two steps. First, by Lemma \ref{lemma:square-fragmentation} below, a diffeomorphism which is sufficiently $C^0$-close to identity can be appropriately fragmented into maps supported in balls of small area. Second, we observe that moving theses maps with small support into $B$ can be achieved with small Hofer norm; this is the content of Lemma \ref{lemma:moving-balls} below.

Before starting the proof of Lemma \ref{lemma:small-hofer}, let us state the first of these two lemmas.

\begin{lemma}[\cite{LSV}, Lemma 47]\label{lemma:square-fragmentation} 
Let $\omega_0$ denote the standard area form on $\R^2$. Let $m$ be a positive integer and $\rho$ a positive real number. For %
  $i=1,\dots, m$, denote by $U_i$ the open rectangle $(0,1)\times(\frac{i-1}{m},\frac {i}{m})$. Then, there exists $\delta>0$, such that for every $g\in\Diff_c((0,1)\times(0,1),\omega_0)$ with $d_{C^0}(g,\id)<\delta$, there exist $g_1\in\Diff_c(U_1,\omega_0)$, $\dots$, $g_m\in \Diff_c(U_m,\omega_0)$ and $\theta\in \Diff_c((0,1)\times(0,1),\omega_0)$ supported in a disjoint union of topological discs whose total area is less than $\rho$, such that $g=g_1\circ\dots\circ g_m\circ\theta$.
\end{lemma}

We can now give the promised proof.  

\begin{proof}[Proof of Lemma \ref{lemma:small-hofer}]
  Let $\eps>0$. We pick an integer $N$ satisfying $\frac1{2N}<\textrm{area}(B)$, $m$ a positive multiple of $N$ such that $2\frac{N+1}{m}<\eps$, and $\rho=\frac1{m}$.

  Let $\delta$ be as provided by Lemma \ref{lemma:square-fragmentation} and let $g\in\DiffS$ be such that $d_{C^0}(g,\id)<\delta$. The map $g$ admits a fragmentation into the form $g=g_1\circ\dots\circ g_m\circ\theta$, with all the $g_i$
supported in pairwise disjoint topological discs $U_i$ of area $\frac{1}{2m}$ and $\theta$ supported in a disjoint union of discs of total area less than $\frac1{2m}$ (here, the factor $\frac12$ comes from the fact that the northern hemisphere $S^+$ has area $\frac12$, while Lemma \ref{lemma:square-fragmentation} is stated on $(0,1)\times(0,1)$ which has area $1$).  %

For $j=1,\dots, N$, denote by $f_j$ the composition $f_j:=\prod_{i=0}^{\tfrac mN-1}g_{j+iN}$.  Also write $f_{N+1}:=\theta$, so that, noting that the $g_i$ commute, we have the following formula
\[g=\prod_{j=1}^{N+1}f_j.\]
Each $f_j$ for $j=1,\dots, N$ is supported in $V_j=\bigcup_{i=0}^{\tfrac mN-1}U_{j+iN}$ whose area is $\frac{1}{2N}<\textrm{area}(B)$. 
Note that $V_j$ is a disjoint union of discs, each 
of area $\frac1{2m}$. By assumption, the support of $f_{N+1}=\theta$, which we denote by $V_{N+1}$, is also included in a disjoint union of discs of total 
area smaller than $\frac1{2m}$.

Let us now state our second lemma, whose proof we postpone to the end of this section.

\begin{lemma}\label{lemma:moving-balls} Let $a\in(0,\frac12)$, let $B_1, \dots, B_k\subset S^+$ be pairwise disjoint open topological discs each 
of area smaller than $a$, and $B\subset S^+$ a topological disc with $\mathrm{area}(B)> ka$. Then, there exists a Hamiltonian diffeomorphism $h\in \Diff_{\S^2\setminus\{p_-\}}(\S^2,\omega)$ which maps $B_1\cup\dots\cup B_k$ into $B$ and satisfies $\|h\|\leq 2a$. 
\end{lemma}

In our situation, this lemma implies that for any $j=1, \dots, N+1$, there exists a Hamiltonian diffeomorphism $h_j\in \Diff_{\S^2\setminus\{p_-\}}(\S^2,\omega)$, such that $h_j(V_j)\subset B$ and $\|h_j\|\leq \frac1m$. 

Consider the diffeomorphism
  \[\phi=\prod_{j=1}^{N+1}h_j\circ f_j\circ h_j^{-1}.\]
  By construction, $\phi$ is supported in $B$. We will show that $\|\phi^{-1}g\|<\eps$, which will achieve the proof of Lemma \ref{lemma:small-hofer}.

  To prove that  $\|\phi^{-1}g\|<\eps$, let us introduce 
 $\tilde{g}_k=f_1\circ f_2\circ \dots \circ f_k$ and $\phi_k=(h_1\circ f_1\circ h_1^{-1})\circ (h_2\circ f_2\circ h_2^{-1})\circ\dots\circ (h_k\circ f_k\circ h_k^{-1})$, for $k=1,\dots, N+1$. In particular, $\tilde{g}_{N+1}=g$ and $\phi_{N+1}=\phi$. Then, for all $k=1,\dots, N$, we have
  \begin{align*}
    \phi_{k+1}^{-1}\circ \tilde{g}_{k+1}&= h_{k+1}\circ f_{k+1}^{-1}\circ h_{k+1}^{-1}\circ(\phi_{k}^{-1}\circ \tilde{g}_{k})\circ f_{k+1}\\
   &= h_{k+1}\circ (f_{k+1}^{-1}\circ h_{k+1}^{-1}\circ f_{k+1})\circ (f_{k+1}^{-1}\circ(\phi_{k}^{-1}\circ \tilde{g}_{k})\circ f_{k+1})
  \end{align*}
Thus, by the triangle inequality and the conjugation invariance,
  \begin{align*}
    \|\phi_{k+1}^{-1}\circ \tilde{g}_{k+1}\|&\leq \|h_{k+1}\|+\|f_{k+1}^{-1}\circ h_{k+1}^{-1}\circ f_{k+1}\|+ \|f_{k+1}^{-1}\circ(\phi_{k}^{-1}\circ \tilde{g}_{k})\circ f_{k+1}\|\\
                                    &\leq \|h_{k+1}\|+\|h_{k+1}^{-1}\|+ \|\phi_{k}^{-1}\circ \tilde{g}_{k}\|\\
    &\leq \frac2m+ \|\phi_{k}^{-1}\circ \tilde{g}_{k}\|.
  \end{align*}
Hence, by induction, $\| \phi_i^{-1} \circ \tilde{g}_i \| \le \frac{2i}{m},$ 
and so this yields, as wished:
 \[\|\phi^{-1}\circ g\|\leq 2\frac{N+1}m<\eps.\]
\end{proof}

The last remaining proof is now the following.

\begin{proof}[Proof of Lemma \ref{lemma:moving-balls}] 
Denote $U:=B_1\cup\dots\cup B_k$. Since $U$ has smaller area than $B$, there exists a Hamiltonian diffeomorphism $\psi\in \Diff_{S^+}(\S^2,\omega)$ such that $\psi(U)\subset B$.
The Hofer norm of $\psi$ may not be small, but we will replace $\psi$ with an appropriate commutator of $\psi$ whose Hofer norm will be controlled.

Since the displacement energy of $U$ is smaller than $a<\frac12$, there exists a Hamiltonian diffeomorphism $\ell\in\Diff_{\S^2\setminus\{p_-\}}(\S^2,\omega)$ such that $\ell(\overline{U})\cap \overline{U}=\emptyset$ and $\|\ell\|\leq a$. 

Since $\ell(U)$ has area smaller than $\frac12$, there exists $\chi\in \Diff_{\S^2\setminus\{p_-\}}(\S^2,\omega)$ 
which fixes $U$ and such that  $\chi(\ell(U))\cap S^+=\emptyset$, which in particular implies that $\chi(\ell(U))\cap \supp(\psi)=\emptyset$. 
Then, $\ell'=\chi\circ\ell\circ\chi^{-1}$ satisfies the following properties:
  \begin{enumerate}[(i)]
  \item  $\ell'(U)\cap U=\emptyset$, 
  \item   $\|\ell'\|\leq a$ , 
  \item   $\ell'(U)\cap \mathrm{supp}(\psi)=\emptyset$.
\end{enumerate}
To prove Property (i), start from $\ell(U)\cap U=\emptyset$ and compose with $\chi$, to get $\chi\circ \ell (U)\cap \chi(U)=\emptyset$. Since $U=\chi^{-1}(U)$, we obtain  $\chi\circ\ell\circ\chi^{-1}(U)\cap U=\emptyset$, hence Property (i). Property (ii) follows from the conjugation invariance of the Hofer norm. Property (iii) is a consequence of Property (i), since $\chi$ fixes $U.$ 

Now, set $h:=\psi\circ\ell'^{-1}\circ\psi^{-1}\circ\ell'$.
By Property (iii), $\ell'^{-1}\circ\psi^{-1}\circ\ell'(U)=U$. 
Thus, $h(U)=\psi(U)\subset B$. Moreover,
\[\|h\|\leq \|\psi\circ\ell'^{-1}\circ\psi^{-1}\|+\|\ell'\|=2\|\ell'\|\leq 2a.\]
This concludes our proof.
\end{proof}

\section{The periodic Floer homology of monotone twists}
\label{sec:PFH_monotone_twist} 
The goal of this section and the next is to prove Theorem~\ref{theo:PFHspec_Calabi_property} which establishes the Calabi property for monotone twist maps of the disc which were introduced in  Section \ref{sec:calabi_inf_twist} .  Recall from Remark \ref{rem:PFHspec-disc} that we define PFH spectral invariants for maps of the disc by identifying $\Diff_c(\D, \omega)$ with maps of the sphere supported in the northern hemisphere $S^+$.  In particular, we will view any monotone twist as an element of the set $\mathcal{S}$ appearing in Theorem \ref{thm:PFHspec_initial_properties}.

Theorem~\ref{theo:PFHspec_Calabi_property} will follow from the following theorem for the invariants $c_{d,k}$, which, as alluded to in the introduction, was originally conjectured in greater generality by Hutchings \cite{Hutchings_unpublished}.

\begin{theo}
\label{thm:s2case}
Let $(d,k)$ be any sequence, with $k = d$ mod $2$, and $d$ tending to infinity.  Then, for any positive monotone twist $\varphi$ we have:
\begin{equation}
\label{eqn:calabiconjecture}
\Cal(\varphi) = \lim_{d \to \infty} \left( \frac{ c_{d,k}(\varphi)}{d} - \frac{k}{2(d^2 + d)} \right).
\end{equation}
\end{theo}

A first observation, concerning Equation \eqref{eqn:calabiconjecture}, 
which is also due to Hutchings, is that it suffices to establish \eqref{eqn:calabiconjecture} for a single such sequence $(d,k_d)$ with $d = 1, 2, \ldots$ ranging over the positive integers.  Indeed, %
for $d$-nondegenerate  $\varphi$, there is an automorphism of the twisted PFH chain complex given by
\[ (\alpha,Z) \mapsto (\alpha, Z + [S^2]),\]
where $[S^2]$ denotes the class of the sphere.  This increases the grading by $2d + 2$, by Formula \eqref{eqn:index_ambiguity}.  It also increases the action by $1$.  So, we have
\begin{equation}
\label{eqn:auto}
c_{d,k + 2d + 2} = c_{d,k} + 1
\end{equation}
for all $\varphi$.  Now, the right hand side of Equation \eqref{eqn:calabiconjecture} is invariant under increasing the numerator of the first fraction by one, and increasing the numerator of the  second fraction by $2d + 2$.  Moreover, as a corollary of Theorem \ref{theo:spec_computation}, we will obtain
\begin{equation}
\label{eqn:monotonicink}
c_{d,k} \leq c_{d,k'},
\end{equation}
when $k' \geq k$, with $k=k' = d$ mod $2$ and $\varphi$ a positive monotone
twist\footnote{For more general $\varphi$, \eqref{eqn:monotonicink} can still be established, by using the PFH ``$U$-map", but we will not need this in the present work.}; 
see Remark  \ref{rem:monotonicink}.  Thus, given an arbitrary sequence $(\tilde{d},\tilde{k})$, we can assume by the above analysis that $\tilde{k}$ is within $2d+2$ of $k_{\tilde{d}}$, and $| c_{\tilde{d},\tilde{k}} - c_{\tilde{d},k_{\tilde{d}}}| \leq 1$; the limit on the right hand side of \eqref{eqn:calabiconjecture} is then the same for $c_{\tilde{d},k_{\tilde{d}}}$ and  $c_{\tilde{d},\tilde{k}}$.     

\medskip

The remainder of Section  \ref{sec:PFH_monotone_twist} is dedicated to describing a combinatorial model of  $\widetilde{PFH}(\varphi^1_H, d)$ where the Hamiltonian $H$ belongs to  a certain class of Hamiltonians $\mathcal{D}$ with the following property:      The set $\mathcal{D}$ is contained in $\mathcal{H}$, which we introduced in Section  \ref{sec:PFH_spec_initial_properties}, and  any Hamiltonian in $\mathcal{H}$ whose time-1 map is a monotone twist can be approximated, in the Hofer norm $\Vert \cdot \Vert_{(1, \infty)}$, arbitrarily well by Hamiltonians in $\mathcal D$.    We will use  this combinatorial model in Section \ref{sec:calabi_for_montone_twists} to prove Theorem \ref{thm:s2case}.

The main result of Section \ref{sec:PFH_monotone_twist} is Proposition \ref{prop:model} which describes this combinatorial model.  

While the current section is the longest section of our paper, it might help the reader to note that our model is inspired by an extensive literature.  In particular, our combinatorial model is inspired by similar combinatorial models that have been developed  for computing the PFH of a Dehn twist \cite{Hutchings-Sullivan-Dehntwist},  the ECH of $T^3$ \cite{Hutchings-Sullivan-T3}, and the ECH of contact toric manifolds \cite{Choi}, and our methods in this section are inspired by the techniques used to establish these combinatorial models; we should also mention the appendix of \cite{Hutchings_beyond}, which was very influential to our thinking, and which we say more about later in this section.  The index calculations in \cite{C-CG-F-H-R} were also useful.   For the benefit of the reader, we will frequently point out the parallels to all this literature more explicitly below.

\subsection{Perturbations of rotation invariant Hamiltonian flows}
\label{sec:ham}

Here, and through the end of Section \ref{sec:PFH_monotone_twist}, we consider Hamiltonian flows on $(\S^2, \omega=\frac1{4\pi}d\theta\wedge dz)$, for an autonomous Hamiltonian
\[ H = \frac{1}{2} h(z),\]
where $h$ is some function of $z$.  We have 
\begin{equation}
\label{eqn:Hvf} 
X_H = 2 \pi h'(z) \partial_{\theta}.
\end{equation}
Hence,
\[ \varphi^1_H(\theta,z) = (\theta + 2 \pi h'(z), z) \]
We further restrict $h$ to satisfy 
\[ h' > 0, h'' > 0, h(-1) = 0.\]

Furthermore, we  demand that $h'(-1), h'(1)$ are irrational numbers satisfying $h'(-1) \leq \frac{\varepsilon_0}{d}$ and $\lceil h'(1) \rceil - h'(1) \leq \frac{\varepsilon_0}{d}$, where $\varepsilon_0$ is a small positive number and $\lceil \cdot \rceil$ denotes the ceiling function.    Let $\mathcal{D}$ denote the set of Hamiltonians $H$ that satisfy all of these conditions and observe that $\mathcal D \subset \mathcal H$ where $\mathcal H$ was defined in Section \ref{sec:PFH_spec_initial_properties}.  As a consequence of Theorem \ref{thm:PFHspec_initial_properties},  we have well-defined PFH spectral invariants $c_{d,k}(\varphi^1_H)$ for all $H \in \mathcal{D}$.

\medskip

The periodic orbits of $\varphi^1_H$ are then as follows:
\begin{enumerate}
\item There are elliptic orbits $p_+$ and $p_-$, corresponding to the north and south poles, respectively.
\item For each $p/q$ in lowest terms such that $h' = p/q$ is rational, there is a circle of periodic orbits, all of which have period $q$.
\end{enumerate}

These circles of periodic orbits are familiar from Morse-Bott theory, and are sometimes referred to as ``Morse-Bott circles".  There is also a standard $\varphi^1_H$-admissible\footnote{This means that the almost complex structure is compatible with the standard SHS on the mapping torus for $\varphi^1_H$.} %
almost complex structure $J_{std}$ respecting this symmetry; its action on $\xi = T(\S^2 \times \lbrace pt \rbrace) = T\S^2$ is given by the standard almost complex structure on $\S^2$.  

As is familiar in this context (see \cite[Section 3.1]{Hutchings-Sullivan-Dehntwist}), we can perform a $C^2$-small perturbation of  $H$, to split such a circle corresponding to the locus where $h' = p/q$ into one elliptic and one hyperbolic periodic points, such that the elliptic one $e_{p,q}$ has slightly negative monodromy angle, and the eigenvalues for the hyperbolic one $h_{p,q}$ are positive.     Furthermore, the $C^2$-small perturbation can be taken to be supported in an arbitrarily small neighborhood of the circle where $h' = p/q$.  More precisely, given a $\varphi^1_H$ such as above, for any positive $d$  and arbitrarily small $\eps > 0$, we can find another area-preserving diffeomorphism $\varphi_0$ of $\S^2$, which we call a {\bf nice perturbation} of $\varphi^1_H$, such that:

\begin{enumerate}
\item The only periodic points of $\varphi_0$ which are of degree at most $d$ are  $p_+, p_-$, and  the orbits $e_{p,q}$ and $h_{p,q}$ from above such that $q \leq d.$  Furthermore, all of these orbits are non-degenerate.
\item The eigenvalues of the linearized return map for $e_{p,q}$ are within $\eps$ of $1$.
\item   $\varphi_0(\theta,z) = \varphi^1_H(\theta,z)$ as long as $z$ is not within $\eps$ of a value such that $h' = p/q$ where  $ q \leq d$.
\end{enumerate}
  Observe that we can pick a time-dependent Hamiltonian $\tilde H $ such that $\varphi^1_{\tilde H} = \varphi_0$ and $\tilde H = H$ as long as $z$ is not within $\eps$ of a value such that $h' = p/q$ with  $q \leq d$.

It is also familiar from the work of Hutchings-Sullivan, see \cite{Hutchings-Sullivan-Dehntwist}, Lemma A.1, that we can choose our perturbation of $\varphi^1_H$ such that we can assume the following: 

\begin{enumerate}
\item[4.] $\varphi_0$ is chosen so that ``Double Rounding", defined in Section \ref{sec:comb} below, can not occur for generic $J$ close to $J_{std}$. 
\end{enumerate}
Later, it will be clear why it simplifies the analysis to rule out Double Rounding.  

\subsection{The combinatorial model}
\label{sec:comb}

We now aim to describe the promised combinatorial model of  $\widetilde{PFH}$ for the Hamiltonians described in the previous section. Fix $d \in \N$ and $\varphi^1_H$, where $H \in \mathcal{D}$, for the remainder of Section \ref{sec:PFH_monotone_twist}.   

To begin, define a {\bf concave lattice path} to be a piecewise linear, continuous path $P$, in the $xy$-plane, such that $P$ starts and ends on integer lattice points,  its starting point is on the $y$--axis, the nonsmooth points of $P$ are also at integer lattice points, and $P$ is concave, in the sense that it always lies above any of the tangent lines at its smooth points. Lastly, every edge of $P$ is labeled by either $e$ or $h$.  Below, we will associate one such lattice path to every $\widetilde{PFH}$ generator $(\alpha, Z)$.

Let $\alpha = \{(\alpha_i, m_i)\}$ be an orbit set of degree $d$, for a  
nice perturbation  $\varphi_0$ of $\varphi^1_H$.
First of all, recall that the simple Reeb orbits for $Y_{\varphi_0}$, with degree no more than $d$, are as follows :

\begin{enumerate}
\item The  Reeb orbits $\gamma_{\pm}$ corresponding to $p_{\pm}$.
\item For each $z$ such that $h'(z) = p/q$ in lowest terms, where $q \leq d$,  there are Reeb orbits of degree $q$ corresponding to the periodic points $e_{p,q}$ and $h_{p,q}$, that we will also denote by $e_{p,q}$ and $h_{p,q}$.
\end{enumerate}

 \begin{figure}[h!]
 \centering 
 \def\svgwidth{1.0\textwidth}
 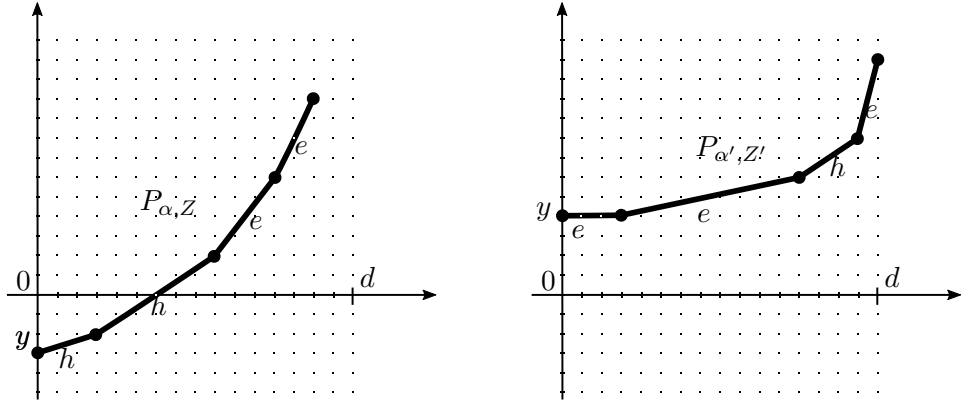
 \caption{The lattice path $P_{\alpha,Z}$ for $\alpha=\{(h_{1/3},1),(e_{2/3},1),(h_{2/3},1),(e_{4/3},1),(e_2,2)\}$, $Z=Z_\alpha-3[\S^2]$, and $P_{\alpha',Z'}$ for $\alpha'=\{(\gamma_-,3),(e_{2/9},1),(h_{2/3},1),(\gamma_+,1)\}$, $Z'=Z_\alpha+4[\S^2]$
 (assuming $\lceil h'(1)\rceil=4$).}\label{fig:lattice-path1}
\end{figure}

We will now associate to the orbit set $\alpha  = \{(\alpha_i, m_i)\}$ a concave lattice path $P_\alpha$ whose starting point we require to be $(0,0)$.  
If $(\gamma_-, m_-) \in \alpha$, we set $v_- = m_- (1,0)$ and label it by $e$.  If $(\gamma_+, m_+) \in \alpha$ we set $v_+ = m_+ (1, \lceil h'(1) \rceil )$ and label it by $e$.  
Next, consider the orbits in $\alpha$ corresponding to $z=z_{p,q}$ such that $h'(z) = p/q$; note that there are at most two such entries in $\alpha$: one corresponding to $e_{p/q}$ and another corresponding to $h_{p/q}$.  To these entries, we associate the labeled vector $v_{p,q} = m_{p,q} (q,p)$, where $m_{p,q}$ is the sum of multiplicities of $e_{p/q}$ and $h_{p/q}$; the vector is labeled $h$ if $(h_{p/q},1) \in \alpha$, and $e$ otherwise. (For motivation, note that by the conditions on the PFH chain complex, an $m_i$ corresponding to a hyperbolic orbit must equal $1$.)  To build the concave lattice path $P_{\alpha}$ from all of the data in $\alpha$, we simply concatenate the vectors $v_-, v_{p,q}, v_+$ into a concave lattice path. Note that there is a unique way to do this: the path must begin with $v_-$, it must end with $v_+$, and  the vectors $v_{p/q}$ must be concatenated in  increasing order with respect to the ratios $p/q$.

 Now, given a chain complex generator $(\alpha, Z)$ for $\widetilde{PFC}$, we define an assignation
\[ (\alpha,Z) \mapsto P_{\alpha,Z}\]
which associates a  concave lattice path $P_{\alpha,Z}$ to the generator $(\alpha,Z)$.   More specifically, when $Z = Z_\alpha$, where $Z_{\alpha}$ is the class from Lemma~\ref{lemma:action-action}, we define $P_{\alpha, Z_{\alpha}}$ to be the concave lattice path $P_\alpha$.   Since $H_2(Y_{\varphi}) = \mathbb{Z}$, generated by the class of $\S^2 \times \lbrace \text{pt} \rbrace$, for any other $(\alpha,Z)$, we have $Z = Z_{\alpha} + y [\S^2]$.  We then define $P_{\alpha,Z}$ to be $P_{\alpha}$ shifted by the vector $(0,y)$. Figure \ref{fig:lattice-path1} shows two examples of such concave lattice paths.

\bigskip

We now state some of the key properties of the above assignation  \[ (\alpha,Z) \mapsto P_{\alpha,Z}.\]

\noindent {\bf Degree: }  Note that since we are fixing the degree of $\alpha$ to be $d$, the horizontal displacement of $P_{\alpha, Z}$ must be $d$; we therefore call the horizontal displacement of a concave lattice path its {\bf degree}.  Clearly, the degree of $P_{\alpha, Z}$ agrees with the degree of the $\widetilde{PFH}$ generator $(\alpha, Z)$.

\medskip

\noindent {\bf Action:}
Define the {\bf action} $\mathcal{A}(P_{\alpha,Z})$ as follows.   We first define the actions of the edges of $P_{\alpha,Z}$ by the  formulae:
\begin{gather*}
\mathcal{A}(v_-) = 0, \; \mathcal{A}(v_+) = m_+\frac{h(1)}2, \\
\mathcal{A}(v_{p,q}) = \frac{m_{p,q}}{2} ( p (1-z_{p,q}) + q h(z_{p,q})),
\end{gather*}
where $v_-, v_+$, and $v_{p,q}$ are as above.   %
We then define the action of  $P_{\alpha,z}$ to be 
\begin{equation} \label{eq:action_formula}
\mathcal{A}(P_{\alpha,Z}) = y   + \mathcal{A}(v_+) + \sum_{v_{p,q}}   \mathcal{A}(v_{p,q}),
\end{equation}
where $y$ is such that $P(\alpha,Z)$ begins at $(0,y)$. 

We claim that by picking the nice perturbation $\varphi_0$ to be sufficiently close to $\varphi^1_H$ we can arrange for $\mathcal{A}(\alpha, Z)$ to be as close to $\mathcal{A}(P_{\alpha, Z})$ as we wish.  To show this it is sufficient to prove it when $\alpha$ is a simple Reeb orbit and $Z=Z_\alpha$, where $Z_\alpha$ is the relative class constructed in the proof of Lemma \ref{lemma:action-action}.  We have to consider the following three cases:
\begin{itemize}
\item If $\alpha = \gamma_-$, then $\mathcal{A}(\alpha, Z_\alpha) = 0$, by Equation \eqref{eqn:action_Z_alpha}, which coincides with $\mathcal{A}(1,0)$.
\item  If $\alpha = \gamma_+$, then $\mathcal{A}(\alpha, Z_\alpha) = \frac{h(1)}{2}$, by Equation \eqref{eqn:action_Z_alpha}, which coincides with $\mathcal{A}(1, \lceil h'(1) \rceil)$.  Note that, in Equation \eqref{eqn:action_Z_alpha},  the term $\int_{\D^2} u_\alpha^* \omega$ is zero. 
\item The remaining case is when  $\alpha = e_{p,q}$ or $h_{p,q}$; here, it is sufficient to show that the action of the Reeb orbits at $z_{p,q}$,  for the unperturbed diffeomorphism $\varphi^1_H,$ is exactly the quantity  $\frac{1}{2} \left( p (1-z_{p,q}) + q h(z_{p,q})\right)$.  This follows from Equation \eqref{eqn:action_Z_alpha_2}:  the term $\int_{\D^2} u_\alpha^* \omega$ is exactly $\frac{1}{2} p(1 - z_{p,q})$ and the term $\int_0^q H_t(\varphi^t_H(q))dt$ is exactly $\frac{1}{2} q h(z_{p,q})$.
\end{itemize}

\medskip

\noindent {\bf Index:} 
Next, we associate an {\bf index}  to a concave lattice path $P$ which begins at a point $(0, y)$, on the $y$--axis, and has degree $d$.

First, we form (possibly empty) regions $R_{\pm}$, where $R_-$ is the  closed region bounded by the $x$-axis, the $y$-axis, and the part of $P$ below the $x$-axis, while $R_+$ is the  closed region bounded by the $x$-axis, the line $x = d$, and the part of $P$ above the $x$-axis. 
Let $j_+$ denote the number of lattice points in the region $R_+$, not including lattice points on $P$, and let $j_-$ denote the number of lattice points in the region $R_-$, not including the lattice points on the $x$--axis; see Figure \ref{fig:combinatorial-index} below.  We now define
\begin{equation}\label{eqn:index_path_unlabelled}
 j(P) := j_+(P) - j_-(P).
\end{equation}
 This definition of $j$ is such that if one shifts $P$ vertically by $1$, then $j(P)$ increases by $d+1$ %

Given a  path $P_{\alpha,Z}$, associated to a $\widetilde{PFH}$ generator $(\alpha,Z)$, we define its index by 
\begin{equation}\label{eqn:index_path_labelled}
I(P_{\alpha,Z}) := 2 j(P_{\alpha,Z}) -d + h,
\end{equation}
where $h$ denotes the number of edges in $P_{\alpha, Z}$ labeled by $h$. See Figure \ref{fig:combinatorial-index} for an example of computation of this combinatorial index. We will show in Section \ref{sec:index} that $I(P_{\alpha,Z})$ coincides with the PFH index of $I(\alpha,Z)$ as defined in Equation \eqref{eqn:twistedgrading}.

 \begin{figure}[h!]
 \centering 
 \def\svgwidth{0.5\textwidth} 
 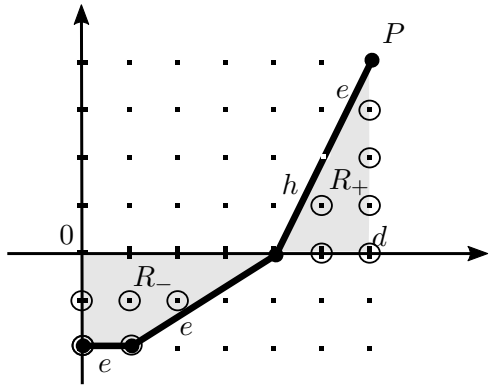
 \caption{The lattice points included in the count of $j(P)$ are circled. On this example, $j_+(P)=6$, $j_-(P)=5$, $d=6$, $h=1$, thus $j(P)=1$ and $I(P)=-3$.}\label{fig:combinatorial-index}
\end{figure}

\medskip
\noindent {\bf Corner rounding and the differential:} 
Lastly, we define a combinatorial process  which corresponds to the PFH differential.  Let $P_{\beta}$ be a concave lattice path of degree $d$ which begins on the $y$--axis.
  Then, if we attach vertical rays to the beginning and end of $P_{\beta}$, in the positive $y$ direction, we obtain a closed convex subset $R_{\beta}$ of the plane; see Figure \ref{fig:rounding-corner1}.  For any given corner of $P_{\beta}$, where we include the initial and final endpoints of $P_\beta$ as corners, we can define a {\bf corner rounding} operation by removing this corner, taking the convex hull of the remaining integer lattice points in $R_{\beta}$, and taking the lower boundary of this region, namely the part of the boundary that does not consist of vertical lines.   Note that the newly obtained path is of degree $d$.

 \begin{figure}[h!]
 \centering 
 \def\svgwidth{0.4\textwidth} 
\begingroup%
  \makeatletter%
  \providecommand\color[2][]{%
    \errmessage{(Inkscape) Color is used for the text in Inkscape, but the package 'color.sty' is not loaded}%
    \renewcommand\color[2][]{}%
  }%
  \providecommand\transparent[1]{%
    \errmessage{(Inkscape) Transparency is used (non-zero) for the text in Inkscape, but the package 'transparent.sty' is not loaded}%
    \renewcommand\transparent[1]{}%
  }%
  \providecommand\rotatebox[2]{#2}%
  \newcommand*\fsize{\dimexpr\f@size pt\relax}%
  \newcommand*\lineheight[1]{\fontsize{\fsize}{#1\fsize}\selectfont}%
  \ifx\svgwidth\undefined%
    \setlength{\unitlength}{205.00198445bp}%
    \ifx\svgscale\undefined%
      \relax%
    \else%
      \setlength{\unitlength}{\unitlength * \real{\svgscale}}%
    \fi%
  \else%
    \setlength{\unitlength}{\svgwidth}%
  \fi%
  \global\let\svgwidth\undefined%
  \global\let\svgscale\undefined%
  \makeatother%
  \begin{picture}(1,0.83884396)%
    \lineheight{1}%
    \setlength\tabcolsep{0pt}%
    \put(0,0){\includegraphics[width=\unitlength,page=1]{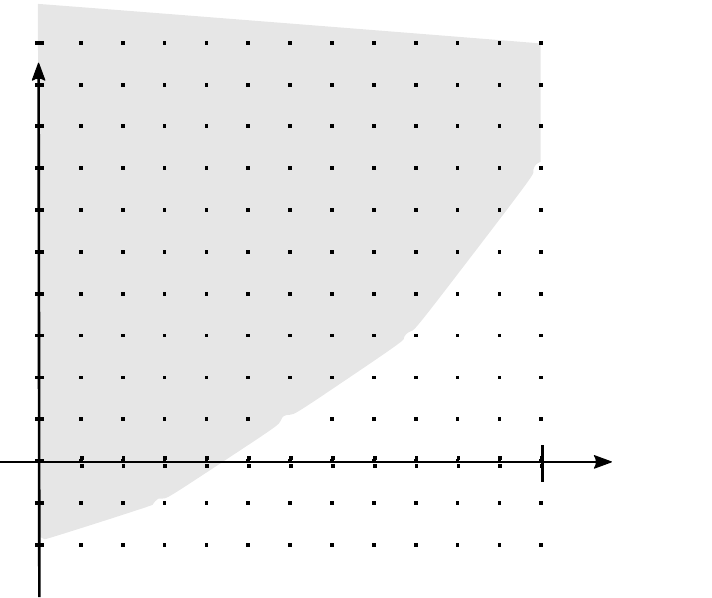}}%
    \put(0.64673144,0.39598172){\color[rgb]{0,0,0}\makebox(0,0)[lt]{\lineheight{1.25}\smash{\begin{tabular}[t]{l}$P_\beta$\end{tabular}}}}%
    \put(0.34301642,0.56397391){\color[rgb]{0,0,0}\makebox(0,0)[lt]{\lineheight{1.25}\smash{\begin{tabular}[t]{l}$R_\beta$\end{tabular}}}}%
    \put(0,0){\includegraphics[width=\unitlength,page=2]{rounding-corner1.pdf}}%
  \end{picture}%
\endgroup%

 \caption{The region $R_\beta$.}\label{fig:rounding-corner1}
\end{figure}

We now say that another concave lattice path $P_{\alpha}$ is obtained from $P_{\beta}$ by {\bf rounding a corner and locally losing one $h$}, if $P_{\alpha}$ is obtained from $P_{\beta}$ by a corner rounding such that the following conditions are satisfied; see the examples in Figure \ref{fig:rounding-corner2}:
\begin{enumerate}[(i)]
\item Let $k$ denote the number of edges in $P_\beta$,  with an endpoint at the rounded corner, which are labeled $h$.  We require that $k >0$; so $k=1$ or $k=2$. 

\item    Of the new edges in $P_\alpha$, created by the corner rounding operation, exactly $k-1$ are labelled $h$. %
\end{enumerate}

 \begin{figure}[h!]
 \centering 
 \def\svgwidth{0.8\textwidth} 
 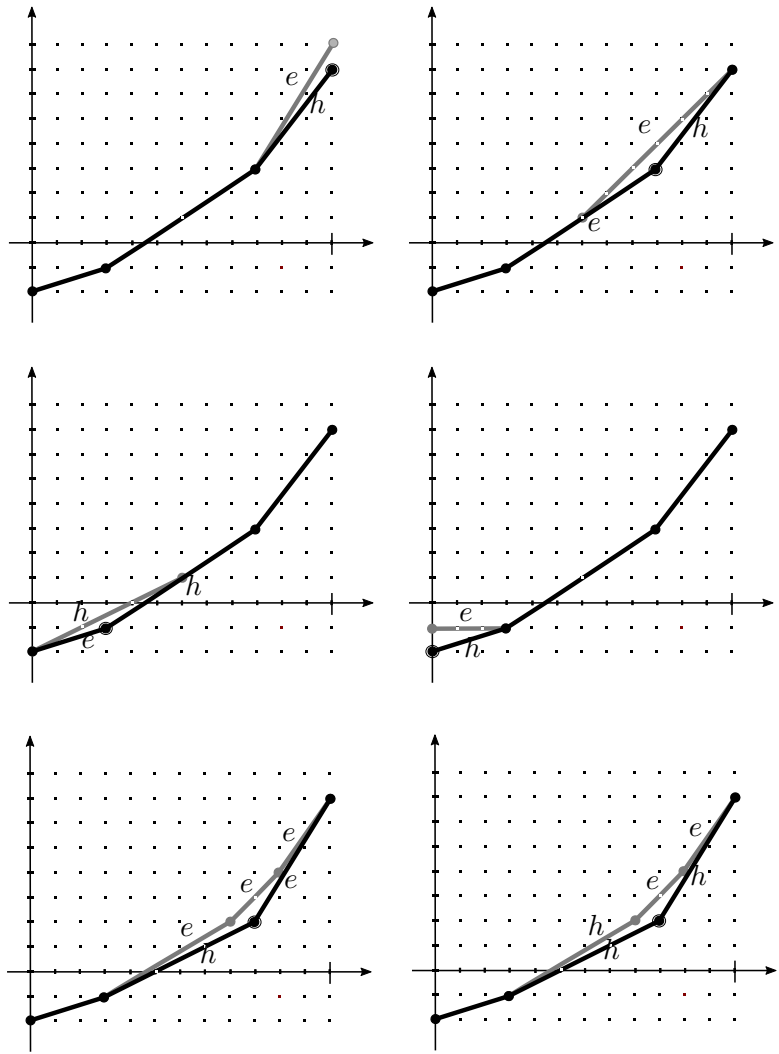
 \caption{Some examples for the ``rounding corner and locally losing one $h$'' operation. The path $P_\beta$ is in black, the new edges in $P_\alpha$ are in grey. The rounded corner is circled. We only label the relevant edges.}\label{fig:rounding-corner2}
\end{figure}
 
The notions introduced above and the proposition below give a complete combinatorial interpretation of the PFH chain complex:

\begin{prop}
\label{prop:model}
Fix $d > 0$ and let $\varphi_0$ be a nice perturbation of $\varphi^1_H$, where $H \in \mathcal{D}$.  Then, for generic $\varphi_0$-admissible almost complex structure $J$ close to $J_{std}$, there is an assignation 
\[ (\alpha,Z) \mapsto P_{\alpha,Z} \]
with the following properties:	
\begin{enumerate}
\item $\mathcal{A}(\alpha,Z) \sim \mathcal{A}(P_{\alpha,Z})$,
\item $I(\alpha,Z) = I(P_{\alpha,Z})$,
\item $\langle \partial (\alpha, Z) , (\beta, Z') \rangle \ne 0$ if and only if $P_{\alpha, Z}$ is obtained from $P_{\beta, Z'}$ by rounding a corner and locally losing one $h.$ 
\end{enumerate}

Here, by $\mathcal{A}(\alpha,Z) \sim \mathcal{A}(P_{\alpha,Z})$, we mean that by choosing our nice perturbation  $\varphi_0$ sufficiently close to $\varphi^1_H$, we can arrange for $\mathcal{A}(\alpha,Z)$ to be as close to $\mathcal{A}(P_{\alpha,Z})$ as we wish.
\end{prop}

We have already proven the first of the three listed properties in the above proposition.  The second will be proven below in Section \ref{sec:index}.  The proof of the third takes up the remainder of Section \ref{sec:index}.  We have already mentioned that the proposition is inspired by similar combinatorial models, but it might help the reader to note that, more specifically, the second property most closely corresponds to the index calculations in \cite[Section 3.2]{C-CG-F-H-R}, while the third property is inspired\footnote{In particular, the idea of attaching vertical rays in our definition of corner rounding is inspired by \cite[Definition A.2]{Hutchings_beyond}.} by \cite[Conjecture A.3]{Hutchings_beyond}, which was later proven in \cite{Choi}, and which we return to below.  %

\bigskip

Finally, we end this section by defining the {\bf Double Rounding} operation which will appear in the following sections and which has already been  introduced in Section \ref{sec:ham}.  Namely, if $P_{\beta,Z'}$ has three consecutive edges, all labeled by $h$, we say that $P_{\alpha,Z}$ is obtained from $P_{\beta,Z'}$ by {\bf double rounding}  if we remove both interior lattice points for these three edges, take the convex hull of the remaining lattice points (in the region formed by adding vertical lines, as above), and label all new edges by $e$.

\subsection{Computation of the index}
\label{sec:index}
In this section, we prove the second item in Proposition~\ref{prop:model}.  Before giving the proof, we first summarize the definitions of the various terms of the PFH grading as defined in Equation \eqref{eqn:twistedgrading}, which we recall here:
\begin{equation*}
I(\alpha,Z) = c_{\tau} (Z) + Q_{\tau}(Z) + \sum_i \sum^{m_i}_{k  = 1} CZ_{\tau}(\alpha^k_i).
\end{equation*}
  These definitions can be found explained in detail, for example, in \cite[Sec. 2]{Hutchings-index}.

To define the {\bf relative Chern class}, $c_{\tau}(Z)$, we first take a surface $S$, representing $Z$ in $[-1, 1] \times \S^2 \times \S^1$, assumed transverse to the boundary  $\{-1,1\} \times \mathbb{S}^2 \times \mathbb{S}^1$ and embedded in $(-1,1) \times \mathbb{S}^2 \times \mathbb{S}^1$. 
We then define $c_{\tau}(Z)$ to be a signed count of zeros of a generic section $\Psi$ of $\xi|_S$, where recall that $\xi$ is the vertical tangent bundle of the fibration $[-1,1] \times \S^2 \times \S^1\to [-1,1]\times\S^1$, 
such that the restriction of $\Psi$ to $\partial S$ agrees with the trivialization $\tau$.  We similarly define the {\bf relative intersection number} $Q_{\tau}(Z)$ by the formula
\begin{equation}
\label{eqn:intersection}
Q_{\tau}(Z) : = c_1(N,\tau) - w_{\tau}(S),
\end{equation}
where $c_1(N,\tau)$, the {\bf relative Chern number of the normal bundle}, is a signed count of zeros of a generic section of $N|_S$, such that the restriction of this section to $\partial S$ %
agrees with $\tau$; note that the normal bundle $N$ can be canonically identified with $\xi$ along $\partial S$.  Meanwhile, the term $w_{\tau}(S)$, the {\bf asymptotic writhe}, is defined by using the trivialization $\tau$ to identify a neighborhood of each Reeb orbit with $\S^1 \times D^2 \subset \mathbb{R}^3$, and then computing the writhe\footnote{This is defined by identifying $\S^1 \times D^2$ with the product of an annulus and interval, projecting to the annulus, and counting crossings with signs.} of a constant $s$ slice  of $S$ near the boundary using this identification.  

Finally, to define the {\bf Conley-Zehnder} index, we first clarify the definitions of elliptic and hyperbolic Reeb orbits, and define the {\bf rotation number} $\theta$ for a simple orbit, relative to the trivialization $\tau$.  Specifically, the elliptic case is characterized by the property that the linearized return map has eigenvalues on the unit circle;  in this case, one can homotope the trivialization so that the linearized flow at time $t$ with respect to the trivialization is always a rotation by angle $2 \pi \theta_t$, for a continuous function $\theta_t$, and then the rotation number is the change in $\theta_t$  as one goes around the orbit once.  In the hyperbolic case, the linearized return map has real eigenvalues, and the linearized return map rotates by angle $2\pi k$ for some half-integer $k\in\tfrac12\Z$ as one goes around the orbit; the integer $k$ is the rotation number in this case.  In either case, denoting by $\theta$ the rotation number, for any cover of $\gamma$, we have
\begin{equation}
\label{eqn:cz}
CZ_{\tau}(\gamma^n) := \lfloor n \theta \rfloor + \lceil n \theta  \rceil,
\end{equation}

\begin{proof} [Proof of the second item in Proposition \ref{prop:model}]
  By the index ambiguity formula, Equation (\ref{eqn:index_ambiguity}), we have
  \begin{equation}\label{eq:index-ambiguity-here}
I(\alpha, Z+a[\S^2])=I(\alpha, Z)+a(2d+2).
\end{equation}
 Therefore, we only have to compute the index $I(\alpha, Z'_{\alpha})$ for a given relative class $Z'_\alpha\in H_2(\S^2\times\S^1,\alpha,d\gamma_-)$. 
  Let us now define the relative class $Z'_\alpha$ we will be using. Write $\alpha=\{(\gamma_-,m_-)\}\cup\{(\alpha_i,m_i)\}\cup\{(\gamma_+, m_+)\}$, where each $(\alpha_i, m_i)$ is either an $(h_{p_i/q_i}, 1)$ or an $(e_{p_i/q_i},m_{p_i/q_i})$. We define $Z'_\alpha=m_-Z'_-+m_+ Z'_++\sum_im_iZ'_{\alpha_i}$, where
  \begin{itemize}
  \item  $Z_-'\in H_2(\S^2\times\S^1,\gamma_-,\gamma_-)$ is the trivial class,
  \item  $Z_+'\in H_2(\S^2\times\S^1,\gamma_+,\gamma_-)$ is represented by the map 
    \[S_+:[0,1]\times[0,q]\to \S^2\times\S^1, \quad (s,t)\mapsto (R_{t \lceil h'(1) \rceil}(\eta(s)),t),\]
    where $\eta$ is a meridian from the South pole $p_-$ to the north pole $p_+$,  and $R_{t \kappa}$ denotes the rotation on $\S^2$ by the angle $2\pi t \kappa$.
  \item for $\alpha_i=e_{p,q}$ or $h_{p,q}$, the relative class $Z'_{\alpha_i}\in H_2(\S^2\times\S^1,\alpha_i,q\gamma_-)$  is represented by the map \[S_{\alpha_i}:[0,1]\times[0,q]\to \S^2\times\S^1, \quad (s,t)\mapsto (R_{t \frac{p}{q}}(\eta(s)),t),\]
where $\eta$ is a portion of the great circle which begins at $p_-$ and ends at $z_{\frac{p}{q}}$.
\end{itemize}

The class $Z_\alpha$ from \ref{sec:comb} is related to the class $Z'_\alpha$ as follows. We have $Z_-=Z'_-$, $Z_+=Z'_++\lceil h'(1)\rceil\,[\S^2]$ and for $\alpha_i=e_{p,q}$ or $h_{p,q}$, then $Z_{\alpha_i}=Z'_{\alpha_i}+p[\S^2]$. If we denote by $(0,y_\alpha)$ and $(d,w_\alpha)$ the endpoints of $P_{\alpha, Z}$, we thus obtain $Z_\alpha=Z'_\alpha+ (w_\alpha-y_\alpha)[\S^2]$. Using (\ref{eq:index-ambiguity-here}), this yields
\begin{equation}
  \label{eq:Zalpha-Zalpha_prime}
  I(\alpha, Z_{\alpha})=I(\alpha, Z'_{\alpha})+(w_\alpha-y_\alpha)(2d+2).
\end{equation}

We will now compute the index  $I(\alpha, Z'_\alpha)$. For that purpose, we first need to make choices of trivializations along periodic orbits.

Along the orbit $\gamma_-$,  the trivialization is given by any frame of $T_{p_-}\S^2$ independent of $t$.

Along $\gamma_+$, we take a frame which rotates positively with rotation number $\lceil h'(1)\rceil$.
Along other orbits, 
we use as trivializing frame $(\partial_\theta, \partial_z)\in T\S^2$.

We denote by $\tau$ these choices of trivialization.

Recall that we are also assuming for simplicity that $h'(-1)$ is arbitrarily close to 0 and $h'(1)$ is arbitrarily close (but not equal) to its ceiling $\lceil h'(1)\rceil$.

In order to compute the grading, we now have to compute the Conley-Zehnder index, the relative Chern class, and the relative self-intersection; we then have to put this all together to give the stated interpretation in terms of a count of lattice points.  

\medskip
\noindent{\em Step 1: The Conley-Zehnder index.}

We begin by computing the Conley-Zehnder index of each orbit, relative to the trivialization above.
\begin{enumerate}
\item The north pole orbit $\gamma_+$ is elliptic with rotation number $h'(1)-\lceil h'(1)\rceil$.   Picking $h'(1)$ to be sufficiently close to its ceiling, we then find by \eqref{eqn:cz} that
\[CZ(\gamma_+^k)=\lfloor  k(h'(1) - \lceil h'(1) \rceil ) \rfloor + \lceil  k(h'(1) - \lceil h'(1) \rceil ) \rceil =  -1,\]
for any $k=1,\dots, d$.
\item The South pole orbit $\gamma_-$ is elliptic with rotation number $-h'(-1)$ with respect to the considered trivialization. Since $h'(-1)$ is positive but arbitrarily small, we then find by \eqref{eqn:cz} that:   
\[CZ(\gamma_-^k)=\lfloor -k\,h'(-1) \rfloor+\lceil -k\,h'(-1) \rceil=-1.\]
\item For other orbits, we are in the same settings as \cite{Hutchings-Sullivan-Dehntwist}. Namely, for hyperbolic orbits, the rotation number is $0$, so from \eqref{eqn:cz} we have
\[CZ(h_{p/q})=0,\]
and for the elliptic orbits $e_{p/q}$, the rotation number is slightly negative, so that from \eqref{eqn:cz} we have 
\[CZ(e_{p/q}^k)= -1,\]
for any $k =1, \ldots, d$.  
\end{enumerate}

It follows from the above that the contribution of the Conley-Zehnder part to the index in \eqref{eqn:twistedgrading} is given by 
\begin{equation}
\label{eqn:CZcomp}
CZ_{\tau} (\alpha) = 
 \sum_{i} \sum_{k=1}^{m_i} CZ(\alpha_i^k)  + \sum_{k=1}^{m_-} CZ(\gamma_-^k) + \sum_{k=1}^{m_+} CZ(\gamma_+^k) = 
-M + h,
 \end{equation}
where $M$ denotes the total multiplicity of all orbits, and $h$ denotes the total number of hyperbolic orbits. 

\medskip
 {\noindent \em Step 2: The relative Chern class.}
 
The relative Chern class $ c_\tau(Z'_-)$  is  obviously $0$. For $\alpha_i=e_{p,q}$ or $h_{p,q}$, we consider the representative $S_{\alpha_i}$ of $Z'_{\alpha_{i}}$ given above. We choose the section $\partial_\theta$ as a non-winding section of $\xi|_{S_i}$ along $\alpha_{i}$, and any constant non-zero vector along $q\gamma_-$. Then, the section $\partial_\theta$ over the orbit $\alpha_{p/q}$ has index $-p$ while the section over $q \gamma_-$ has index $0$. 
It follows that any extension of these sections over $S_{\alpha_{i}}$ must have $-p$ zeros.  Hence,
\[c_\tau(Z'_{\alpha_{i}})=-p.\]
For $Z_+$, an argument analogous to that of the previous paragraph  gives
\[c_\tau(Z'_+)=-\lceil h'(1)\rceil.\]
The Chern class is additive, so we conclude from the above that the $c_\tau$ term of the index is 
\begin{align}
\label{eqn:chernclass}
c_\tau(Z'_{\alpha}) &= \sum m_i c_\tau(Z'_{\alpha_{i}}) +  m_- \; c_\tau(Z'_-) + m_+ \; c_\tau(Z'_+) \nonumber \\ &= 
 -\sum m_i p_i - m_+ \lceil h'(1) \rceil =-w_\alpha+y_\alpha.
 \end{align}

{\noindent \em Step 3: The relative self-intersection.}

Inspired by an analogous construction performed in \cite[Lemma 3.7]{Hutchings-Sullivan-Dehntwist}, we construct a representing surface $S\subset [0,1]\times\S^2\times\S^1$ of $Z'_\alpha$ as a movie of curves as follows. Denote by $\sigma$ the variable in $[0,1]$, and $S_\sigma \subset \{\sigma\} \times \S^2\times\S^1$.  We will describe $S_\sigma$ as $\sigma$ decreases from $1$ to $0$.
\begin{itemize}
\item For $\sigma=1$, $S_1$ is the union of the orbits appearing in  $\alpha$ with non-zero multiplicity.
\item For values of $\sigma$  close to 1, $S_\sigma$ consists of 
\begin{enumerate}[(a)]
\item $m_i$ circles, parallel to the orbit  $\alpha_i$,  in the torus $\{z=z_{p_i/q_i}\}\times\S^1\subset\S^2\times\S^1$ (these circles have slope $\frac{q_i}{p_i}$ if we see this torus as $\R/\Z\times\R/\Z$),
\item $m_+$ parallel circles with slope $\frac 1{\lceil h'(1) \rceil}$  in the torus $\{z=\sigma\}\times\S^1$,
\item $m_-$  parallel ``vertical'' circles $\{pt\}\times\S^1$ in the torus $\{z=-\sigma\}\times\S^1$.
\end{enumerate}
\item As $\sigma$ decreases to $1/2$, we move all these circles to the same   $\{z=\mathrm{constant}\}\times\S^1$ torus. As in \cite{Hutchings-Sullivan-Dehntwist}, we perform negative surgeries, around $\sigma = 1/2$ to obtain an embedded union of circles in a single $\{z=\mathrm{constant}\}\times\S^1$ torus; this union will consist of  $k$  (straight) parallel embedded 
circles directed by a primitive integral vector $(a,b)$. The vector $(a,b)$ and the number $k$ of circles are determined by our data: for homological reasons we must have 
  \begin{align*}
 kb=d,\qquad ka=\sum_i m_ip_i + m_+ \lceil h'(1) \rceil = w_\alpha-y_\alpha.
  \end{align*} 
\item As $\sigma$ decreases from $\frac12$ to $0$, we 
modify the torus in which the curves are located, without changing the curves themselves, so that $S_{0}$ is $\gamma_-$. 
\end{itemize}

We will compute $Q_\tau (Z'_{\alpha})$ using the above surface $S$ and the formula \eqref{eqn:intersection}.

To compute $c_1(N, \tau)$, we take $\psi \in \Gamma (N)$  as follows: we take $\psi = \pi_N \partial_\theta$ everywhere on $S$ except in a small neighborhood of the boundary components $\gamma_+, \gamma_-$ where the vector $\partial_\theta$ is not well-defined.  The surface is constructed such that $\psi$ may be extended to a $\tau$-trivial section of $N$ over $\gamma_+$ without introducing any zeroes.  However, extending $\psi$ to a $\tau$--trivial section of $N$ over $\gamma_-$ necessarily creates $-a$ zeroes for each of the $k$ embedded circles which gives a total of $-ka$ zeroes.

As in \cite{Hutchings-Sullivan-Dehntwist}, the other zeroes of $\psi$ appear at the surgery points with negative signs and their number is  given by \[ - \sum \det \begin{pmatrix} 
 p & p' \\
 q & q' 
\end{pmatrix},\]
where the sum runs over all pairs of distinct edges $v_{p,q}$, $v_{p',q'}$ in $P_{\alpha,Z}$, with $\tfrac{p'}{q'}<\tfrac pq$. Geometrically, this sum can be interpreted as $- 2  \mathrm{Area}(\mathcal{R_\alpha}')$, where $\mathcal{R}_\alpha'$ is  the region between $P_{\alpha,Z}$ and the straight line connecting $(0,y_\alpha)$ to $(d,w_\alpha)$.
Thus, we obtain: 
\[ c_1(N, \tau) = -(w_\alpha-y_\alpha) - 2  \mathrm{Area}(\mathcal{R_\alpha}').\]

We must now compute the writhe $w_\tau(S)$.  By construction, there is no writhe near $\sigma = 1$.  Near $\sigma = 0$ the writhe is given by the writhe of the braid $k(a,b)$ on the torus which is $(w_\alpha-y_\alpha)(1-d)$, so we get
\[ w_{\tau}(S) = (w_\alpha-y_\alpha)(d-1).\]

Summing the above, we therefore get
\begin{align}
  Q_\tau(Z'_{\alpha}) = -(w_\alpha-y_\alpha) - 2  \mathrm{Area}(\mathcal{R_\alpha}')-(w_\alpha-y_\alpha)(d-1). \label{eqn:selfintersection}
\end{align}

{\noindent \em Step 4: The combinatorial interpretation.}
 
 We now put all of this together to prove the second item in Proposition~\ref{prop:model}.  

By combining \eqref{eqn:CZcomp}, \eqref{eqn:chernclass} and \eqref{eqn:selfintersection} and the definition of the grading \eqref{eqn:twistedgrading}, we have
\begin{align*}
I(\alpha,Z'_{\alpha}) = -M+h -(w_\alpha-y_\alpha)(d+1)-2 \mathrm{Area}(\mathcal{R}_\alpha').
\end{align*}
Using Equation (\ref{eq:Zalpha-Zalpha_prime}), we obtain
\[ I(\alpha,Z_{\alpha})=-M+h+2\mathrm{Area}(\mathcal{R}_\alpha)+(w_\alpha-y_\alpha),\]
where $\mathcal{R}_\alpha$ denotes the region between $P_{\alpha,Z_{\alpha}}$  and the $x$-axis.

By Pick's theorem,
\[ 2 \mathrm{Area}(\mathcal{R}_\alpha) = 2 T - (M + d+ (w_\alpha-y_\alpha)) - 2,\]
where $T$ denotes the total number of lattice points in the closed region $\mathcal{R}_{\alpha}$.  So, by combining the previous two equations, we get
\[ I(\alpha,Z_{\alpha}) = 2(T-M - 1) -d + h.\]
Now $(T-M-1)$ is exactly the number of lattice points in the closed region $\mathcal{R}_{\alpha}$, not including the lattice points on the path,  
and so, $T-M-1 = j$, hence the second item of Proposition~\ref{prop:model} is proved for $Z=Z_\alpha$.

Now remember that if we shift our path upwards by $(0,1)$, then $j$ increases by $d+1$. Thus, using (\ref{eq:index-ambiguity-here}), we deduce that the second item of Proposition~\ref{prop:model} is satisfied for all relative classes $Z$.
\end{proof}

\subsubsection{Fredholm index in the combinatorial model}\label{sec:fredholm_combinatorial}
The goal of this section is to give a simple formula for the Fredholm index which relates it to our combinatorial model.

Let $C$ be a $J$--holomorphic curve in $\mathcal{M}_{J}( (\alpha,Z), (\beta,Z'))$, where $J$ is weakly admissible.   Recall from Equation \eqref{eqn:indexformula_Fredholm} that the Fredholm index of $C$ is given by the formula
$$\text{ind}(C) = - \chi(C) + 2 c_{\tau}(C) + CZ^{\mathrm{ind}}_{\tau}(C),$$
where $\chi(C)$ denotes the Euler characteristic of the curve, $c_{\tau}(C)$ is the relative first Chern class which we discussed above, and $CZ^{\mathrm{ind}}_{\tau}$ is a term involving the Conley-Zehnder index  defined as follows: Write $\alpha = \{(\alpha_i, m_i)\}$ and $\beta =\{(\beta_j, n_j)\}$.  Suppose that $C$ has ends at $\alpha_i$ with multiplicities $q_{i,k}$ and  ends at $\beta_j$ with multiplicities $q_{j,k}'$;  note that we must have $\sum_k q_{i,k} = m_i$ and $\sum_k q_{j,k}' = n_j$.  Then,   

$$ CZ^{\mathrm{ind}}_{\tau}(C) := \sum_i \sum_k CZ_\tau(\alpha_i^{q_{i,k}}) \;  - \; \sum_j \sum_k CZ_\tau(\beta_j^{q_{j,k}'}).$$
The next lemma explains  how to compute $\mathrm{ind}(C)$  from the combinatorial model.   In this lemma, we denote the  starting points of $P_{\alpha, Z}$ and $P_{\beta, Z'}$ by $(0, y_\alpha)$ and $(0, y_\beta)$, and their endpoints by $(d, w_\alpha)$ and $(d, w_\beta)$, respectively.

\begin{lemma}\label{lem:fredholm_index} Let $\varphi_0$ be a nice perturbation of $\varphi^1_H$, where $H \in \mathcal{D}$, let $J$ be any weakly admissible almost complex structure, and let $C$ be any irreducible $J$-holomorphic curve from $(\alpha, Z)$ to $(\beta, Z')$.  Then,
\begin{equation}
\label{eqn:Fredholmindex}
\text{ind}(C) = -2 + 2g + 2e_- + h + 2v ,
\end{equation}
 where $g$ is the genus of $C$, $e_-$ is the number of negative ends of $C$ which are at elliptic orbits, $h$ is the number of ends of $C$ at hyperbolic orbits, and $v = (y_\alpha - y_\beta) + (w_\alpha - w_\beta) $.
\end{lemma}

\begin{proof}
  We will prove the lemma by describing each of the three terms $\chi(C), c_\tau(C)$, and $CZ_\tau^{\mathrm{ind}}(C)$, which appear in $\mathrm{ind}(C)$, in terms of our combinatorial model. 
  
 The number of ends of the curve $C$ is given by the sum $e_-+ e_+ +h$, where $e_+$ denotes the number of positive ends of $C$ which are at elliptic orbits.  The Euler characteristic of $C$ is given by the formula 
  $$\chi(C) = 2 -2g  - e_- - e_+ - h.$$  
  
  As for the Chern class, because $[C] = Z- Z'$, we can write $c_{\tau}(C) = c_{\tau}(Z) -c_\tau(Z')$.  Now,  by the index computations of the previous section, $c_{\tau}(Z) = w_\alpha + y_\alpha$.  Similarly, $c_{\tau}(Z') = w_\beta + y_\beta$. It follows that $$c_{\tau}(C) = v.$$
  
  To compute $CZ_\tau^{\mathrm{ind}}(C)$, note that by our computations from the previous section, we have $CZ_\tau(\alpha_i^{q_{i,k}}) = -1$ if $\alpha_i$ is elliptic and $0$, otherwise; a similar formula holds for $CZ_\tau(\beta_j^{q_{j,k}'})$.  This implies that
  $$CZ_\tau^{\mathrm{ind}}(C) = -e_+ + e_-.$$
  
  Combining the above, we get $\mathrm{ind}(C) = -2 + 2g + 2e_- + h + 2v.$
\end{proof}

\begin{remark}
\label{rmk:why!}
As already mentioned in Section~\ref{sec:admissible_vs_weakly}, in some situations we want the added flexibility of being able to work with weakly admissible almost complex structures.  Lemma~\ref{lem:fredholm_index} above is also stated for weakly admissible almost complex structures, as are the forthcoming  Lemma \ref{lem:positivity} and Lemma~\ref{lem:above}.  Ultimately, the reason we want to be able to work with weakly admissible almost complex structures is because of the very useful Lemma~\ref{lem:comp} below, which will allow us to connect an admissible $J$ to a weakly admissible one in order to facilitate computations; the point is that the proof of Lemma~\ref{lem:comp} requires Lemma~\ref{lem:fredholm_index},  Lemma \ref{lem:positivity},  and Lemma~\ref{lem:above} in the weakly admissible case.  
\end{remark}

\subsection{Positivity}
\label{sec:pos}

We now begin the proof of the third item of Proposition~\ref{prop:model}; this will take several subsections and will require a close examination of those $J$--homolorphic  curves  in  $\R \times \S^1 \times \S^2$ which appear in the definition of the PFH differential.  (Recall from Section \ref{sec:specdefn} that $X$ is identified with $\R \times \S^1 \times \S^2$.) In this section, we prove a very useful lemma which puts major restrictions on what kind of $J$--holomorphic curves can appear.  The lemma and its proof are inspired by \cite[Lemma 3.5]{Choi}, which is in turn inspired by arguments in \cite{Hutchings-Sullivan-Dehntwist} and \cite{Hutchings-Sullivan-T3}, see for example \cite[Lemma 3.11]{Hutchings-Sullivan-Dehntwist}.      

To prepare for the lemma of this section, we need to introduce some new terminology. 
  For $-1 < z_0 < 1 $,  define the {\bf slice}
\[ S_{z_0} := \lbrace (s,t,\theta,z) \in \mathbb{R} \times \S^1 \times \S^2 : z = z_0\rbrace.\]
This is homotopy equivalent to a two-torus, and in particular we have $H_1(S_{z_0}) = H_1(\S^1_t \times \S^1_{\theta})$; we identify $H_1(\S^1_t \times \S^1_{\theta})$ with $\mathbb{Z}^2$ so that the positively oriented circle factors $\S^1_t$ and $ \S^1_{\theta}$ correspond to the vectors $(1,0)$ and $(0,1)$, respectively. 

Let $C$ be any $J$-holomorphic curve.  If $C$ is transverse to $S_{z_0}$ (which happens for generic $z_0$) and $C$ has no ends at 
$z_0$, then $C_{z_0}=C \cap S_{z_0}$ is a (possibly empty) compact  $1$-dimensional submanifold of $C$.  When non-empty, it is the boundary of the subdomain %
given by $C \cap \{z\leq z_0\}$. Thus, the orientation of $C$  induces an orientation on $C_{z_0}$: our convention is that we take the opposite 
of the usual ``outer normal first''\footnote{ \label{footnote:orientation} Recall that in the ``outer normal first" orientation, a vector $v$ on $C_{z_0}$ is positive if for an inner normal vector $w$, the frame $(v,w)$ is positive.} convention.  Therefore, we get a well-defined class $[C_{z_0}] \in H_1(S_{z_0})=\Z^2$,
which we call the {\bf slice class}.

\begin{lemma}
\label{lem:positivity}
Let $\varphi_0$ be a nice perturbation of $\varphi^1_H$, where $H \in \mathcal{D}$.  
Let $C$ be a $J$--holomorphic curve, where $J$ is weakly admissible, and let $z_0 \in (-1,1)$ be such that $C$ intersects $S_{z_0}$ transversally, and the nice perturbation vanishes in an open neighborhood of $z_0$.  (In particular, $C$ has no ends at Reeb orbits near $z_0$.)
Then,
\begin{equation}
\label{eqn:keyinequality}
\left(1,h'(z_0)\right) \times [C_{z_0}] \geq 0,
\end{equation}
with equality only if $C$ does not intersect $S_{z_0}$.   
\end{lemma}  
Here,   $(a,b) \times (c,d)$, where $(a,b), (c,d) \in \R^2$, is defined to be the quantity $ad -bc$.

\begin{remark}\label{rem:positivity_irreducible}
In the context of the above lemma, suppose that  $C$  is irreducible and let $z_{\min} := \inf\{ z_0: [C_{z_0}] \neq 0\}$ and $z_{\max} := \sup\{ z_0: [C_{z_0}] \neq 0\}$; here, we consider the $z_0$ such that the above lemma is applicable. The curve $C$ is connected, because it is irreducible, and thus its projection to $\S^2$ is also connected.  Therefore, $C$ is contained in $\{(s,t,\theta, z) \in \R \times  \S^1 \times  \S^2: z_{\min} \leq z \leq z_{\max} \}$ .
\end{remark}

\begin{remark}\label{rem:locality}
In the context of the above lemma, suppose that  $C$  is an irreducible 
$J$--holomorphic curve such that $[C_{z_0}] = 0$  for all $z_0$ satisfying the conditions of the lemma.  Then, as a consequence of the above lemma, $C$ must be a {\it local} curve in the following sense:  there exists $z_{p,q}$ such that all ends of $C$ are either at $e_{p,q}$ or $h_{p,q}$. 
\end{remark}

\begin{proof} We use the fact that, by Lemma~\ref{lem:pointwisenonnegative} the canonical $2$-form $\omega_{\varphi}$ is pointwise nonnegative on $C$, with equality only if the tangent space is the span of $\partial_s$ and $R$.  
Namely, as a function of $z > z_0$, close to $z_0$, we have that the mapping 
  \[ \rho:z \mapsto \int_{C \cap (\mathbb{R} \times \S^1_t \times \S^1_\theta \times [z_0,z]) } \omega_{\varphi} \]
is non-decreasing, as $z$ increases.  Hence, its derivative with respect to $z$ is nonnegative.  We will prove Equation \eqref{eqn:keyinequality} by showing that 
\begin{equation}\label{eqn:derivative}
\rho'(z) = \tfrac{1}{2}\left(1,h'(z)\right) \times [C_{z_0}],
\end{equation}
for $z$ close to $z_0$ and $z \geq z_0$. 

Recall from Section \ref{sec:specdefn} that $\omega_\varphi$ is identified with $\omega + d\tilde H \wedge dt$.  Since $\tilde H$ coincides with $H$ near $z_0$ we can write $\omega_\varphi = \omega + dH \wedge dt. $     Now, we have 
\[\omega_{\varphi}=\omega + dH \wedge dt = \tfrac{1}{4 \pi} d(- z\,  d \theta)  + d(H\,dt)  =  d\left( -\tfrac{1}{4\pi}z\, d \theta  + H(z)\, dt\right).\]
Let $\alpha =  H(z) dt - \frac{1}{4\pi}z\, d \theta$ and note that $\alpha$ restricts to a closed 1-form $\alpha_z$ on any slice $S_z$. %

Our choice of orientation 
gives: \[\partial (C \cap (\mathbb{R} \times \S^1_t \times \S^1_\theta \times [z_0,z])) = C_{z_0} - C_z.\]
Thus it follows from the above:
\begin{align*}
  \rho(z)&=\int_{C \cap (\mathbb{R} \times \S^1_t \times \S^1_\theta \times [z_0,z])} d\alpha = \int_{C_{z_0} } \alpha - \int_{C_z} \alpha \\ &= \langle \alpha_{z_0} - \alpha_z , [C_{z_0}] \rangle.
\end{align*}
The rate of change of the expression above with respect to $z$ is given by
\begin{equation}
\label{eqn:todiff}
\rho'(z)=- \left (\tfrac{1}{2}h'(z), -\tfrac{1}{2}\right) \cdot [C_{z_0}] = \tfrac{1}{2} \left(1, h'(z)\right) \times [C_{z_0}],
\end{equation}
which extends by continuity to $z_0$.  This proves \eqref{eqn:keyinequality}.

To prove that equality in Equation \eqref{eqn:keyinequality} forces the slice to be empty, assume that $\rho'(z_0)=0$.  We have shown that $\rho'\geq 0$, and so $\rho'$ must have a local minimum at $z_0$.  Hence, $\rho''$ and $\rho'$ vanish at $z_0$.  We claim that this implies $[C_{z_0}] = 0$.  To prove this, write $[C_{z_0}] = (a, b)$ and note that, by Equation \eqref{eqn:derivative}, we have 
\begin{equation*}
\rho'(z) = \tfrac12(b - a\, h'(z))\quad\text{and}\quad \rho''(z)  =  -\tfrac12\,{  a\, h''(z)}.
\end{equation*}
Since we are assuming $h''(z_0) \ne 0$, the vanishing of both of the above quantities can only take place if $a = b =0$.  We can therefore conclude that $[C_{z_0}] = 0$.

By the assumption that $C$ intersects $S_{z_0}$ transversely, we conclude that  $[C_{z}] = 0$  for $z$ sufficiently close to $z_0$.  Hence, $\rho(z)=0$ for $z$ sufficiently close to $z_0$.  But,  by Lemma~\ref{lem:pointwisenonnegative} this can only occur if the tangent space to $C \cap (\mathbb{R} \times \S^1_t \times \S^1_\theta \times [z_0,z]) $ is always in the span of the Reeb vector field $R$ and $\partial_s$, which are tangent to ${S_z}$.  But, since $C$ is transverse to $S_{z_0}$, this cannot happen for $z$ sufficiently close to $z_0$, unless the intersection $C_{z_0}=C\cap S_{z_0}$ is empty. 
\end{proof}

The next lemma, which is inspired by \cite[Eq. 28]{Hutchings-Sullivan-Dehntwist} and \cite[Definition 1.4]{Choi}, allows us to compute the slice class $[C_{z_0}]$ from our combinatorial model.  We suppose here that $(\alpha, Z), (\beta, Z')$ are two PFH generators for $\varphi_0$, the nice perturbation of $\varphi^1_H$, as described in Section \ref{sec:ham} and that $C$ is a $J$--holomorphic curve from $(\alpha,Z)$ to $(\beta,Z')$; recall that this means that $C$ is a $J$--holomorphic curve from $\alpha$ to $\beta$ such that $[C] = Z - Z'.$  Let $P_{\alpha,Z}$ and $P_{\beta,Z'}$ be the  concave lattice paths associated to $(\alpha,Z)$ and $(\beta,Z')$, respectively, as described in Section \ref{sec:comb}. For any $z_0$, let $P^{z_0}_{\alpha}$ be the vector obtained by summing all of the vectors in the underlying path $P_{\alpha}$ which correspond to Reeb orbits that arise from $z < z_0$. Denote $$P^{z_0}_{\alpha, Z} = (0, y_\alpha) + P^{z_0}_{\alpha},$$
where $(0,y_{\alpha})$ denotes the starting point of $P_{\alpha,Z}$ on the $y$--axis.   We define $P^{z_0}_{\beta}$ and $P^{z_0}_{\beta, Z'}$, analogously.

\begin{lemma}
\label{lem:slice}
Let $C$ be a $J$-holomorphic curve from $(\alpha,Z)$ to $(\beta,Z')$ and let $z_0 \in (-1,1)$ be such that 
$\alpha, \beta$ have no Reeb orbits near $z_0$ and suppose that $C$ intersects $S_{z_0}$ transversally.  Then,
\begin{equation}
\label{eqn:sliceformula}
[C_{z_0}] =  P^{z_0}_{\alpha, Z } - P^{z_0}_{\beta, Z'}. 
\end{equation}
\end{lemma}
\begin{proof} First, note that the first coordinate of $P_{\alpha}^{z_0}$ (hence that of $P_{\alpha, Z}^{z_0}$, too) corresponds to the class in $H_1(\S^1_t\times\S^2)\simeq\Z$ obtained by summing the contributions of the orbits in $\alpha$ that belong to the domain $\{z<z_0\}$. The second coordinate of $P_{\alpha}^{z_0}$ is given by the $\theta$ component of the class in $H_1(\S_t^1\times\S_\theta^1)$ obtained by summing the contributions of the orbits in $\alpha$, other than $\gamma_-$, which are included in $\{z<z_0\}$.  Analogous statements hold for $P_{\beta}^{z_0}$.

Pick any $z_-<z_+$ such that the assumptions of Lemma~\ref{lem:slice} hold for $z_-$ and  $z_+$, %
 and denote the part of $C$ with $z_- \leq z \leq z_+$ by $C_{[z_-, z_+]}$. This is asymptotic to some orbit set $\alpha'$ at $+\infty$ and some orbit set $\beta'$ at $-\infty$.

The boundary of $C_{[z_-, z_+]}$ has components corresponding to $\alpha', \beta', C_{z_+}$ and $C_{z_-}$.  The positive ends $\alpha'$ have the orientation coming from the Reeb vector field and the negative ends $\beta'$ have the opposite orientation.  We therefore have
\begin{equation}
\label{eq:sliceclass}
 [C_{z_+}] = [\alpha'] - [\beta'] + [C_{z_-}]
\end{equation}
in 
$H_1(\S^1_t \times \S^1_{\theta})\simeq \Z^2$.  We can apply similar reasoning to the part $C_{[-1,z_0]}$ of $C$ with $-1 \leq z \leq z_0$ to learn that
\[
\iota_*[C_{z_0}] = [\alpha'] - [\beta'],
\]
in $H_1(\S^1_t \times \S^2),$ where
\[\iota_*: H_1(\S^1_t \times \S^1_{\theta}) \to H_1(\S^1_t \times \S^2)\]
is the map induced by inclusion.  In particular, 
the first component of $[C_{z_0}]$ is that of $[\alpha'] - [\beta']$, which is exactly the first coordinate of $P^{z_0}_{\alpha, Z } - P^{z_0}_{\beta, Z'}$. %

 We now turn our attention to the second coordinate of $[C_{z_0}]$. Note that $[C_{z_0}]$ is determined by $[C]=Z-Z'$ in the following way. Consider the Mayer-Vietoris sequence associated to the cover by the two  open subsets $\R\times\S^1\times\{z\in\S^2: z<z_0+\delta\}$ and $\R\times\S^1\times\{z\in\S^2:z>z_0-\delta\}$. The class $-[C_{z_0}]$ %
 is the image of $[C]$ under the connecting map\footnote{Our convention is that the connecting map $\partial$ in Mayer-Vietoris is determined by the ordering of the cover; in particular, in the present situation, for any class $K$, $\partial K$ is given by taking the boundary of the component of $K$ in the region with $z  < z_0 + \delta$.} \[H_2(\R\times\S^1\times\S^2, \alpha, \beta)\to H_1(\R\times\S^1\times\{z\in\S^2:z_0-\delta<z<z_0+\delta\})\simeq H_1(S_{z_0}),\]
where $\delta>0$ is small enough.  In particular, adding $y[\S^2]$ with $y\in\Z$ to the class $[C]$, adds $y$ to the second component of $[C_{z_0}]$, since recall that our convention for the orientation on $[C_{z_0}]$ is opposite the standard boundary orientation. 

To compute the second component of $[C_{z_0}]$, we apply Equation
(\ref{eq:sliceclass}) in the case where $z_+=z_0$ and $z_-$ is such that $C_{[-1, z_-]}$ has no ends other than possibly $\gamma_{-}$. We obtain in this case that the second component of $[C_{z_0}]-[C_{z_-}]$ is the second coordinate of $P^{z_0}_{\alpha} - P^{z_0}_{\beta}$. 

There only remains to establish that the second component of $[C_{z_-}]$ is $y_\alpha-y_\beta$. We will use here the fact that $[C_{z_-}]$ is determined by $Z$ and $Z'$. In the case, where $Z=Z_\alpha$ and $Z'=Z_\beta$ (as defined in Section \ref{sec:comb}), the second coordinate of $[C_{z_-}]$ vanishes. In general, $[C]=Z_\alpha+y_\alpha[\S^2]-(Z_\beta+y_\beta[\S^2])$. Thus, the second component of $[C_{z_-}]$ is $y_\alpha-y_\beta$, which concludes the proof.
\end{proof}

\subsection{Paths can not cross}

As a consequence of the results in Section \ref{sec:pos}, we can prove the following useful fact which will play an important role in our proof of Proposition \ref{prop:model}.  It is inspired by  \cite[Prop 3.12]{Hutchings-Sullivan-Dehntwist} and \cite[Prop 10.12]{Hutchings-Sullivan-T3}.

\begin{lemma}
\label{lem:above}
Let $\varphi_0$ be a nice perturbation of $\varphi^1_H$, where $H \in \mathcal{D}$, and let $J$ be any weakly admissible almost complex structure.  If there exists a $J$-holomorphic curve $C$ from $(\alpha,Z)$ to $(\beta,Z')$,  then $P_{\beta,Z'}$ is never above $P_{\alpha,Z}.$
\end{lemma}
\begin{proof}
  Let $(0, y_\alpha)$ and $(0, y_\beta)$ denote the starting points of $P_{\alpha,Z}$ and $P_{\beta,Z'}$, respectively.  We will first show that $y_\alpha \geq y_\beta$.   Denote by $m_\alpha, m_\beta$ the multiplicities  of $\gamma_-$ in $\alpha$ and $\beta$,  respectively.  Let $z_0 = -1 + \varepsilon$ for some small $\varepsilon > 0$.   We have $P^{z_0}_{\alpha,Z} = (m_\alpha, y_\alpha)$ and $P^{z_0}_{\beta,Z'} = (m_\beta, y_\beta) $.  Hence, by Lemma \ref{lem:slice},   $[C_{z_0}] = (m_\alpha - m_\beta, y_\alpha - y_\beta)$.   Applying Lemma \ref{lem:positivity}, we obtain
$$(1, h'(z_0)) \times [C_{z_0}] = (y_\alpha - y_\beta) -h'(z_0) (m_\alpha - m_\beta) \geq 0.$$
By our conventions from  Section \ref{sec:ham}, $ h'(z_0) \approx 0$ and thus $(1, h'(z_0)) \times [C_{z_0}] \approx y_\alpha - y_\beta$.  Since  $y_\alpha - y_\beta$ is integer valued, the above inequality yields $y_\alpha \geq y_\beta$.

Next, let $(d, w_\alpha)$ and $(d, w_\beta)$ denote the  endpoints of $P_{\alpha,Z}$ and  $P_{\beta,Z'}$, respectively.  We will now show that $w_\alpha \geq w_\beta$. Denote by $n_\alpha, n_\beta$ the multiplicities  of $\gamma_+$ in $\alpha$ and $\beta$,  respectively.  Let $z_0 = 1  - \varepsilon$ for some small $\varepsilon > 0$.   We have $P^{z_0}_{\alpha,Z} = (d, w_\alpha) - n_\alpha \left(1, \lceil h'(1) \rceil \right)$ and $P^{z_0}_{\beta,Z'} = (d, w_\beta) - n_\beta \left(1, \lceil h'(1) \rceil \right)$.  Hence,   by Lemma \ref{lem:slice},   $$[C_{z_0}] = (n_\beta - n_\alpha, w_\alpha - w_\beta + (n_\beta -n_\alpha) \lceil h'(1) \rceil ).$$ 
Now, applying Lemma \ref{lem:positivity}, we obtain
$$(1, h'(z_0)) \times [C_{z_0}] = w_\alpha - w_\beta + (n_\beta - n_\alpha) \left( \lceil h'(1) \rceil  - h'(z_0) \right) \geq 0.$$
By our conventions from  Section \ref{sec:ham}, $ \lceil h'(1) \rceil  - h'(z_0) \approx 0$ and thus $(1, h'(z_0)) \times [C_{z_0}] \approx w_\alpha - w_\beta$.  Since $w_\alpha - w_\beta$ is integer valued, the above inequality yields $w_\alpha \geq w_\beta$.

  To complete the proof, suppose that the conclusion of the lemma does not hold.  We have shown that $P_{\beta,Z'}$ cannot begin or end above $P_{\alpha,Z}$. Hence, we can find two intersection points $(a,b )$ and $(c,d)$, with $a< c$, between the two paths, such that the path $P_{\beta,Z'}$ is strictly above $P_{\alpha,Z}$ in the strip $\{(x,y)\in\R^2: a<x<c\}$.  Let $U, L$ denote the parts of $P_{\beta,Z'}, P_{\alpha,Z}$, respectively, which are contained in $\{(x,y) \in \R^2: a \leq x \leq c \}$. Consider the line connecting $(a,b)$ and $(c,d)$. We can find a point $z_0$, with $-1 < z_0 < 1$, such that $h'(z_0) = \frac{d-b}{c-a}$. 
  We will compute the slice class $[C_{z_0+ \varepsilon}]$, for sufficiently small $\varepsilon >0$ and will show that

 $$(1,h'(z_0 + \varepsilon)) \times [C_{z_0 + \varepsilon}] < 0,$$ 
which contradicts Lemma \ref{lem:positivity}.  The reason for considering $[C_{z_0+ \varepsilon}]$ instead of $[C_{z_0}]$ itself is that there might be Reeb orbits in $\alpha, \beta$ corresponding to $z_0$ in which case we cannot apply Lemmas  \ref{lem:positivity} \& \ref{lem:slice}. %

To compute the slice class $[C_{z_0 + \varepsilon}]$, we will compute $P^{z_0 + \varepsilon}_{\alpha, Z}, P^{z_0 + \varepsilon}_{\beta, Z'}$ and use Lemma \ref{lem:slice}.  We  begin with $P^{z_0 + \varepsilon}_{\alpha, Z}$. Let $(p,q)$ be the corner of $P_{\alpha, Z}$, on $L$,  with the following property: the edge in $P_{\alpha, Z}$ to the left of  $(p,q)$ has slope  at most $\frac{d-b}{c-a}$, and the edge to the right of $(p,q)$ has slope strictly larger than $\frac{d-b}{c-a}$.   The corner $(p,q)$ exists because $L$ is strictly below the line passing through $(a,b)$ and $(c,d)$.  Then, $P^{z_0 + \varepsilon}_{\alpha, Z} = (p,q)$.  Now, denote $P^{z_0 + \varepsilon}_{\beta, Z'}= (p', q')$; this vector may be computed as follows:  If the line passing through $(a,b)$ and $(c,d)$ is strictly above $U$, then $(p', q')$ is computed exactly as above.  If not, the line passing through $(a,b)$ and $(c,d)$ must coincide with $U$; then $(p', q')$ is the endpoint of the edge in  $P_{\beta, Z'}$ containing $U$.

We obtain $$[C_{z_0 + \varepsilon}]= (p-p', q-q') $$
with $(p,q)$ and $(p', q')$ as described in the previous paragraph.  Now, we have $$(1,h'(z_0 + \varepsilon)) \times [C_{z_0 + \varepsilon}] = \left(1, \tfrac{d-b}{c-a}\right) \times (p-p', q-q').$$  This quantity is negative because $L$ is strictly below $U$.  Indeed, one can see this by applying a rotation,  which does not change the determinant, so that $(1 , \frac{d-b}{c-a})$ is rotated to a positive multiple of $(1,0)$, and $(p-p',q-q')$ is rotated to a vector with a negative second component.  Hence,        $(1,h'(z_0 + \varepsilon)) \times [C_{z_0 + \varepsilon}] < 0$  which contradicts Lemma \ref{lem:positivity}.
\end{proof}

\subsection{Curves correspond to corner rounding}\label{sec:curves-corner-rounding}
Using the results we have obtained thus far, we can now describe the configurations of concave paths which could give rise to a non-trivial term in the PFH differential.  More precisely, we can now prove the ``only if'' part of Proposition~\ref{prop:model}, which we state as a lemma below:

\begin{lemma}
\label{lem:nice}
 Let $\varphi_0$ be a nice perturbation of $\varphi^1_H$, where $H \in \mathcal{D}$.  Assume that $I(P_{\alpha,Z}) - I(P_{\beta,Z'}) = 1$. Then, for generic admissible $J$ close to $J_{std}$,
\[ \langle \partial(\alpha,Z), (\beta,Z') \rangle \ne 0 \]
only if $P_{\alpha,Z}$ is obtained from $P_{\beta,Z'}$ by rounding a corner and locally losing one $h$.
\end{lemma}

\begin{proof}
 Assume that \[ \langle \partial(\alpha,Z), (\beta,Z') \rangle \ne 0\] for some generically chosen $J$ and some generators $(\alpha, Z)$ and $(\beta, Z')$.  We first choose $J$ generically to rule out double rounding, which we can do by the argument in \cite[Lemma A.1]{Hutchings-Sullivan-Dehntwist}.\footnote{In this argument, other than notational changes, we need to make one minor modification: the $2$-form $dt\wedge dy-ds\wedge dx$ in the proof of Lemma A.2  must be replaced with the $2$-form $dt\wedge d\theta - f(z) ds\wedge dz$, for a function $f$ determined by $h$.  (We could give an explicit formula for $f$, but it is not necessary for what we write here.)  The reason we need to add the function $f$ is because the almost complex structure $J_{std}$ does not map $\partial_z$ to $\partial_{\theta}$: in contrast, the almost complex structure $J_0$ from \cite{Hutchings-Sullivan-Dehntwist}, Lemma A.1 maps $\partial_x$ to $\partial_y$.  However, the rest of the argument there can be repeated essentially verbatim, since $\int_C f(z) ds dz = \int_C d (s f(z) dz)$.}
By Lemma \ref{lem:above}, we know that $P_{\beta,Z'}$ is never  above $P_{\alpha,Z}$.
Consider the region between $P_{\alpha,Z}$ and $P_{\beta,Z'}$.  We can take this region and decompose it into two kinds of subregions:  {\bf non-trivial} subregions where $P_{\alpha,Z}$ is {\bf above} $P_{\beta,Z'}$ --- meaning that the parts of $P_{\alpha,Z}$ and $P_{\beta,Z'}$ intersect at most at two points in these regions; and, {\bf trivial} subregions where the concave paths (without the labels) coincide. See Figure \ref{fig:rounding-corner3}.

 \begin{figure}[h!]
 \centering 
 \def\svgwidth{1.0\textwidth} 
\begingroup%
  \makeatletter%
  \providecommand\color[2][]{%
    \errmessage{(Inkscape) Color is used for the text in Inkscape, but the package 'color.sty' is not loaded}%
    \renewcommand\color[2][]{}%
  }%
  \providecommand\transparent[1]{%
    \errmessage{(Inkscape) Transparency is used (non-zero) for the text in Inkscape, but the package 'transparent.sty' is not loaded}%
    \renewcommand\transparent[1]{}%
  }%
  \providecommand\rotatebox[2]{#2}%
  \newcommand*\fsize{\dimexpr\f@size pt\relax}%
  \newcommand*\lineheight[1]{\fontsize{\fsize}{#1\fsize}\selectfont}%
  \ifx\svgwidth\undefined%
    \setlength{\unitlength}{263.77122506bp}%
    \ifx\svgscale\undefined%
      \relax%
    \else%
      \setlength{\unitlength}{\unitlength * \real{\svgscale}}%
    \fi%
  \else%
    \setlength{\unitlength}{\svgwidth}%
  \fi%
  \global\let\svgwidth\undefined%
  \global\let\svgscale\undefined%
  \makeatother%
  \begin{picture}(1,0.60692724)%
    \lineheight{1}%
    \setlength\tabcolsep{0pt}%
    \put(0,0){\includegraphics[width=\unitlength,page=1]{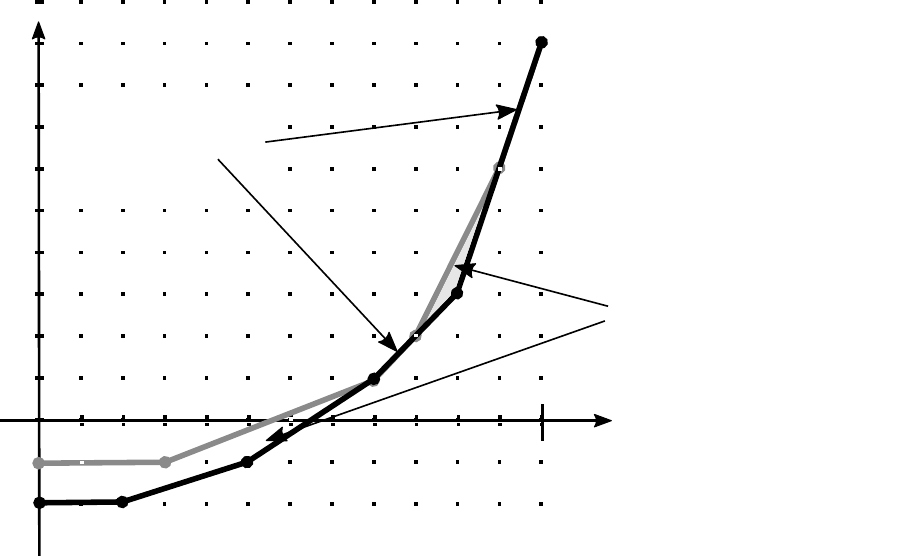}}%
    \put(0.67100986,0.253968){\color[rgb]{0,0,0}\makebox(0,0)[lt]{\lineheight{1.25}\smash{\begin{tabular}[t]{l}non-trivial regions\end{tabular}}}}%
    \put(0.07493909,0.44247374){\color[rgb]{0,0,0}\makebox(0,0)[lt]{\lineheight{1.25}\smash{\begin{tabular}[t]{l}trivial regions\end{tabular}}}}%
    \put(0,0){\includegraphics[width=\unitlength,page=2]{rounding-corner3.pdf}}%
    \put(2.88377436,-3.21655905){\color[rgb]{0,0,0}\makebox(0,0)[lt]{\begin{minipage}{2.67050492\unitlength}\raggedright \end{minipage}}}%
  \end{picture}%
\endgroup%

 \caption{Examples of trivial and non-trivial regions. The path $P_{\beta, Z'}$ is in black, the path $P_{\alpha, Z}$ is in grey were it does not coincide with $P_{\beta, Z'}$.}\label{fig:rounding-corner3}
\end{figure}

We will first show that there is at least one non-trivial region.  Arguing by contradiction, assume there is no non-trivial region, hence that $P_{\alpha,Z}$ and $P_{\beta,Z'}$ coincide as {\em unlabeled concave paths.} Let $C$ be the unique embedded component of a given $J$-holomorphic curve from  $(\alpha,Z)$ to $(\beta,Z')$; see Proposition \ref{prop:structure_Jcurves}.  We claim that $C$ must be {\em local} in the following sense:  It is a $J$--holomorphic cylinder from the hyperbolic orbit, near some $z= z_{p,q}$, that arises after the good perturbation, to the elliptic orbit near  $z = z_{p,q}$; furthermore,  it does not leave the neighborhood of $z_{p,q}$ where our good perturbation is non-trivial.    Indeed, if $C$ were not local, then we could find $z_0 \in (-1, 1)$ such that both of Lemmas \ref{lem:positivity} \& \ref{lem:slice} would be applicable at $z_0$.  Now, Lemma \ref{lem:positivity} would imply $[C_{z_0}] \neq 0$, while Lemma \ref{lem:slice} would imply $[C_{z_0}] = 0$ because the two (unlabeled) concave paths coincide.  Hence, $C$ must be local.  As explained in the proof of \cite[Lemma 3.14]{Hutchings-Sullivan-Dehntwist} local curves appear in pairs and so their mod $2$ count vanishes. 
   
 We will prove that if there exists $C \in \mathcal{M}_J( (\alpha,Z), (\beta,Z'))$, then $P_{\alpha,Z}$ is obtained from $P_{\beta,Z'}$  by rounding a corner and locally losing one $h$.   First, observe that it suffices to prove this under the assumption that $C$ is irreducible.  Indeed, if $C$ is not irreducible, consider its embedded component $C'$.  Then, as a consequence of Proposition \ref{prop:structure_Jcurves}, there exists PFH generators $(\alpha_1,Z_1)$ and $(\beta_1,Z_1')$ such that $C' \in \mathcal{M}_J( (\alpha_1,Z_1), (\beta_1,Z_1'))$ and $P_{\alpha_1,Z_1}$ is obtained from $P_{\beta_1,Z_1'}$ by rounding a corner and locally losing one $h$ if and only if $P_{\alpha,Z}$ is obtained from $P_{\beta,Z'}$ via the same operation.  We will suppose for the rest of the proof that $C$ is irreducible.
 
 \begin{claim} Under the assumption that $C$ is irreducible,  the region between $P_{\alpha, Z}$ and $P_{\beta, Z'}$ contains  no trivial regions and one non-trivial region. 
 \end{claim}
 \begin{proof}[Proof of Claim]

 Lemma \ref{lem:fredholm_index} implies that if the number of non-trivial regions  between $P_{\alpha, Z}$ and $P_{\beta, Z'}$ is at least $2$, then the Fredholm index of $C$ is also at least $2$, because a non-trivial region makes a contribution of size at least $2$ to the sum $2e_- + h + 2 v$ from Equation \eqref{eqn:Fredholmindex}.  Since  $\mathrm{ind}(C) = 1$, we conclude that the number of non-trivial regions must be one.

Next, we will prove that the number of trivial regions must be zero. First, we will show that an edge in $P_{\beta, Z'}$ cannot  have lattice points in its interior.  Indeed, if such an edge existed it would make a contribution of size at least $3$ to the term $2e_- + h$ in Equation \eqref{eqn:Fredholmindex}.  This  would then force $2e_- +h + 2v$ to be at least four which cannot happen because $\mathrm{ind}(C) =1$.  Now, suppose that there exists a trivial region between $P_{\alpha, Z}$ and $P_{\beta, Z'}$.  We will treat the case where there is a trivial region to the left of the non-trivial region, leaving the remaining case, which is very similar, to the reader. There can be at most one trivial region because each trivial region makes a contribution of size at least $1$ to $2e_- +h$, and so if there were two or more such regions  $2e_- +h + 2v$ would be at least four.  Let $v_{p, q}$ the vector/edge of the trivial region.  Then, the edge in $P_{\beta, Z'}$  immediately to the right of $v_{p,q}$ corresponds to a vector $v_{p_1, q_1}$ with $\frac{p}{q} < \frac{p_1}{q_1}$; this inequality is strict because otherwise $P_{\beta, Z'}$ would have an edge with an interior lattice point.  It follows that we can find $z_0$ such that  $z_{p,q} < z_0< z_{p_1, q_1}$ and Lemma \ref{lem:slice} is applicable at $z_0$; moreover, $P^{z_0}_{\alpha, Z} - P^{z_0}_{\beta, Z'} = 0$. Hence, $[C_{z_0}] = 0$. This contradicts Remark \ref{rem:positivity_irreducible} because    $C$ is asymptotic to orbits in both of $\{z \in \S^2: z \geq z_0\}$ and  $\{z \in \S^2: z \leq z_0\}$.  We therefore conclude that the region between $P_{\alpha, Z}$ and $P_{\beta, Z'}$ consists of a single non-trivial region.
 \end{proof}
Continuing with the proof of Lemma \ref{lem:nice}, observe that by the second item of Proposition~\ref{prop:model}, we have
\[  I(\alpha,Z)- I(\beta,Z') =  2j + h_{\alpha} - h_{\beta}, \]
where $j$  is the number of lattice points in the region between $P_{\alpha,Z}$ and $P_{\beta,Z'}$, not including lattice points on $P_{\alpha,Z}$, and $h_\alpha, h_\beta$ denote the number of edges labeled $h$ in $P_{\alpha, Z}, P_{\beta, Z'}$, respectively.   The number of edges in $P_{\beta, Z'}$, which we denote by $r_\beta$, satisfies the following inequality:  
$ h_\beta \leq r_\beta  \leq j + 1$ . Hence, we have \[ I(\alpha,Z)- I(\beta,Z') \geq 2 (r_{\beta}-1) - r_{\beta} = r_{\beta} - 2,\]
with equality if and only if $P_{\alpha,Z}$ and $P_{\beta,Z'}$ start at the same point, end at the same point, 
the region between $P_{\alpha,Z}$ and $P_{\beta,Z'}$ contains no interior lattice points, every edge of $P_{\beta,Z'}$ is labelled $h$, and no edge of $P_{\alpha,Z}$ is labeled $h$.  Since $I(\alpha,Z)- I(\beta,Z') =1$, we can rewrite the above inequality as $r_{\beta} \leq 3$.

If $r_\beta = 1$, then the equality $1 =   2j + h_{\alpha} - h_{\beta}$ can hold if and only if $j=1,  h_{\alpha} = 0, h_\beta =1$.  This implies that there are no interior lattice points in the region between $P_{\alpha, Z}$ and $P_{\beta, Z'}$; moreover,  the two paths either begin at the same lattice point or end at the same lattice point.  We see that in both cases $P_{\alpha, Z}$ is obtained from $P_{\beta, Z'}$ by rounding a corner and locally losing one $h$.  Note that the corner rounding takes place at the extremity of $P_{\beta, Z'}$ which is not on $P_{\alpha, Z}$.  

If $r_{\beta} = 2$, and if the region between $P_{\alpha,Z}$ and $P_{\beta,Z'}$ has at least one interior lattice point, then $j \geq 2$, so as $h_{\beta} \leq 2$, the index difference must be at least $2$, which can not happen.  Thus, the region between $P_{\alpha,Z}$ and $P_{\beta,Z'}$ must have no interior lattice points, and $h_\alpha = h_\beta - 1$.  Thus, in this case $P_{\alpha,Z}$ must be obtained from $P_{\beta,Z'}$ by rounding a corner and locally losing one $h$.  

If $r_{\beta} = 3$, then equality holds in the above inequality, 
so every edge of $P_{\beta,Z'}$ must be labelled $h$, no edge of $P_{\alpha,Z}$ can be, $P_{\alpha,Z}$ and $P_{\beta,Z'}$ must start and end at the same points 
and there are no interior lattice points between $P_{\alpha,Z}$ and $P_{\beta,Z'}$; thus $P_{\alpha,Z}$ must be obtained from $P_{\beta,Z'}$ by double rounding, which can not occur given our choice of $J$ from the beginning of the proof.  %
\end{proof}

\subsection{Corner rounding corresponds to curves} \label{sec:corner_rounding_curves}
In the previous section we  proved that if $\langle \partial (\alpha, Z), (\beta, Z')\rangle \neq 0$, then $P_{\alpha, Z}$ is obtained from $P_{\beta, Z'}$ by rounding a corner and locally losing one $h$. To complete the proof of Proposition \ref{prop:model}, we must show the converse:

\begin{lemma}
\label{lem:converse}
 Let $\varphi_0$ be a nice perturbation of $\varphi^1_H$, where $H \in \mathcal{D}$.  If $P_{\alpha,Z}$ is obtained from $P_{\beta,Z'}$ by rounding a corner and locally losing one $h$, then $\langle \partial (\alpha, Z), (\beta, Z')\rangle \neq 0$.  In other words, counting mod $2$ we have  $$\# \mathcal{M}_{J}( (\alpha,Z) , (\beta,Z')) =1,$$ for  generic admissible $J\in  \mathcal{J}(dr,\omega_\varphi)$.  
\end{lemma} 

The proof of the above lemma takes up the rest of this section.  As we will now explain, it is sufficient to prove the lemma under the assumption that every $C \in \mathcal{M}_{J}( (\alpha,Z) , (\beta,Z') ) $ is irreducible:  We can write  $P_{\alpha, Z}$ and $P_{\beta, Z'}$ as concatenations 
\begin{align*}
P_{\alpha, Z} = P_{\mathrm{in}} P_{\alpha_1, Z_1} P_{\mathrm{fin}},\\
P_{\beta, Z'} = P_{\mathrm{in}} P_{\beta_1, Z_1'} P_{\mathrm{fin}},
\end{align*}
where $P_{\mathrm{in}}$ and $P_{\mathrm{fin}}$ correspond to the (possibly empty) trivial  subregions between  $P_{\alpha, Z}$ and $P_{\beta, Z'}$, and $P_{\alpha_1, Z_1}, P_{\beta_1, Z_1'} $ correspond to the non-trivial subregion; here we are using the terminology of Section \ref{sec:curves-corner-rounding}.  The concave path  $P_{\alpha_1, Z_1}$ is obtained from $P_{\beta_1, Z_1'} $ by rounding a corner and locally losing one $h$, and the region between   $P_{\alpha_1, Z_1}$ and  $P_{\beta_1, Z_1'} $ consists of a single non-trivial subregion. 

\begin{claim}\label{cl:every_curve_irreducible}
Every current  $C_1 \in \mathcal{M}_{J}( (\alpha_1,Z_1) , (\beta_1,Z'_1) )$ is irreducible.  
\end{claim}
\begin{proof} By the structure of the corner rounding operation, $C_1$ has at most two negative ends.  Thus, by degree considerations it must have at most two irreducible components (closed components can not exist for various reasons, for example because the curve has $I = 1$, see Proposition \ref{prop:structure_Jcurves}); and, if it has exactly two irreducible components, then it must have exactly two negative ends, with one component corresponding to each end; assume this for the sake of contradiction, and write the components as $D_0$ and $D_1$.  

By Lemma~\ref{lem:above}, each of the $D_i$ must correspond to a region between concave paths that do not go above $P_{\alpha_1,Z_1}$, and do not go below $P_{\beta_1,Z'_1}$.  In fact, there are no such concave lattice paths, other than $P_{\alpha_1,Z_1}$ and $P_{\beta_1,Z'_1}$ because $P_{\alpha_1,Z_1}$ is obtained from $P_{\beta_1,Z_1'}$ by rounding a corner. Therefore, the region corresponding to each $D_i$ has a lower edge corresponding to one of the two edges of $P_{\beta_1, Z'_1}$ and the upper edges of each region must be on the edges of $P_{\alpha_1, Z_1}$.  It follows from the concavity of the paths that each region must have $v \geq 1$, where $v$ is defined as in Lemma \ref{lem:fredholm_index}.  Hence, by the Fredholm index formula \eqref{eqn:Fredholmindex}, each $D_i$ must have index at least one and so $C_1$ must have index $2$ which is not possible for generic $J$, by \eqref{eqn:indexinequality}. 
\end{proof}

Next, note that as a consequence of Proposition \ref{prop:structure_Jcurves}, every curve  $C \in \mathcal{M}_{J}( (\alpha,Z) , (\beta,Z') )$ can be written as a disjoint union 
\begin{equation}\label{eqn:more_structure}
C = C_{\mathrm{in}}  \sqcup C_1 \sqcup C_{\mathrm{fin}},
\end{equation}
where $C_1$
is the irreducible component of $C$ and $C_{\mathrm{in}}, C_{\mathrm{fin}}$ are unions of covers of trivial cylinders. It follows from Lemma \ref{lem:slice} and the equality case of Lemma \ref{lem:positivity} that  $C_{\mathrm{in}}, C_{\mathrm{fin}}$ correspond to the orbits in $P_{\mathrm{in}}, P_{\mathrm{fin}}$ and that $C_1 \in \mathcal{M}_{J}( (\alpha_1,Z_1) , (\beta_1,Z'_1) ).$  
 
Combining Claim \ref{cl:every_curve_irreducible}, and Equation \eqref{eqn:more_structure} in view of the discussion in the previous paragraph, we obtain a canonical bijection  $$\mathcal{M}_{J}( (\alpha,Z) , (\beta,Z') ) \stackrel\sim\to \mathcal{M}_{J}( (\alpha_1,Z_1) , (\beta_1,Z'_1) )$$ given by removal of covers of trivial cylinders.  

We conclude from the above discussion that it is indeed sufficient to prove Lemma \ref{lem:converse} under the assumption that, for generic admissible $J$, every $C \in \mathcal{M}_{J}( (\alpha,Z) , (\beta,Z') ) $ is irreducible.  In terms of our combinatorial model, this assumption is equivalent to requiring that the region between $P_{\alpha,Z}$ and $P_{\beta, Z'}$ consists of one non-trivial subregion and no trivial subregions.  This will be our standing assumption for the rest of this section.

\subsubsection{Deformation of $J$}
\label{sec:var}

The next ingredient, which we will need to prove the Lemma \ref{lem:converse},  allows us to deform $J$ within the  class of weakly admissible almost complex structures, while keeping the count of curves the same.  The method of proof is inspired by the proofs of Lemmas 3.15 and 3.17 of \cite{Hutchings-Sullivan-Dehntwist}.

Recall our standing assumption that the region between $P_{\alpha,Z}$ and $P_{\beta, Z'}$ consists of one non-trivial subregion and no trivial subregions.

\begin{lemma}
\label{lem:comp}
Let $\varphi_0$ be a nice perturbation of $\varphi^1_H$, where $H \in \mathcal{D}$. Let $J_0$ be an admissible and $J_1$ a weakly admissible almost complex structure.
Assume that $P_{\alpha,Z}$ is obtained from $P_{\beta,Z'}$ by rounding a corner and locally losing one $h$.
Then, if $J_0$ and $J_1$ are generic,
\[ \# \mathcal{M}_{J_0}( (\alpha,Z) , (\beta,Z') ) = \#  \mathcal{M}_{J_1}( (\alpha,Z), (\beta,Z')).\] 
\end{lemma}

\begin{remark}\label{rem:J_defined_on_X1}
We will apply the above lemma in a setting where $J_0, J_1$ are only defined on $\R \times X_1$, where $X_1$ is a subset of $Y_\varphi$ with the following property:  There exists a compact subset $K \subset X_1$ such that for any weakly admissible almost complex structure $J$ on $\R \times Y_\varphi$, every $C \in \mathcal{M}_{J}( (\alpha,Z) , (\beta,Z') )$ is contained in $\R \times K$.  

Because all $J$--holomorphic curves are contained in $\R \times K$, with $K\subset X_1$ compact, the proof we give below for Lemma \ref{lem:comp} works verbatim in the setting of the previous paragraph as well.  However, for clarity of exposition we give the proof in the setting were $J_0, J_1$ are globally defined.
\end{remark}

To prove the above lemma, we will need the following claim which will be used in the proof below and the next section.
\begin{claim} \label{cl:curves_dont_go_over_both_poles}
Suppose that $P_{\alpha, Z}$ and $P_{\beta, Z'}$ are as in Lemma \ref{lem:comp} and recall the definitions of $P^{z_0}_{\alpha,Z},P^{z_0}_{\beta,Z'}$ from Lemma \ref{lem:slice}.   Let $z_{\min}$ and $z_{\max}$ be the minimum and maximum values of $z_0\in[-1, 1]$ such that $P^{z_0}_{\alpha,Z} - P^{z_0}_{\beta,Z'} \neq 0$.  Then,
either $-1 < z_{\min}$ or $z_{\max} < 1$.
\end{claim}

\begin{proof}[Proof of Claim]

Denote by $(0, y_\alpha)$ and $(0, y_\beta)$ the starting points of $P_{\alpha, Z}$ and $P_{\beta, Z'}$, respectively, and by $(d, w_\alpha)$ and $(d, w_\beta)$ 
their endpoints. We begin by supposing  $z_{\min} = -1$,  and we will show this entails $z_{\max} < 1$.  By Lemma \ref{lem:slice}, if  $z_{\min} = -1$ then at least one of the following two scenarios must hold: First, the path $P_{\beta, Z'}$ begins with a horizontal edge  $(1,0)$.  Second, $y_\alpha > y_\beta$.  

In the first scenario, $P_{\beta, Z'}$ must have a second edge $(q,p)$ labelled $h$; this edge corresponds to some $z_{p/q} < 1$.  Note that we must also have $w_\beta = w_\alpha$, since $P_{\alpha,Z}$ is obtained from $P_{\beta,Z'}$ by rounding a corner and locally losing one $h$.  
It then follows that  $P^{z_0}_{\alpha,Z} - P^{z_0}_{\beta,Z'} = 0$
 for $z_{p/q} < z_0$ and so $z_{\max} < 1$.
 
Similarly, in the second scenario, $P_{\beta, Z'}$ has only one $(q,p)$ and this edge is labelled $h$.  Since it is labelled by $h$,  this edge must correspond to some $z_{p/q} < 1$.  As in the first scenario, we must also have $w_\beta = w_\alpha$.  It then follows that  $P^{z_0}_{\alpha,Z} - P^{z_0}_{\beta,Z'} = 0$
 for $z_{p/q} < z_0$ and so $z_{\max} < 1$.
 
 The case where $z_{\max} =1$ is similar to above and hence, we will not provide a proof.
\end{proof}  
\begin{proof}[Proof of Lemma \ref{lem:comp}] 
Any curve $C$ in $\mathcal{M}_{J_0}$ or $\mathcal{M}_{J_1}$ has Fredholm index $1$, and so it follows from Equation \eqref{eqn:Fredholmindex} that the genus of $C$ must be zero; this is because $2e_- + h+ 2v \geq 2$ in our setting.

  We next argue as in the proof of Lemma 3.17 of \cite{Hutchings-Sullivan-Dehntwist}.  As we explained in Section \ref{sec:admissible_vs_weakly}, we can connect $J_0$ and $J_1$ with a smooth  family of weakly admissible almost complex structures $J_s, s\in [0,1]$. 
  Consider  the moduli space $\mathcal{M} := \cup_s \mathcal{M}_{J_s}((\alpha, Z), (\beta, Z'))$, for a generic choice of $J_s, s\in [0,1] $.   There exists a global bound on the energy of curves in $\mathcal{M}$:  
 indeed, any two $C, C'$ are homologous, as elements of $H_2(Y_\varphi, \alpha, \beta),$ and so have the same energy\footnote{ More precisely, this argument shows that, in the language of \cite{BEHWZ},  $C,C'$ have the same $\omega$-energy.  It then follows from Proposition 5.13 of \cite{BEHWZ} that there exists a global bound on the $\lambda$-energy of the curves in  $\mathcal{M}$, as well.} $\int_C \omega_\varphi = \int_{C'} \omega_\varphi$.  
 Moreover, these curves all have genus zero as we explained in the previous paragraph. Hence, we can appeal to the SFT compactness theorem\footnote{The SFT compactness theorem holds in our setting where we are allowing the stable Hamiltonian structures to vary smoothly; see, for example, \cite[page 170]{Wendl-Notes}.} from \cite{BEHWZ} 
to conclude that if a degeneration of the moduli space $\mathcal{M}$ occurs at some $s$, then there is convergence to a broken $J_s$-holomorphic building $(C_0,\ldots,C_k)$;    here, $C_i$ is the $i^{th}$ level of the building and it is a $J$-holomorphic curve between PFH generators $(\alpha_i, Z_i)$ and $(\alpha_{i+1}, Z_{i+1})$ with $(\alpha_0, Z_0) =  (\alpha, Z)$ and $(\alpha_{k+1}, Z_{k+1}) = (\beta, Z')$. This building is in the homology class $Z-Z'$,  has genus $0$, and the top and bottom levels must have at most one end at any Reeb orbit,  by the partition conditions provided by Theorem 1.7 and Definition 4.7 in \cite{Hutchings-index} (here we are using the fact that the monodromy angles of our elliptic periodic orbits are close to zero and negative, and that since the curve is supposed irreducible, the orbits constituting $\beta$ have multiplicity $1$); the compactness theorem \cite{BEHWZ} guarantees that it is {\bf stable} and possibly {\bf nodal}; we will define these terms and elaborate on them below; each $C_i$ is also nontrivial, in the sense that no $C_i$ is a union of trivial cylinders

 Recall that a {\bf nodal} $J$-holomorphic curve $C$ is an equivalence class of tuples $(\Sigma,j, u, \Delta)$, where $(\Sigma,j,u)$ is a possibly disconnected $J$-holomorphic curve, and $\Delta \subset \Sigma$ is a finite set of points with even cardinality, called {\bf nodal points}; the set $\Delta$ has an involution $i$ with no fixed points, and $u(z) = u(i(z))$ for any $z \in \Delta$; see \cite[Sec.\ 2.1.6]{otherwendl} for the precise defintion.   We next recall the stability conditions mentioned above.  A nodal curve $(\Sigma,j,u,\Delta)$ is called {\bf stable} if every connected componet of $\Sigma - \Delta$ on which $u$ is constant has negative Euler characteristic.  The compactness theorem of  \cite{BEHWZ} guarantees that the curves $C_i$ in our building are all stable.

 To proceed, we will need the following claim.

\begin{claim} \label{cl:no_nodes_no_ghosts}
Each of the curves $C_i$, appearing in  our building $(C_0, \ldots, C_k)$, has the following properties:

\begin{enumerate}
\item $C_i$ has no irreducible component whose domain is closed,  no irreducible component which is constant, and no nodal points.
\item Irreducible components of $C_i$ have Fredholm index $0$ or $1$.
\end{enumerate}

Moreover, we have
\begin{equation}
\label{eqn:sumind}
\sum_i \text{ind}(C_i) = 1.
\end{equation}

\end{claim}
We postpone proving the above claim and continue with our  proof of Lemma \ref{lem:comp}. 

 By Lemma~\ref{lem:positivity}, any $C_i$ must correspond to a region between paths that do not go above $P_{\alpha,Z}$, and do not go below $P_{\beta,Z'}$.  In fact, there are no such concave lattice paths, other than $P_{\alpha,Z}$ and $P_{\beta,Z'}$ because $P_{\alpha,Z}$ is obtained from $P_{\beta,Z'}$ by rounding a corner. It follows that there exists $0 \leq j < k $ such that the paths $P_{\alpha_0, Z_0},  \ldots, P_{\alpha_j, Z_j}$ coincide with $P_{\alpha,Z}$  and the remaining paths $P_{\alpha_{j+1}, Z_{j+1}}, \ldots, P_{\alpha_k, Z_k}$ coincide with $P_{\beta,Z'}$,  as {\it unlabeled} lattice paths.  In other words, the curve $C_j$ is the unique curve in our building whose asymptotics correspond to two distinct  unlabeled lattice paths.

As a result of the last item of Claim \ref{cl:no_nodes_no_ghosts},  in combination with \eqref{eqn:sumind}, %
 we learn that one of the irreducible curves in our building must have index $1$,  and all other  irreducible curves must have index $0$.  
This conclusion about indices in turn implies that the curve $C_j$, whose asymptotics correspond to two distinct unlabeled lattice paths, must be irreducible.    Indeed, we can repeat the argument given in the proof of Claim \ref{cl:every_curve_irreducible} to conclude that if $C_j$ were not irreducible it would then have index $2$. 

Next, we claim that every irreducible curve in the building must be a cylinder, other than $C_j$.  To see this, let $P_{\alpha,Z}$ have $k$ edges and $P_{\beta,Z'}$ have $k'$ edges.  Remember that since the curve is irreducible, $\beta$ admits either only $1$ edge or $2$ distinct edges. Then, by the partition condition considerations mentioned above, the Euler characteristic of the building must be $-k - k'+2$.  Thus, the sum of the Euler characteristics of all curves in the building must also be $-k-k'+2$.  Each curve must have at least one positive end and one negative end, so the Euler characteristic of each curve must be non-positive.  And, since $C_j$ is irreducible, with ends at $k+k'$ distinct orbits, the Euler characteristic of $C_j$ must be at most $2-k-k'$. 
  It follows that every component other than $C_j$ must have Euler characteristic $0$, hence must be a cylinder; and, $C_j$ has exactly one end at each of its Reeb orbits.

For, $i\neq j$, the slice class is always zero, and therefore, $C_i$ consists of cylinders which are either trivial or local; see Remark \ref{rem:locality}.   We proved  above that each irreducible component of our building, and in particular these cylinders,  can have Fredholm index $0$ or $1$.  We will now show that the only non-trivial cylinders which could exist are local cylinders of Fredholm index $1$ with a positive end at a hyperbolic orbit and a negative end at an elliptic orbit.  First, observe that any local cylinder with both ends at the same orbit must in fact be trivial; see for example \cite[Prop 9.1]{Hutchings-index}.  Thus, a non-trivial cylinder must have one elliptic end and one hyperbolic end, and, by Formula \eqref{eqn:Fredholmindex}, for the Fredholm index to be non-negative, the positive and negative ends must be, respectively,  hyperbolic and elliptic.  The index will then be $1$.

Since the sum of the indices of the curves $C_i$ is  $1$, by \eqref{eqn:sumind}, %
we therefore learn that if the building is non-trivial, then it must have a single non-trivial index $1$ local cylinder, with a positive end at a hyperbolic orbit, and a negative end at an elliptic orbit, and the curve $C_j$ must have index $0$, and so it must have two negative ends, both at hyperbolic orbits, and all positive ends at elliptic orbits; all other curves in the building must be trivial cylinders.  Thus the building must have only two levels, and in the level structure, $C_j$ could be either on top or on bottom.   
In either case, as observed in \cite[Lemma 3.15]{Hutchings-Sullivan-Dehntwist}, there are two such canceling non-trivial local cylinders, so standard gluing arguments show that the mod $2$ count is unaffected by this degeneration.

To complete our proof of Lemma \ref{lem:comp}, it remains to prove Claim \ref{cl:no_nodes_no_ghosts}.
\begin{proof}[Proof of Claim \ref{cl:no_nodes_no_ghosts}]
We begin by showing that $C_i$ has no  non-constant irreducible component whose domain is closed.   
Let $z_{\min}$ and $z_{\max}$ be as in Claim \ref{cl:curves_dont_go_over_both_poles}.  According to Lemma \ref{lem:slice} and Remark \ref{rem:positivity_irreducible},  the curves $C_i$ are all contained in $$\mathcal{Y}= \{(s,t,\theta, z) \in \R \times \S^1 \times \S^2: z_{\min} \leq z \leq z_{\max}\}.$$
  If a non-constant closed curve was formed it would be contained in the above set  $ \mathcal{Y}$ and therefore, it would be null homologous because, by Claim \ref{cl:curves_dont_go_over_both_poles}, either $-1 < z_{\min}$ or $z_{\max} <1$.   However, closed non-trivial pseudo-holomorphic maps are never null homologous. 
  
  We next prove that every non-constant irreducible curve in our building must have  Fredholm index $0$ or $1$.  To see why, first note that Equation \eqref{eqn:Fredholmindex} applies to any such curve $C$;  since any such $C$ has both positive and negative ends, $C$ must have index at least $-1$, and the only way for $C$ to have index $-1$ is for $C$ to have exactly one negative end at a hyperbolic orbit, all positive ends at elliptic orbits, and $v = g = 0$. %
If $C$ is such a curve, then by Equation \eqref{eqn:Fredholmindex} the index of any somewhere injective curve that it covers must be $-1$: such somewhere injective curves do not exist, even in (generic) one parameter families of $J$, as observed in \cite[Lemma 3.15]{Hutchings-Sullivan-Dehntwist}, because this is the index before modding out by translation.

To finish the proof of the claim, it remains to show that the curves $C_i$ have no irreducible constant components and no nodal points, and then show that this implies \eqref{eqn:sumind}. %
To do so it will be helpful to visualize a nodal $J$-holomorphic curve by associating each nodal point $z$ with $i(z)$, and we can think of the two associated points as a single node.  Given a nodal $J$-holomorphic curve with domain $\Sigma$, we can form a new oriented topological surface $\widehat{\Sigma}$ by replacing each nodal point with a circle, and gluing these circles together; we refer the reader to \cite[Fig. 2.2]{otherwendl} for a helpful illustration, and we refer to \cite[Sec. 2]{otherwendl} for more details.  An important identity which we will use below is that
\begin{equation}
\label{eqn:node}
\chi(\widehat{\Sigma}) = \sum_i \chi(\Sigma_i) - N,
\end{equation}
where $\chi$ denotes the Euler characteristic, the sum is over all connected components $\Sigma_i$ of $\Sigma$, and $N$ denotes the number of nodal points, i.e.\ $N := \# \Delta$.  %
When we associate $\widehat{\Sigma}$ to a nodal curve $C$, which is an equivalence class, we write it as $\widehat{C}$; we should regard this as an equivalence class of topological surfaces, all of which are homeomorphic.

The significance of the construction of the surface $\widehat{\Sigma}$ for us is as follows. Let $C \in \mathcal{M}$ be some curve that is close to breaking into the building $(C_0, \ldots, C_k)$ in the sense of \cite{BEHWZ}.  The compactness result \cite{BEHWZ} then guarantees that  
\[ c_{\tau}(C) = \sum^k_{i=0} c_{\tau}(C_i), \quad \chi(C) = \sum^{k}_{i=0} \chi( \widehat{C_i} ).\]
In particular, using this equation, the additivity of the terms in \eqref{eqn:indexformula_Fredholm}, and the relation \eqref{eqn:node}, we learn from \eqref{eqn:node} that
\begin{equation}
\label{eqn:keyindexresult}
 \text{ind}(C) = \sum^k_{i=0} \text{ind}(C_i) + N.
 \end{equation}
 For later use, it will be convenient to call the term $\chi(C)$ the {\bf Euler characteristic of the building} $(C_0,\ldots,C_k);$ we similarly call $\text{ind}(C)$ the {\bf index of the building}.  These terms do not depend on the choice of $C   \in \mathcal{M}$ that is close to breaking into the building $(C_0, \ldots, C_k)$ 

We now conclude from \eqref{eqn:keyindexresult} that $N = 0$, in other words nodal degenerations can not occur; note that \eqref{eqn:sumind} will also be an immediate consequence of this in view of \eqref{eqn:keyindexresult}.  To see why $N = 0$, first observe that in the present situation, we have $\text{ind}(C) = 1$ in \eqref{eqn:keyindexresult}.  We also showed that the index is nonnegative for nonconstant curves; and, we know that $N$ is even. Hence, we will be done with the proof of the claim if we rule out the possibility of $C_i$ having  constant irreducible components.    In the present situation, the space $\mathcal{M}$ consists of $g = 0$ curves, so each $\Sigma$ must have $g = 0$, and hence any constant component must have at least three nodal points, and index $-2$ by \eqref{eqn:indexformula_Fredholm}.     Thus, any constant component must contribute at least $1$ to the right hand side of \eqref{eqn:keyindexresult}.   Since we have now shown that every component must contribute nonnegatively to the right hand side of \eqref{eqn:keyindexresult}, then it follows that there must be exactly one component that contributes $1$, and the rest must contribute $0$.  The component contributing $1$ to the sum can not be constant, since then the component would have to have exactly three nodes, and no other components could have any nodes at all; but the number of nodes is even.  It therefore follows, again using the fact that $N$ is even, that in fact we must have $N = 0$, and no constant components.  This completes the proof of the claim.
\end{proof}
\end{proof}

\subsubsection{Existence of $J$--holomorphic curves}\label{sec:curves_exist}  
In this section, we complete the proof of Lemma \ref{lem:converse} by showing that  $\# \mathcal{M}_{J_1}( (\alpha,Z) , (\beta,Z') ) =1$ for generic weakly admissible $J_1$.  

Below, we will introduce an open subset $X_1$ of $Y_\varphi$ satisfying the conditions of Remark \ref{rem:J_defined_on_X1}.  More specifically, it will satisfy the following property: There exists a compact subset $K \subset X_1$ such that for any weakly admissible almost complex structure $J$ on $\R \times Y_\varphi$, every $C \in \mathcal{M}_{J}( (\alpha,Z) , (\beta,Z') )$ is contained in $\R \times K$.   We will then prove the following lemma.

\begin{lemma}\label{lem:curves_exist}
  For generic weakly admissible almost complex structure  $J_1$ on $\R \times X_1,$   
 $$\#\mathcal{M}_{J_1}( (\alpha,Z) , (\beta,Z') ) =1.$$
\end{lemma}

We now explain why the above lemma completes the proof of  Proposition \ref{prop:model}.  Let $J_0$ be a generic and admissible almost complex structure on $Y_{\varphi}$, and connect its restriction to $K$ to $J_1$ via a  generic smooth path $J_s$ of weakly admissible almost complex structures.  Then, appeal to Lemma \ref{lem:comp}, Lemma \ref{lem:curves_exist} and Remark \ref{rem:J_defined_on_X1} to conclude that, counting mod 2, we have   
 \[ \# \mathcal{M}_{J_0}( (\alpha,Z) , (\beta,Z') ) = \# \mathcal{M}_{J_1}( (\alpha,Z) , (\beta,Z') ) = 1,\] which implies that   $\langle \partial (\alpha, Z), (\beta, Z')\rangle =1.$

\medskip

It remains to prove Lemma \ref{lem:curves_exist}.  This will occupy  the remainder of this section.   We begin by recalling certain preliminaries which will be used in the course of the proof.  Let $\Omega \subset \mathbb{R}^2$ be a subset of the first quadrant, and consider the set 
\[ X_{\Omega} := \lbrace (z_1,z_2) | \pi (|z_1|^2, |z_2|^2)  \in \Omega \rbrace \subset \mathbb{C}^2 = \mathbb{R}^4.\]
When $\Omega$ is the region bounded by the axes and the graph of a function $f$ with $f'' \leq 0$, then we call $X_{\Omega}$ a {\bf convex toric domain}.  When $\Omega$ is the region bounded by the axes and the graph of a function $f$ with $f'' \geq 0$, then we call $X_{\Omega}$ a {\bf concave toric domain}.  Much work has been done about these domains, see \cite{C-CG-F-H-R, Hutchings_beyond, CG_concave_convex}. 

The boundary $\partial X_\Omega$ of a concave or convex toric domain is a contact manifold  with the contact form $\lambda$ being the restriction to $\partial \Omega$ of the standard one-form on $\R^4$ 
\[ \frac{1}{2}(x_1 dy_1 - y_1 dx_1 + x_2 dy_2 - y_2 dx_2). \]
Recall from Section \ref{sec:prelim_Jcurves} that $\mathbb{R} \times \partial X_\Omega$  is referred to as  the {\it contact symplectization} and its symplectic form  is given by $\omega = d(e^s \lambda)$, where $s$ denotes the coordinate on $\R$.   In the argument below, we will be considering $J$--holomorphic curves in $\R \times \partial X_\Omega$ where $J$ is admissible for the SHS determined by the contact structure.  Here, we will refer to such $J$ as {\bf contact admissible}.  

\medskip

We now begin the proof of Lemma \ref{lem:curves_exist}.  

\begin{proof}
Recall the definitions of $P^{z_0}_{\alpha,Z},P^{z_0}_{\beta,Z'}$ from Lemma \ref{lem:slice}.   Let $z_{\min}$ and $z_{\max}$ be the minimum and maximum value of $z_0\in[-1, 1]$ such that $P^{z_0}_{\alpha,Z} - P^{z_0}_{\beta,Z'} \neq 0$.   In view of Claim  \ref{cl:curves_dont_go_over_both_poles}, either $-1 < z_{\min}$ or $z_{\max} < 1$.  We will treat the two cases separately.

\noindent \textbf{Case 1: $z_{\min}>-1$.} 
Define $X_1:= \{(t, \theta, z) \in \S^1 \times \S^2: -1< z \}.$ Let $J$ be a weakly admissible almost complex structure on $\R \times Y_\varphi$  and denote by  $M_J(\alpha, \beta; X_1)$ the set of %
$J$-holomorphic curves from $\alpha$ to $\beta$ with image contained in $\R \times X_1$,  modulo $\R$-translation.   We claim that 
 \begin{equation}\label{eqn:curves_in_X1}
 \mathcal{M}_{J}( (\alpha,Z) , (\beta,Z')) = M_J(\alpha, \beta; X_1).
\end{equation}  
 
 To prove the above, first note that any $J$--holomorphic curve from $(\alpha, Z)$ to $(\beta, Z')$ is contained in $\R \times X_1$ because we have $[C_z] \neq 0$ only if $ z \geq z_{\min}  > -1$; see Remark \ref{rem:positivity_irreducible}.   This proves $\mathcal{M}_{J}( (\alpha,Z) , (\beta,Z')) \subset M_J(\alpha, \beta; X_1)$.  As for the other inclusion, observe that, because $P_{\alpha, Z}$ is obtained from $P_{\beta, Z'}$ by corner rounding, $Z-Z'$ is an element of  $H_2(Y_\varphi, \alpha, \beta)$ which can be represented by a chain in $X_1$.  Moreover, it is the only such element of  $H_2(Y_\varphi, \alpha, \beta)$ as other elements are of the form  $[Z-Z'] + k [\S^2]$, where $k\in \Z$ is nonzero.  This proves $M_J(\alpha, \beta; X_1) \subset  \mathcal{M}_{J}( (\alpha,Z) , (\beta,Z'))$.
 
 Observe that the previous paragraph implies, in particular, that $X_1$ has the property stated in  Remark \ref{rem:J_defined_on_X1} with $K = \{(t, \theta, z) : z \geq z_{\min}\}$.

 We will next  identify $X_1$ with a subset of the concave toric domain $X_\Omega$ defined below.  Consider the concave toric domain $X_\Omega$, where $\Omega$ is the region bounded by the axes and the 
 graph of $f(x)$, where $f(x) := \frac{h(1-2x)}{2}$ for $0 \leq x \leq 1$. 
The boundary $\partial X_\Omega$ is a contact manifold  as described above.
 Consider the subset of  $\partial X_\Omega$ given by $$ X_2 : = \lbrace (z_1,z_2)  | \pi (|z_1|^2, |z_2|^2)  \in \partial \Omega- \{(1, 0)\} \rbrace. $$ Note that this is $\partial X_\Omega$ with a Reeb orbit removed. 
Define the mappings
\begin{align}
\label{eqn:parametrization_concave}
\begin{split}
\psi:\ &X_1 \to  X_2,\quad    (t,\theta,z) \mapsto \left(\tfrac12(1-z),\theta,\tfrac12{h(z)},2 \pi t\right),\\
\Psi:\ &\R \times X_1 \to \R \times X_2,\quad    (s,t,\theta,z) \mapsto \left(s,\tfrac12(1-z),\theta, \tfrac12{h(z)}, 2 \pi t\right).
\end{split}
\end{align}
 Here, we are regarding $\partial X_{\Omega} \subset \mathbb{C}^2$, and we are equipping $\mathbb{C}^2$ with coordinates $(\rho_1 := \pi |z_1|^2, \theta_1, \rho_2 := \pi |z_2|^2, \theta_2)$.   Note that the standard-one form $\lambda$, defined above, is given in these coordinates by
\begin{equation}
\label{eqn:standardpolar}
\lambda = \frac{1}{2\pi}(\rho_1 d \theta_1+\rho_2 d \theta_2).
\end{equation}
The above diffeomorphisms have the following properties: 

       \begin{enumerate}[(i)]
\item The Reeb vector field $R$ on $X_1$ pushes forward under $\psi$ to a positive multiple of the contact Reeb vector field $\hat R$ on $X_2$.  
\item The two-form $d\lambda$ on $X_2$ pulls back under $\psi$ to $\omega_{\varphi}$ on $X_1$.   Thus, the SHS $(\lambda, d\lambda)$ on $X_2$ pulls back under $\psi$ to the SHS $(\psi^* \lambda, \omega_\varphi)$ on $X_1$.  
\end{enumerate}
 Item (i) above holds because at a point $(x, \theta_1, f(x), \theta_2)$, $R'$ is a positive multiple of 
\[ -f'(x) \partial_{\theta_1} + \partial_{\theta_2},\] 
see for example Eq. 4.14 in \cite{Hutchings-Notes}, while 
\[ R= \partial_t + 2\pi h'(z) \partial_\theta,\] 
by combining \eqref{eqn:trivializedreeb} and \eqref{eqn:Hvf}, and $f'(x) = - h'(z).$  Item (ii) holds because $d\lambda$ is the restriction of
\[ \frac{1}{2\pi}(d\rho_1\wedge d \theta_1+d\rho_2\wedge d \theta_2),\]
which pulls back to
\[ \frac{1}{4 \pi} \left( d \theta \wedge dz + 2 \pi h'(z) dz \wedge dt \right),\]
which is exactly $\omega_{\varphi}.$

\medskip

Having established the above properties of the diffeomorphisms $\psi, \Psi$, we now proceed with the proof of Lemma \ref{lem:curves_exist}.  By property (i),  $\psi$ induces a bijection between the Reeb orbit sets of $R$ in $X_1$ and the Reeb orbit sets of $\hat R$ in $X_2$.  We will denote the induced bijection by  
$$\alpha \mapsto \hat \alpha.$$

Now, suppose $P_{\alpha, Z}$ is obtained from $P_{\beta, Z'}$ via rounding a corner and locally losing one $h$.  Let $\hat J$ be a contact admissible almost complex structure on the symplectization  $\R \times \partial X_\Omega$ and consider  $ \mathcal{M}_{\hat J}( \hat \alpha , \hat \beta)$ the space of $\hat J$-holomorphic currents $C$, modulo translation in the $\mathbb{R}$ direction,
which are asymptotic to $\hat \alpha$ as $s \to + \infty$ and $\hat \beta$ as $s \to - \infty$.  For a generic choice of contact admissible $\hat J$, this  moduli space  is finite and, as alluded to in Remark \ref{rem:ECH}, its mod $2$   cardinality determines\footnote{In general when defining the ECH differential, one would also demand the additional condition that the curves have ECH index $1$.  However, it follows from the index calculations in \cite{Choi} that this is automatic given the asymptotics $\hat{\alpha}, \hat{\beta}$.}
the ECH differential in the sense that
$$\langle \partial_{ECH} \, \hat \alpha, \hat \beta \rangle = \# \mathcal{M}_{\hat J}( \hat \alpha , \hat \beta).$$ 
As we will explain below, it follows from results proven in \cite{Choi} that the following hold:
\begin{enumerate}
\item[A1.] The image of every curve in $ \mathcal{M}_{\hat J}( \hat \alpha , \hat \beta)$  is contained in $\R \times X_2$. 
\item[B1.] The mod 2 count of curves in $ \mathcal{M}_{\hat J}( \hat \alpha , \hat \beta)$ is $1$.
\end{enumerate}
We will now explain why A1 \& B1 imply Lemma \ref{lem:curves_exist}. Indeed, define $J_1$ to be the almost complex structure on $\R \times X_1$  given by the pull back under $\Psi$ of the restriction 
of a generic $\hat J$ as above.  Then,  $J_1$ is an almost complex structure on $\R \times X_1$ which is admissible for the SHS given by $(\psi^* \lambda, \psi^* d\lambda) = (\psi^* \lambda, \omega_\varphi)$; this SHS gives the same orientation as $(dr,\omega_{\varphi})$, by property (i) above, 
which means that $J_1$ is weakly admissible, see Section \ref{sec:admissible_vs_weakly}.  Next, by property A1, $\Psi$ induces a bijection between $ \mathcal{M}_{\hat J}( \hat \alpha , \hat \beta)$ and $ \mathcal{M}_{J_1}( \alpha , \beta, X_1)$.  Property B1 
then implies that $\# \mathcal{M}_{J_1}( \alpha , \beta, X_1) = 1.$  Lastly, by Equation \eqref{eqn:curves_in_X1}, we have  $$ \# \mathcal{M}_{J_1}( (\alpha,Z) , (\beta, Z')) = 1,$$ which proves Lemma \ref{lem:curves_exist} in the case $z_{\min} > -1$.

Before elaborating on A1 \& B1, we will treat Case 2.

\medskip 

\noindent \textbf{Case 2: $z_{\max}<1$.}

The proof of Case 2 is similar, and in a sense dual, to that of Case 1.

\medskip

Define $X_1:= \{(t, \theta, z) \in \S^1 \times \S^2:  z< 1 \}.$ Let $J$ be a weakly admissible almost complex structure  on $\R \times Y_\varphi$  and denote by  $M_J(\alpha, \beta; X_1)$ the set of $J$ holomorphic  curves from $\alpha$ to $\beta$ with image contained in $\R \times X_1$, modulo translation.  As in Case 1, we have  
 \begin{equation}\label{eqn:curves_in_X1_case2}
 \mathcal{M}_{J}( (\alpha,Z) , (\beta,Z')) = M_J(\alpha, \beta; X_1).
\end{equation}  
 
 As before,  $X_1$ has the property stated in Remark \ref{rem:J_defined_on_X1} with $K = \{(t, \theta, z) : z \leq z_{\max}\}$.  We will next identify $X_1$ with a subset of the convex toric domain $X_\Omega$ defined below.
 
 Consider the convex toric domain $X_\Omega$, where $\Omega$ is the region bounded by the axes and the graph of $g(x) = \frac{h(1) - h(2x-1)}{2}$
for $0 \leq x \leq 1$.
The boundary $\partial X_\Omega$ is a contact manifold  as described above.
Consider the subset of  $\partial X_\Omega$ given by $$ X_2 : = \lbrace (z_1,z_2)  | \pi (|z_1|^2, |z_2|^2)  \in \partial \Omega- \{ (1,0) \} \rbrace. $$ 
Define the mappings

\begin{align}
\label{eqn:parametrization_convex}
\begin{split}
\psi:\ &X_1 \to  X_2,\quad (t,\theta,z) \mapsto \left( \tfrac12(z+1), \theta, \tfrac12(h(1) - h(z)), 2 \pi t \right),\\
\Psi:\ &\R \times X_1 \to \R \times X_2, \quad    (s,t,\theta,z) \mapsto \left( -s, \tfrac12(z+1), \theta, \tfrac12(h(1) - h(z)), 2 \pi t\right).
\end{split}
\end{align}
 Here, we are regarding $\partial X_{\Omega} \subset \mathbb{C}^2$, and we are equipping $\mathbb{C}^2$ with coordinates $(\rho_1 := \pi |z_1|^2, \theta_1, \rho_2 := \pi |z_2|^2, \theta_2)$.   
Similarly (but dual) to Case 1, these diffeomorphisms have the following properties: 
\begin{enumerate}[(i)]
\item The Reeb vector field $R$ on $X_1$ pushes forward under $\psi$ to a positive multiple of the contact Reeb vector field $\hat R$ on $X_2$.  

\item The two-form $d\lambda$ on $X_2$ pulls back under $\psi$ to $-\omega_{\varphi}$ on $X_1$.   Thus, the SHS $(\lambda, -d\lambda)$ on $X_2$ pulls back under $\psi$ to the SHS $(\psi^* \lambda, \omega_\varphi)$ on $X_1$.   
\end{enumerate}

 By property (i),  $\psi$ induces a bijection $$\alpha \mapsto \hat \alpha,$$ 
 between the Reeb orbit sets of $R$ in $X_1$ and the Reeb orbit sets of $\hat R$ in $X_2$.  Now, suppose $P_{\alpha, Z}$ is obtained from $P_{\beta, Z'}$ via rounding a corner and locally losing one $h$.  As in Case 1, let $\hat J$ be a contact admissible almost complex structure on  $\R \times \partial X_\Omega$ and consider the moduli space  $ \mathcal{M}_{\hat J}( \hat \beta , \hat \alpha)$ which for a generic choice of contact admissible $\hat J$  is finite and its mod $2$   cardinality determines\footnote{Just as in the previous case, the condition that the curves have ECH index $1$ is actually automatic here, by the asymptotics.} the ECH differential in the sense that
 $$\langle \partial_{ECH} \, \hat \beta, \hat \alpha \rangle = \# \mathcal{M}_{\hat J}( \hat \beta , \hat \alpha).$$ 
As we will explain below, it follows from results proven in \cite{Choi} that the following hold:
\begin{enumerate}
\item[A2.] The image of every curve in $ \mathcal{M}_{\hat J}( \hat \beta , \hat \alpha)$  is contained in $\R \times X_2$. 
\item[B2.] The mod 2 count of curves in $ \mathcal{M}_{\hat J}( \hat \beta, \hat \alpha)$ is $1$.
\end{enumerate}

We will now explain why A2 and B2 imply \ref{lem:curves_exist}.  Define $J_1$ to be the almost complex structure on $\mathbb{R} \times X_1$ given by the pull back under $\Psi$ of the restriction %
of $- \hat J$, for a generic $\hat J$ as above.  Then, $J_1$ is an almost complex structure on $\mathbb{R} \times X_1$ which is admissible for the SHS given by $(\psi^* \lambda, \psi^* (- d \lambda)) = (\psi^* \lambda, \omega_{\varphi})$; this SHS has the same orientation %
as $(dr,\omega_{\varphi})$, by property (i) above, which means that $J_1$ is weakly admissible; note that $-\hat{J}$ is admissible for the SHS given by $(-\lambda, - d \lambda)$.  Next, by property A2, $\Psi$ induces a bijection 
between $\mathcal{M}_{- \hat{J} }( - \hat{\beta}, - \hat{\alpha} )$ and $\mathcal{M}_{J_1}(\alpha,\beta,X_1)$ (here, the notation $- \hat{\beta}$ means that the orientation is reversed, which we can regard as viewing $\hat{\beta}$ as an orbit set for $-\lambda$); and, there is a canonical bijection\footnote{For the convenience of the reader, we remark that this bijection is familiar from the ``charge conjugation invariance" property of ECH, see for example \cite{Hutchings-Notes}.} between $\mathcal{M}_{- \hat{J} } (- \hat \beta, - \hat \alpha)$ and $\mathcal{M}_{ \hat J} (\hat \beta, \hat \alpha)$ 
given by associating a $\hat J$-holomorphic curve 
\[ u: (\Sigma,j) \to (X_2,\hat J)\] 
to the $-\hat J$ holomorphic curve 
\[ u: (\Sigma,-j) \to (X_2, -\hat J).\]

We can now repeat the rest of the argument given in Case 1 to  see that A2 \& B2 imply Lemma \ref{lem:curves_exist}. 

As promised, we will next elaborate on  properties A1, B1, from Case 1, and A2, B2, from Case 2.

\subsubsection{Elaborations on properties A1, B1 and A2, B2}\label{sec:elaborations}
For the benefit of the reader, we now elaborate on why A1, B1, A2, and B2 hold.

 We first explain items $A1$ and $A2$.  We showed above, while proving Equations \eqref{eqn:curves_in_X1} and \eqref{eqn:curves_in_X1_case2}, that as a consequence of Lemma \ref{lem:positivity} any $J$--holomorphic curve from $(\alpha, Z)$ to $(\beta, Z')$ is contained in $\R \times X_1$.  A reasoning similar to what we have given above, proves that items $A1, A2$ are consequences of \cite[Lemma 3.5]{Choi} which is analogous to our Lemma~\ref{lem:positivity}.

We now explain items $ B1$ and $B2$, beginning with some context.  The purpose of \cite{Choi} is to give a combinatorial realization of the ECH chain complex for any toric contact three-manifold.  Here, we use Choi's results in the special case of concave and convex toric domains.  

It turns out that in these cases, the ECH chain complex has a particularly nice form, as was originally conjectured by Hutchings \cite[Conj. A.3]{Hutchings_beyond}.

We first explain the ECH combinatorial model in the case where $X_\Omega$ is a  convex toric domain, which  appears in Case 2 above.  Let $X_{\Omega}$ be such a domain, described by a smooth function $f: [0,a] \to [0,b]$.  In this case, we can represent an ECH generator $\hat \alpha$ by an {\bf ECH convex lattice path} $P_{\hat \alpha}$: this is a piecewise linear lattice path that starts on the $y$-axis, ends on the $x$-axis, stays in the first quadrant,  and is the graph of a concave function; we usually orient this convex lattice path ``to the right".    Note that the region bounded by such a path and the coordinate axes is a convex region.

Hutchings conjectured (\cite[Conj. A.3]{Hutchings_beyond}), and Choi proved  (\cite[Cor. 4.15]{Choi})  that the ECH chain complex differential in this case has a particularly nice form: it is given by rounding a corner and locally losing one $h$.  More precisely, we can augment any convex lattice path to a closed lattice path by adding edges along the axes, and then there is a nonzero differential coefficient $\langle \partial_{ECH} \hat \beta, \hat \alpha \rangle$ if and only if this augmentation of $P_{\hat \alpha}$ is obtained from the augmentation of $P_{\hat \beta}$ by rounding a corner and locally losing one $h$

The ECH combinatorial model in the case where $X_\Omega$ is concave, which appears in Case 1 above, is dual to this.  We represent an ECH generator by an {\bf ECH concave lattice path}, which is defined analogously to the ECH convex lattice paths above, except that the path is the graph of a convex function instead; we augment by adding infinite rays along the axes. Then, by \cite[Prop. 4.14, Thm. 4.8]{Choi}, there is a nonzero differential coefficient $\langle \partial_{ECH} \hat \alpha, \hat \beta \rangle$ if and only if this augmentation of $P_{\hat \alpha}$ is obtained from the augmentation of $P_{\hat \beta}$ by rounding a corner and locally losing one $h$.

We can now explain the proofs of $B1$ and $B2$.   We begin with $B1$.  In the assignation between an ECH orbit set $\hat \alpha$ and a concave lattice path $P_{\hat \alpha}$ in the combinatorial model of \cite{Choi}, an edge $(p,q)$ corresponds to a point $r$ on the graph of $f$ with slope $q/p$, or to the two points $(0,a)$ and $(b,0)$ where the graph of $f$ meets the axes.  We also handle the elliptic orbits corresponding to $(0,a)$ and $(b,0)$ as in the PFH case: in the case of an $m$-fold cover of the elliptic orbit corresponding to $(0,a)$ (which is the only case relevant to the proof), we concatenate the line segment with slope $f'(0)$ and horizontal displacement $m$ to the beginning of the rest of the lattice path.   Now, this does not yield a lattice path, because $f'(0)$ is irrational, and so  we take the lowest lattice path above this concatenated path and then translate so that the path starts on the $y$-axis and ends on the $x$-axis.

We now describe how the identification of the previous section between PFH and ECH orbit sets $\alpha \mapsto \hat \alpha,$ translates to an identification of the corresponding lattice paths $P_{\alpha, Z} \mapsto P_{\hat{\alpha}}:$
We take $P_{\alpha, Z}$ and translate it so that it starts on the $x$-axis and ends on the $y$-axis, and reflect across the $y$-axis, to get $P_{\hat \alpha}$, the ECH lattice path corresponding to $\hat{\alpha}$.  Thus, since $P_{\alpha,Z}$ is obtained from $P_{\beta,Z'}$ by rounding a corner and locally losing one $h$, the ECH lattice path $P_{\hat{\alpha}}$ is obtained from $P_{\hat{\beta}}$ by rounding a corner and locally losing one $h$, and so the above results apply to prove B1.

The proof of $B2$ is similar.  The identification $\alpha \mapsto \hat \alpha$ translates to an identification of the corresponding lattice paths $P_{\alpha, Z} \mapsto P_{\hat{\alpha}}:$    given a PFH lattice path $P_{\alpha, Z}$, we translate it so that it begins on the $y$--axis and ends on the $x$--axis, we then reflect it across the $x$-axis to get $P_{\hat{\alpha}}$. As with B1, this now implies B2.
\end{proof}
 \begin{remark}
\label{rmk:choi}
To make our paper as self-contained as possible, while maintaining some amount of brevity, we sketch those parts of Choi's argument that are relevant for what we need.

The first point to make is a historical one.  Many of our arguments in Section \ref{sec:PFH_monotone_twist} are inspired by arguments in \cite{Choi}, which are themselves inspired by arguments in Hutchings-Sullivan; in particular, 
most of the ideas needed for Prop. 4.14 and Cor. 4.15 of \cite{Choi}  have already been presented here, although we presented them in the PFH case rather than the ECH case.  In particular, Choi proves analogues of all of the results in \ref{sec:comb} through \ref{sec:var}, and indeed his argument for reducing to considering regions given by locally rounding a corner and losing one $h$ is quite similar to the argument we give.  So, the only differences really worth commenting on involve how Choi shows that such a region actually carries a nonzero count of curves --- here there are two main differences between our argument and Choi's that we should highlight.  

\begin{enumerate}

\item  Choi finds his curves by referencing a paper by Taubes \cite{Taubes-beasts}, which works out various moduli spaces of curves for a particular contact form on $\S^1 \times \S^2$.  Taubes' contact form is Morse-Bott, so Choi does a perturbation as in [\cite{Hutchings-Sullivan-Dehntwist}, Lem. 3.17], [\cite{Hutchings-Sullivan-T3}, Thm. 11.11, Step 4] to break the Morse-Bott symmetry so as to obtain a nondegenerate contact form; Choi then does a deformation argument that is very analogous to what we do.  We remark that this strategy was pioneered by Hutchings and Sullivan in the series of papers \cite{Hutchings-Sullivan-Dehntwist, Hutchings-Sullivan-T3}. 

In contrast, we find our curves by referencing Choi's paper rather than Taubes'.  This means that we do not need any Morse-Bott argument.

\item Choi uses an inductive argument to reduce to considering moduli spaces of twice and thrice punctured spheres.  The reason he does this is to be able to use the above paper by Taubes, which does not directly address all the curves needed to analyze corner rounding operations, which could lead for example to curves with an arbitrary number of ends.  This induction works by using the fact that the differential is already known to square to $0$, by \cite{Hutchings-TaubesI, Hutchings-TaubesII} --- once one shows the result for the curves one gets from Taubes, given an arbitrary region that one wishes to show corresponds to a nonzero count of curves, one proves that if it did not have a nonzero count, the differential could not possibly square to zero, essentially by concatenating with a region that can be analyzed through the Taubes curves.  As with the above item, we again remark that this strategy was previously pioneered by Hutchings and Sullivan in \cite{Hutchings-Sullivan-Dehntwist, Hutchings-Sullivan-T3}.

In contrast, we have no need to use this inductive argument, since we can reference Choi's work for our needed curves.
\end{enumerate}

As should be clear from the above, the main reason we have chosen to take a slightly different tactic from Choi concerning these two points is for brevity --- we could have also used a strategy like the above, if for some reason we had wanted to.  Another point from the above is that the ideas in the part of Choi's paper relevant here already appear in the Hutchings-Sullivan papers \cite{Hutchings-Sullivan-Dehntwist, Hutchings-Sullivan-T3}: what is especially new and impressive in Choi's paper  
is the analysis of very general toric contact forms, (for example, the contact form on a toric domain that is neither concave nor convex) for which there could be curves corresponding to regions more general than those that come from rounding a corner and locally losing one $h$ --- however, we do not need to consider these particular kinds of contact forms in this paper.   
\end{remark}

\section{The Calabi property for positive monotone twists}\label{sec:calabi_for_montone_twists}
In this section we provide a proof of Theorem~\ref{thm:s2case}.  We will first use the combinatorial model of PFH, from the previous section, to compute the PFH spectral invariants for monotone twists; this is the content of Theorem \ref{theo:spec_computation}.  We will then use this computation to prove Theorem~\ref{thm:s2case}; this will be carried out in Section \ref{sec:proof_Calabi_property}.

\subsection{Computation of the spectral invariants}\label{sec:comuting_spec_invariant}

We will need to introduce some notations and conventions before stating, and proving, the main result of this section.   Throughout this section, we fix $\varphi$ to be a (smooth) positive monotone twist map of the disc. Recall from Remark \ref{rem:PFHspec-disc} that we define PFH spectral invariants for maps of the disc by identifying $\Diff_c(\D, \omega)$ with maps of the sphere supported in the northern hemisphere $S^+ \subset \S^2$, where the sphere $\S^2$ is equipped with the symplectic form $\omega=\frac1{4\pi}d\theta\wedge dz$. 

Recall that every  monotone twist map of the disc $\varphi$ can be written as the time--$1$ map of the flow of an autonomous Hamiltonian 
  \[ H = \frac{1}{2} h(z),\]
  where $h: \S^2 \rightarrow \R$ is a function of $z$ satisfying \[ h' \geq 0, h'' \geq 0, h(-1) = 0, h'(-1) = 0.\]
 For the main result of  this section, Theorem \ref{theo:spec_computation}, we will need to impose the additional assumption that 
  \begin{equation}\label{eq:additional_assump}
  h'(1)\in \N.
  \end{equation}
The reason for imposing the above assumption is that in our combinatorial model, Proposition \ref{prop:model}, the we assume that $h'(1)$  is assumed to be sufficiently close to $\lceil h'(1) \rceil$. Observe that every Hamiltonian $H$ as above can be $C^\infty$ approximated by the Hamiltonians considered in Section \ref{sec:ham}.

Although $\varphi$ is degenerate, we can still define the notion of a {\bf concave lattice path for $\varphi$} as any lattice path obtained from a starting point $(0, y)$, with $y\in\Z$, and a finite sequence of consecutive edges $v_{p_i,q_i}$, $i=0, \dots, \ell$, such that:
  \begin{itemize}
  \item $v_{p_i,q_i}=m_{p_i, q_i}(q_i, p_i)$ with $q_i, p_i$ coprime,
  \item the slopes $p_i/q_i$ are in increasing order,
  \item we have $0\leq p_0/q_0$ and $p_\ell/q_\ell$ is 
  at most $h'(1)$.
  \end{itemize}

If $p_0=0$, we will, as in Section \ref{sec:comb}, denote $v_-=m_-(1,0)= v_{p_0, q_0}$. If 
$p_\ell/q_\ell= h'(1)$, we will denote $v_+=m_+(1, h'(1) )= v_{p_l, q_l}.$ 
We also let $z_{p,q}$ be such that $h'(z_{p,q})=p/q$.

We can also define the action of such a path just as in Section \ref{sec:comb}:  We first define
\begin{equation}
\label{eqn:degenaction}
\mathcal{A}(v_-) = 0, \quad \mathcal{A}(v_+) = m_+\frac{h(1)}2,\quad  
\mathcal{A}(v_{p,q}) = \frac{m_{p,q}}{2} ( p (1-z_{p,q}) + q h(z_{p,q})).
\end{equation}
We then define the action of a concave lattice path $P$ to be
\begin{equation} \label{eq:action_formula_unlabelled}
\mathcal{A}(P) = y   +  \mathcal{A}(v_+) %
+ \sum_{v_{p,q}}   \mathcal{A}(v_{p,q}).
\end{equation}

The definition of $j(P)$ from Section \ref{sec:comb} (see Equation \eqref{eqn:index_path_unlabelled})  is still valid here.  With this in mind, we have the following:

\begin{theo}\label{theo:spec_computation}
\label{thm:comb}
Let $\varphi \in \Diff_c(\D, \omega)$ be a monotone twist satisfying the assumption \eqref{eq:additional_assump}. 
Then, for all integers $d>0$ and $k=d\mod 2$, 
\begin{equation}
\label{eqn:spectralformula}
c_{d,k}(\varphi) = \max \lbrace \mathcal{A}(P) : 2 j(P) - d  = k \rbrace,
\end{equation}
where the max is over all concave lattice paths $P$ for $\varphi$ of horizontal displacement $d$.  
\end{theo}

\begin{proof}
We can take a $C^{\infty}$ small perturbation of $\varphi$ to a $d$--nondegenerate Hamiltonian diffeomorphism $\varphi_0$ which itself is a {\it nice perturbation} of some $\varphi_H^1$, where $H \in \mathcal{D}$ as in Section \ref{sec:ham}.

Since $c_{d,k}(\varphi)$ is the limit of $c_{d,k}(\varphi_0)$, as we take smaller and smaller perturbations, it suffices to show that the analogous formula to \eqref{eqn:spectralformula} holds for  $\varphi_0$.  In other words, we wish to show
\begin{equation}
\label{eqn:perturbedformula}
c_{d,k}(\varphi_0) = \max \lbrace \mathcal{A}(P_{\alpha,Z}) : I(P_{\alpha,Z}) = k \rbrace,
\end{equation}
where the max is over all concave lattice paths of horizontal displacement $d$.

To prove \eqref{eqn:perturbedformula}, given $(d,k)$, consider the element $\sigma$ of the PFH chain complex for $\varphi_0$ given by
\begin{equation*}
\sigma := \sum (\alpha,Z)
\end{equation*}
where the sum is over all PFH generators $(\alpha,Z)$ where $\alpha$ consists of only elliptic orbits, is of degree $d$ and $I(\alpha, Z) = k$.  Equivalently, the corresponding concave lattice path $P_{\alpha,Z}$ has edges which are all labelled $e$, it has degree $d$  and  index $k$.  

We first claim that $\sigma$  is in the kernel of the PFH differential.  Indeed, by Proposition \ref{prop:model},   the differential is the mod $2$ sum over every $(\beta, Z')$ such that $P_{\alpha,Z}$ can be obtained from $P_{\beta,Z'}$ by rounding a corner and locally losing one $h$.  Fix one such $P_{\beta,Z'}$.  It has exactly one edge labelled $h$ and so there are exactly two concave paths, say $P_{\alpha,Z}$ and  $P_{\tilde \alpha, \tilde Z}$, which are obtained from $P_{\beta,Z'}$ by rounding a corner and locally losing one $h$.  The two paths $P_{\alpha,Z}$ and  $P_{\tilde \alpha, \tilde Z}$  are different, because for example when you round the two corners for an edge, one rounding contains one of the corners, and the other contains the other corner.  Now, $(\alpha, Z)$ and $(\tilde \alpha, \tilde Z)$ both contribute to $\sigma$ and thus, $(\beta,Z')$ appears exactly twice in the differential of $\sigma$; hence, its mod $2$ contribution to the differential is zero.  Consequently, we see that  $\sigma$ is in the kernel of the PFH differential.

Now, by Proposition~\ref{prop:model}, no concave path with all edges labeled $e$ is ever in the image of the differential, because the concave path corresponding to the negative end of a holomorphic curve counted by the differential has more edges labeled $h$ than the concave path corresponding to the positive end, and in particular has at least one edge labeled $h$.  So, $[\sigma] \ne 0$ in homology.  In fact, $\sigma$ must carry the spectral invariant for similar reasons.  Specifically, if there is some other chain complex element $\sigma'$ homologous to $\sigma$, then $\sigma + \sigma'$ must be in the image of the differential.  Nothing in the image of the differential has a path with all edges labeled by $e$, so $\sigma'$ must contain all possible paths of degree $d$ and index $k$ with all edges labeled by $e$, and so its action must be at least as much as $\sigma$.  

In view of the above paragraph, to prove  Equation \eqref{eqn:perturbedformula}, we must show that the supremum in the equation is attained by a path whose edges are labelled $e$.  To see this, consider  a path of degree $d$ and index $k$  with some edges labelled $h$.   As a consequence of the combinatorial index formula in Proposition \ref{prop:model},  the number of edges labelled $h$ must be even; denote this number by $2r$.  If we  round $r$ corners
and remove all $h$ labels, we obtain a path of the same grading all of whose edges are labelled $e$.  Now, the newly obtained path has larger action because the corner rounding operation increases action. This completes the proof.
\end{proof}

\begin{remark}\label{rem:monotonicink}
As a corollary to Theorem~\ref{thm:comb}, we find that 
\[ c_{d,k} \leq c_{d,k+2},\]
as promised in Equation \eqref{eqn:monotonicink}.  
Indeed, if we take an action maximizing  path of index $k$, with all edges labeled $e$ as above, and round a corner, then the grading increases by $2$, and the action increases as well.  
\end{remark}

\subsection{Proof of Theorem \ref{thm:s2case} } \label{sec:proof_Calabi_property}
We now give a proof of Theorem~\ref{thm:s2case}.  It is sufficient to prove Theorem \ref{thm:s2case} for monotone twists $\varphi$ which can be written as $\varphi^1_H$ with the Hamiltonian $H$ satisfying \eqref{eq:additional_assump}.  This is because the left and right hand sides of Equation \ref{eqn:calabiconjecture}, i.e.\ the Calabi invariant and the PFH spectral invariants, are (Lipschitz) continuous with respect to the Hofer norm and, moreover,  every monotone twist can be approximated, in the Hofer norm, by monotone twists satisfying  \eqref{eq:additional_assump}.  Hence, we will suppose for the rest of this section that our monotone twists $\varphi$ satisfy \eqref{eq:additional_assump}.  This allows us to apply Theorem \ref{theo:spec_computation}.

Our proof relies on a version of the isoperimetric inequality for non-standard norms which we now recall; the idea of using this inequality is inspired by Hutchings' proof in \cite[Section 8]{QECH} of the ``Volume Property" for ECH spectral invariants for certain toric contact forms.

Let $\Omega \subset \R^2$ be a convex compact connected subset. Using the standard Euclidean inner product, the dual norm associated to $\Omega$, denoted $||  \cdot ||^*_{\Omega}$, is defined for any $v\in\R^2$ by
\begin{equation}
\label{eqn:normmax}
|| v ||^*_{\Omega} = \max \{ v \cdot w: w \in \partial \Omega\}.
\end{equation}
Let $\Lambda \subset  \R^2$ be an oriented, piecewise smooth curve and denote by $\ell_\Omega(\Lambda)$ its length measured with respect to $|| \cdot ||^*_{\Omega}$.  
When $\Lambda$ is closed, its length remains unchanged under translation of $\Omega$.

For our proof, we will suppose that $\Omega$ is the region bounded by the graph of $h$, the horizontal line through $(1, h(1))$, and the vertical line through $(-1, 0)$.  Denote by $\hat{\Omega}$ the region obtained by rotating $\Omega$ clockwise by ninety degrees; see Figure \ref{fig:isoperimetric1}.  We orient the boundary $\partial \hat \Omega$ counterclockwise with respect to any point in its interior.

 \begin{figure}[h!]
 \centering 
 \def\svgwidth{0.6\textwidth}
\begingroup%
  \makeatletter%
  \providecommand\color[2][]{%
    \errmessage{(Inkscape) Color is used for the text in Inkscape, but the package 'color.sty' is not loaded}%
    \renewcommand\color[2][]{}%
  }%
  \providecommand\transparent[1]{%
    \errmessage{(Inkscape) Transparency is used (non-zero) for the text in Inkscape, but the package 'transparent.sty' is not loaded}%
    \renewcommand\transparent[1]{}%
  }%
  \providecommand\rotatebox[2]{#2}%
  \newcommand*\fsize{\dimexpr\f@size pt\relax}%
  \newcommand*\lineheight[1]{\fontsize{\fsize}{#1\fsize}\selectfont}%
  \ifx\svgwidth\undefined%
    \setlength{\unitlength}{252.20818545bp}%
    \ifx\svgscale\undefined%
      \relax%
    \else%
      \setlength{\unitlength}{\unitlength * \real{\svgscale}}%
    \fi%
  \else%
    \setlength{\unitlength}{\svgwidth}%
  \fi%
  \global\let\svgwidth\undefined%
  \global\let\svgscale\undefined%
  \makeatother%
  \begin{picture}(1,0.56074227)%
    \lineheight{1}%
    \setlength\tabcolsep{0pt}%
    \put(0,0){\includegraphics[width=\unitlength,page=1]{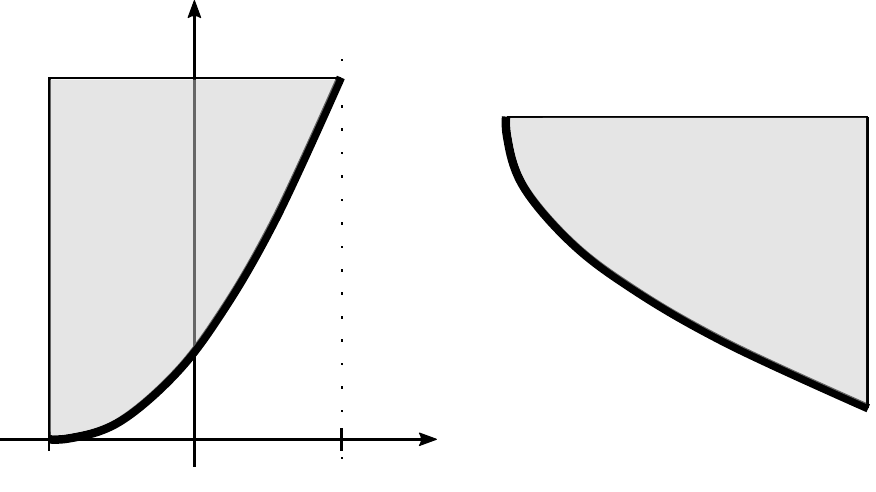}}%
    \put(0.39754516,0.00576704){\color[rgb]{0,0,0}\makebox(0,0)[lt]{\lineheight{1.25}\smash{\begin{tabular}[t]{l}$1$\end{tabular}}}}%
    \put(0.06345542,0.0068675){\color[rgb]{0,0,0}\makebox(0,0)[lt]{\lineheight{1.25}\smash{\begin{tabular}[t]{l}$-1$\end{tabular}}}}%
    \put(0.27606682,0.19868948){\color[rgb]{0,0,0}\makebox(0,0)[lt]{\lineheight{1.25}\smash{\begin{tabular}[t]{l}$\mathrm{graph}(h)$\end{tabular}}}}%
    \put(0.13291868,0.31623227){\color[rgb]{0,0,0}\makebox(0,0)[lt]{\lineheight{1.25}\smash{\begin{tabular}[t]{l}$\Omega$\end{tabular}}}}%
    \put(0.22792498,0.00433191){\color[rgb]{0,0,0}\makebox(0,0)[lt]{\lineheight{1.25}\smash{\begin{tabular}[t]{l}$0$\end{tabular}}}}%
    \put(0.22715157,0.48253377){\color[rgb]{0,0,0}\makebox(0,0)[lt]{\lineheight{1.25}\smash{\begin{tabular}[t]{l}$h(1)$\end{tabular}}}}%
    \put(0.76383531,0.29782186){\color[rgb]{0,0,0}\makebox(0,0)[lt]{\lineheight{1.25}\smash{\begin{tabular}[t]{l}$\hat{\Omega}$\end{tabular}}}}%
    \put(0,0){\includegraphics[width=\unitlength,page=2]{isoperimetric1.pdf}}%
  \end{picture}%
\endgroup%
\caption{The convex subsets $\Omega$, $\hat\Omega$.}
 \label{fig:isoperimetric1}
\end{figure}

\begin{proof}[Proof of Theorem \ref{thm:s2case}]
Let $P$ be a concave lattice path of horizontal displacement $d$ for $\varphi$.  Complete the path $P$ to a closed, convex polygon by adding a vertical edge at the beginning of $P$ and a horizontal edge at the end; orient this polygon counterclockwise, relative to any point in its interior; and, rotate it clockwise by ninety degrees.  Call the resulting path $\Lambda$; see Figure \ref{fig:isoperimetric2}. We will need the following lemma.  

 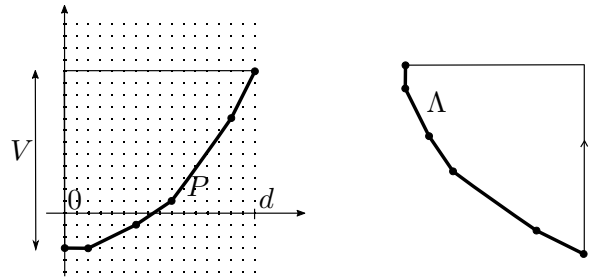
\begin{figure}[h!]
 \centering 
 \def\svgwidth{0.6\textwidth}
\begingroup%
  \makeatletter%
  \providecommand\color[2][]{%
    \errmessage{(Inkscape) Color is used for the text in Inkscape, but the package 'color.sty' is not loaded}%
    \renewcommand\color[2][]{}%
  }%
  \providecommand\transparent[1]{%
    \errmessage{(Inkscape) Transparency is used (non-zero) for the text in Inkscape, but the package 'transparent.sty' is not loaded}%
    \renewcommand\transparent[1]{}%
  }%
  \providecommand\rotatebox[2]{#2}%
  \newcommand*\fsize{\dimexpr\f@size pt\relax}%
  \newcommand*\lineheight[1]{\fontsize{\fsize}{#1\fsize}\selectfont}%
  \ifx\svgwidth\undefined%
    \setlength{\unitlength}{363.81397853bp}%
    \ifx\svgscale\undefined%
      \relax%
    \else%
      \setlength{\unitlength}{\unitlength * \real{\svgscale}}%
    \fi%
  \else%
    \setlength{\unitlength}{\svgwidth}%
  \fi%
  \global\let\svgwidth\undefined%
  \global\let\svgscale\undefined%
  \makeatother%
  \begin{picture}(1,0.47341963)%
    \lineheight{1}%
    \setlength\tabcolsep{0pt}%
    \put(0,0){\includegraphics[width=\unitlength,page=1]{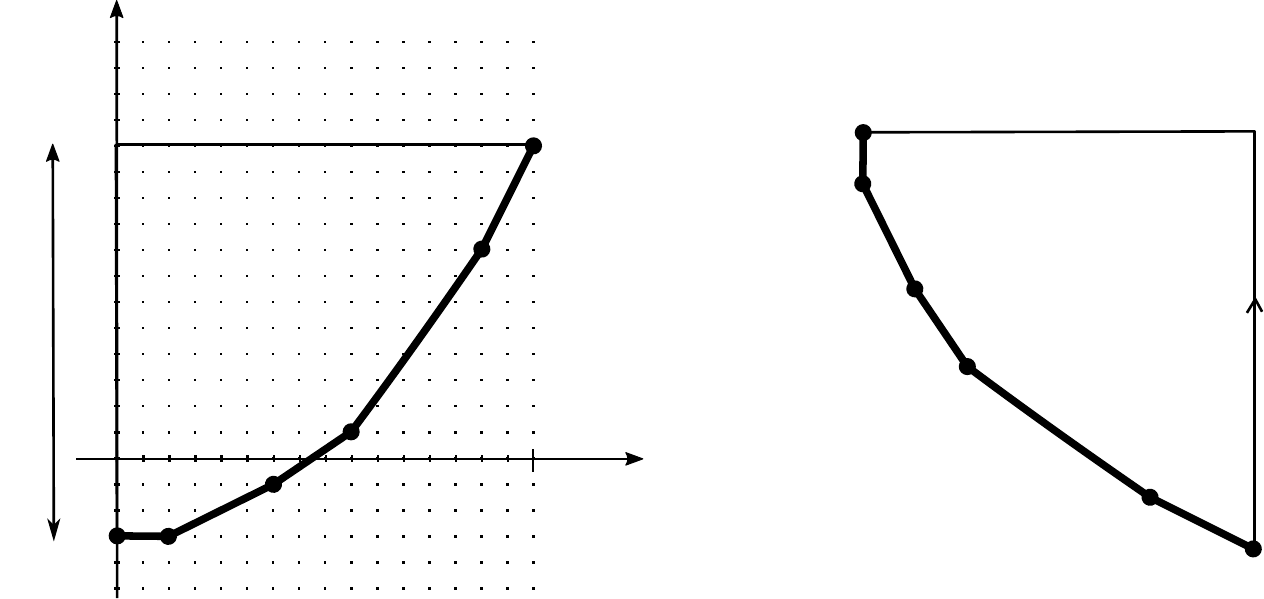}}%
    \put(0.71950065,0.28718763){\color[rgb]{0,0,0}\makebox(0,0)[lt]{\lineheight{1.25}\smash{\begin{tabular}[t]{l}$\Lambda$\end{tabular}}}}%
    \put(0.30393307,0.14071302){\color[rgb]{0,0,0}\makebox(0,0)[lt]{\lineheight{1.25}\smash{\begin{tabular}[t]{l}$P$\end{tabular}}}}%
    \put(0.42930014,0.11869246){\color[rgb]{0,0,0}\makebox(0,0)[lt]{\lineheight{1.25}\smash{\begin{tabular}[t]{l}$d$\end{tabular}}}}%
    \put(-0.00174658,0.20313579){\color[rgb]{0,0,0}\makebox(0,0)[lt]{\lineheight{1.25}\smash{\begin{tabular}[t]{l}$V$\end{tabular}}}}%
    \put(0.09771021,0.1182685){\color[rgb]{0,0,0}\makebox(0,0)[lt]{\lineheight{1.25}\smash{\begin{tabular}[t]{l}$0$\end{tabular}}}}%
  \end{picture}%
\endgroup%
\caption{The path $P$ and the closed path $\Lambda$.}
\label{fig:isoperimetric2}
\end{figure}

\begin{lemma}\label{lem:length_action}  The following identities hold:
\begin{enumerate}
\item $\ell_{\Omega}(\partial\hat \Omega) = 2 (2h(1) - I)$, where $I : = \int_{-1}^1 h(z) dz$.

\item $\ell_{\Omega}(\Lambda) = dh(1) + 2y + 2 V  - 2 \mathcal{A}(P)$,
 where $V$ denotes the vertical displacement of the path $P$. 
\end{enumerate}
\end{lemma}
\begin{proof}[Proof of Lemma \ref{lem:length_action}]
According to the Isoperimetric Theorem \cite{Brothers-Morgan}, for any  simple closed curve $\Gamma$, we have 
\begin{equation}
\ell_{\Omega}^2(\Gamma) \geq 4 A(\Omega) A (\Gamma),\label{eq:isoperimetric}
\end{equation}
 where $A(\Omega)$ and $A(\Gamma)$ denote the Euclidean areas of $\Omega$ and the region bounded by $\Gamma$, respectively.  Moreover, equality holds when $\Gamma$ is a scaling of a ninety degree clockwise rotation of $\partial \Omega$;  see \cite[Example 8.3]{QECH}.   The first item follows immediately from the equality case of the theorem applied to $\Gamma = \partial \hat \Omega$ because $A(\Gamma) = A(\Omega) = 2h(1) - I$. 
Alternatively, Item  1 could be obtained via direct computation.

We now prove the second item.  The length of the polygon $\Lambda$ is given by the sum 
$
\sum_{e \in \Lambda}  || e ||^*_{\Omega},$
where the sum is taken over the edge vectors $e$ of $\Lambda$.  It follows from the method of Lagrange multipliers that  $$|| e ||^*_{\Omega} = e \cdot p_e,$$ for some point $p_e \in \partial \Omega$ where $e$ points in the direction of the outward normal cone at $p_e$.  Hence, we can write 
 \begin{equation}
\label{eq:length}
\ell_{\Omega}(\Lambda) = \sum_{e \in \Lambda} e \cdot p_e  =   \sum_{e \in \Lambda} e \cdot ( p_e - m),
\end{equation}
where the second equality holds, for any $m \in \R^2$, because $\Lambda$ is closed.  
We will calculate $\ell_{\Omega}(\Lambda)$  using the choice $m=(1,0)$. 
  
   Let $e$ denote one of the edges of $\Lambda$ corresponding to a vector $v_{p,q} = m_{p, q} (q,p) $ in $P$.    Now, we have  $e = m_{p,q} (p, -q),$ since we are taking a ninety degree clockwise rotation; moreover, $ p_e - m  = (z_{p,q} -1, h(z_{p,q}))$.  Thus, 
\begin{equation*}
\begin{split}
e \cdot (p_e -m) =&  m_{p,q} (p, -q) \cdot (z_{p,q} -1, h(z_{p,q}))  \\
 =& m_{p,q} \left( p (z_{p,q} - 1 ) - q h(z_{p,q}) \right) \\
 =& - 2  \mathcal A ( v_{p,q} ),
 \end{split}
\end{equation*}
where the final equation follows from \eqref{eqn:degenaction}.  

If $e$ is an edge  of $\Lambda$ corresponding to either of the vectors $v = v_- = m_- (1,0)$ or $v = v_+ = m_+ (1, \lceil h'(1) \rceil)$, then a similar computation to the above yields
$e \cdot (p_e -m) = - 2  \mathcal A ( v )$.   Summing over all of the edges $e$ of $\Lambda$, corresponding to vectors in $P$, we obtain the quantity 
\begin{equation}\label{eq:action_1}
2y- 2 \mathcal{A} (P).
\end{equation}    

The remaining two edges of $\Lambda$ are the vectors $e_1 = (0, d )$ and $e_2 = (-V,0)$ for which we have 
\begin{equation}\label{eq:action_2}
\begin{split}
e_1 \cdot (p_{e_1} -m) &= (0, d) \cdot (-1, h(1)) = dh(1),  \\
e_2 \cdot (p_{e_2} -m) & =(- V, 0) \cdot (-2, 0) = 2V  .
 \end{split}
\end{equation}

We obtain from Equations \eqref{eq:length}, \eqref{eq:action_1}, and \eqref{eq:action_2} that $\ell_{\Omega}(\Lambda) = dh(1) + 2y + 2V - 2 \mathcal{A}(P)$. 
\end{proof}

\medskip 

\noindent {\em Step 1: Calabi gives the lower bound.}  Here, we will prove the lower bound needed for establishing Equation \eqref{eqn:calabiconjecture}.  In other words, we will show that for any sequence $(d,k)$, such that $d\to \infty$, we have
\begin{equation}\label{eqn:lower_bound}
\Cal(\varphi) \leq \liminf_{d \to \infty} \left( \frac{c_{d,k}(\varphi)}{d} - \frac12\frac{k}{d^2 + d} \right).
\end{equation}

 To prove the above, fix $\eps > 0$.  We will show that for all sufficiently large positive integers $d$, there exists a sequence of concave lattice paths $\{P_{\eps,d}\}$,
 such that 
\begin{equation} \label{eqn:lowerbound_for_Pd}
\left| \Cal(\varphi) - \left( \frac{ \mathcal{A}(P_{\eps,d})}{d} - \frac12\frac{k_d}{d^2 + d} \right) \right| \leq \eps,
\end{equation} %
where $k_d = 2j(P_{\eps,d}) - d $ denotes the combinatorial index of $P_{\eps,d}$.  
By Theorem \ref{theo:spec_computation}, we have $\mathcal{A}(P_{\eps,d}) \leq c_{d, k_d} (\varphi)$, and, by the argument we explained in Section \ref{sec:PFH_monotone_twist},  see the discussion after Theorem \ref{thm:s2case}, 
proving \eqref{eqn:lower_bound} for one sequence $\{(d,k)\}$ with $d$ ranging across all sufficiently large positive integers
proves it for all such sequences, and so we conclude \eqref{eqn:lower_bound} from the above, since $\eps$ was arbitrary.

\medskip

We now turn to the description of the concave paths $P_{\eps,d}$.  Let $P$ be a concave path approximating the graph of $h$ such that it begins at $(-1, 0)$, ends on the line $x = 1$, is piecewise linear, and its vertices are rationals with numerator an even integer and denominator $d$.  Let $\Lambda$ be the convex polygon  obtained as follows:   Add a vertical edge at the beginning of $P$ and a horizontal edge at the end of it; orient this polygon counterclockwise, relative to any point in its interior; and, rotate it clockwise by ninety degrees.   The convex polygon $\Lambda$ approximates $\partial \hat \Omega$. More precisely, given $\eps$, we pick, for all sufficiently large positive integers $d$, paths $P$, subject to the conditions above, and such that 
\begin{enumerate}[(A)]
\item $P$ is within $\eps$ of the graph of $h$,
\item $ \vert \ell_\Omega( \Lambda)  - \ell_\Omega(\partial \hat \Omega) \vert \leq \eps$ which by Lemma \ref{lem:length_action} is equivalent to  $$
  \vert \ell_\Omega( \Lambda)  -  2 (2h(1) - I) \vert \leq \eps,$$
\item The area of the region under the path $P$, and above the $x$--axis, is within $\eps$ of $I$.
\end{enumerate}

Let $P_{\eps,d}, \Lambda_{\eps,d}$ be the images of $P, \Lambda$, respectively, under the mapping $$(x, y) \mapsto \frac {d}{2}(x + 1, y).$$   
The path $P_{\eps,d}$ is a concave lattice path of degree $d$. Recall that  $\Cal(\varphi) = \frac{1}{4}I$.  We will prove the two inequalities below,  
which will imply Equation \eqref{eqn:lowerbound_for_Pd}:
\begin{equation}\label{eqn:toconsider} 
\begin{split}
\left|   \frac{ \mathcal{A}(P_{\eps,d} )}{d} - \frac I2 \right| \leq  \frac{3\eps}4, \\
 \left|  \frac12\frac{k_d}{d^2 + d} - \frac{I}{4} \right|\leq \frac{\eps}{4}. 
 \end{split}
\end{equation}

We first examine the term  $ \frac{ \mathcal{A}(P_{\eps,d} )}{d} $.  By Lemma \ref{lem:length_action}, and using the fact that $\ell_\Omega(\Lambda_{\eps,d}) = \frac{d}{2} \ell_\Omega(\Lambda)$, we obtain
\begin{align*}
\frac{ \mathcal{A}(P_{\eps,d} )}{d}  =&   \frac{dh(1) + 2 V - \ell_\Omega(\Lambda_{\eps,d})}{2d} \\
      =& \frac{h(1)}2 + \frac{V}{d} - \frac{\ell_\Omega(\Lambda)}{4}.
\end{align*} 

By item $(A)$ above, the term $\frac{V}{d}$ is within $\frac\eps2$ of $\frac{h(1)}2$.  And, by item $(B)$ above, the term $\ell_{\Omega}(\Lambda)$ is within $\eps$ of $2(2h(1) - I),$ hence the first inequality in \eqref{eqn:toconsider}. 

As for the second inequality, we know from the index computations of Section \ref{sec:index} that, up to an error of $O(d)$,  the index $k_d$ is twice the area between the $x$--axis and the path $P_{\eps,d}$.  Because $P_{\eps,d}$ is a scaling of $P$ by a factor of $\frac{d}{2}$,  item (C) above implies
$$ - \frac{d^2}{2} \eps + O(d) \leq k_d - \frac{d^2}{2} I \leq \frac{d^2}{2} \eps + O(d),  $$
which  for sufficiently large $d$ yields the second inequality in \eqref{eqn:toconsider}.  
\end{proof}

\noindent {\em Step 2:  Calabi gives the upper bound.}

We now complete the proof of Theorem~\ref{thm:s2case}.  We emphasize again, for the convenience of the reader, that as mentioned in Remark \ref{rem:only_need_inequality}, we do not actually need this step of the proof for the proof of our main result Theorem~\ref{thm:main}.

To complete the proof, we need to show that
\begin{equation}\label{eqn:upper_bound}
\Cal(\varphi) \geq \limsup_{d \to \infty} \left( \frac{ c_{d,k}(\varphi)}{d} - \frac12\frac{k}{d^2 + d} \right).
\end{equation}
To do this, we will show that 
\begin{equation}\label{eqn:upper_bound2}
\Cal(\varphi) \geq \limsup_{d \to \infty} \left( \frac{ \mathcal{A}(P)}{d} - \frac12\frac{k}{d^2 + d} \right)
\end{equation}
for all degree $d$  concave lattice paths $P$ of combinatorial index $k$.

Let $P$ be a concave lattice path of degree $d$ and combinatorial index $k$. Let $E$ be a real number with $E> h(1)$  and $E>\frac{2V}d$, and $\Omega_E$ be the compact convex subset of $\R^2$ bounded by the graph of $h$, the vertical segments $\{-1\}\times [0,E]$ and $\{1\}\times [h(1),E]$, and the horizontal segment $[0,d]\times\{E\}$.  For example, with the notations of Lemma \ref{lem:length_action}, we have $\Omega=\Omega_{h(1)}$. The inequality (\ref{eqn:upper_bound2}) will follow from letting $E$ tend to $\infty$ after applying the isoperimetric inequality (\ref{eq:isoperimetric}) to the domain $\Omega_E$ and to the following curve $\Lambda_E$.

To define $\Lambda_E$, consider the region delimited by our lattice path $P$, the vertical segments $\{0\}\times [y,y+\frac{dE}2]$ and $\{d\}\times [y+V,y+\frac{dE}2]$, and the horizontal segment $[0,d]\times\{y+\frac{dE}2\}$. Our curve $\Lambda_E$ is the boundary of this region, rotated clockwise by ninety degrees; see Figure \ref{fig:isoperimetric3}.

 \begin{figure}[h!]
 \centering 
 \def\svgwidth{1\textwidth}
 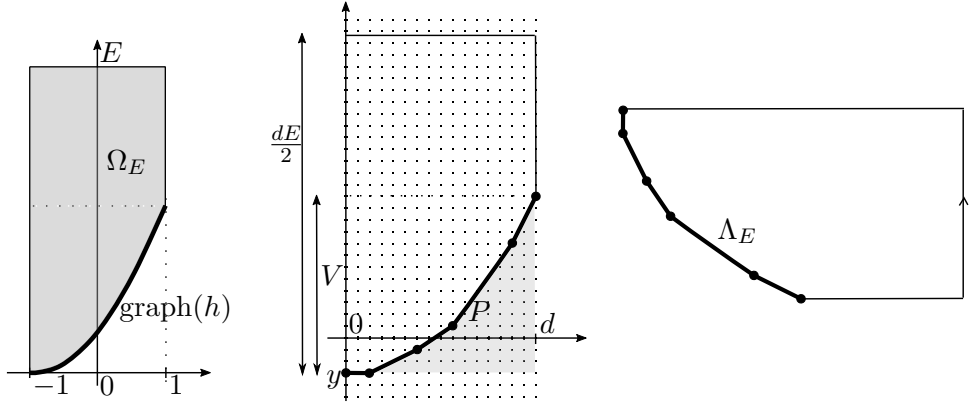\caption{The convex subset $\Omega_E$ and the path $\Lambda_E$.}\label{fig:isoperimetric3}
\end{figure}

The isoperimetric inequality (\ref{eq:isoperimetric}) gives:
\begin{equation}\label{eq:isoperimetric2}
\ell_{\Omega_E}^2(\Lambda_E)\geq 4 A(\Omega_E)A(\Lambda_E).
\end{equation}

The area factors are easily computed. We have:
\[A(\Omega_E)=2E-I,\quad A(\Lambda_E)=\tfrac12d^2E-a(P),\]
where $a(P)$ denote the area of the region between $P$, the horizontal segment $[0,d]\times \{y\}$ and the vertical segment $\{d\}\times[y,y+V]$ (in grey on Figure \ref{fig:isoperimetric3}).
Moreover, a computation similar to that of item 2 in Lemma \ref{lem:length_action}, gives 
\[\ell_{\Omega_E}(\Lambda_E)=2dE +2y-2\mathcal{A}(P).\]
Thus, (\ref{eq:isoperimetric2}) gives
\[(2dE+2y-2\mathcal{A}(P))^2\geq 4(2E-I)(\tfrac12d^2E-a(P)).\]
After simplification, we obtain
\[-2dE\mathcal{A}(P)+\mathcal{A}(P)^2+ 2dEy-2y\mathcal{A}(P) +  y^2 \geq -2a(P)E-\tfrac12d^2EI+Ia(P).\]
Dividing by $2d^2E$ and letting $E$ go to $+\infty$ then yields:
\[\frac14I\geq \frac{\mathcal{A}(P)}{d}-\frac{a(P)+dy}{d^2}.\]
Now $2a(P)+2dy$ corresponds to $k$ up to an error $O(d)$. Thus, for all sequence of paths $P$ of degree $d$ and index $k$,
\[\frac14I\geq \limsup_{d \to \infty} \left( \frac{ \mathcal{A}(P)}{d} - \frac12\frac{k}{d^2 + d} \right),\]
from which (\ref{eqn:upper_bound2}) follows.

\section{Discussion and open questions}\label{sec:questions}
We discuss here some open questions relating to the main results of our article.

\subsection{Simplicity on other surfaces}\label{sec:questions_simplicity}
Let $M$ denote a closed manifold equipped with a volume form $\omega$ and denote by $\Homeo_0(M, \omega)$ the identity component in the group of volume-preserving homeomorphisms of $M$.  In \cite{fathi}, Fathi constructs the {\bf mass-flow homomorphism}
$$\mathcal{F}: \Homeo_0(M, \omega) \rightarrow H_1(M)/ \Gamma,$$ 
mentioned above, where $H_1(M)$ denotes the first homology group of $M$ with coefficients in $\R$ and $\Gamma$ is a discrete subgroup of $H_1(M)$ whose definition we will not need here. Clearly, $\Homeo_0(M, \omega)$ is not simple when the mass-flow homomorphism is non-trivial.  This is indeed the case when $M$ is a closed surface other than the sphere.
Fathi proved that $\mathrm{ker}(\mathcal{F})$ simple if the dimension of $M$ is at least three.  
The following question is posed in \cite[Appendix A.6]{fathi}.

\begin{question}[Fathi]\label{que:massflow}
  Is  $\mathrm{ker}(\mathcal{F})$ simple in the case of surfaces?  In particular, is $\Homeo_0(\S^2, \omega)$, the group of area and orientation preserving homeomorphisms of the sphere, simple?
\end{question}

This question remains open.  We remark that the two-sphere is the only closed manifold for which the simplicity question for volume-preserving homeomorphisms is unknown.

 There exists a circle of classical ideas and techniques, generally attributed to Epstein and Thurston (sometimes referred to as ``Thurston tricks''), which, under certain conditions, allow one to  conclude that simplicity holds on a large class of manifolds once it is established for a single manifold.  According to \cite{Bounemoura}, these classical methods can be used to prove that  $\mathrm{ker}(\mathcal{F})$ is simple if and only if $\Homeo_c(\D, \omega)$ is perfect; this is the statement of \cite[Proposition 4.1.7]{Bounemoura};  %
as a consequence of Corollary \ref{corol:not_perfect}, then, $\mathrm{ker}(\mathcal{F})$ is in principle non-simple for all closed surfaces.  However,  we have been unable to independently verify \cite[Proposition 4.1.7]{Bounemoura}, and indeed the author of \cite{Bounemoura} has  withdrawn the stated proposition in a recent private communication with us. Hence, as far as we know, Question \ref{que:massflow} remains open.  

It is plausible that the aforementioned classical methods could be used to settle the above question.  Another approach would be to adapt the methods of our paper; this requires further development of the theory of PFH spectral invariants.

\subsection{$C^0$-symplectic topology and simplicity in higher dimensions}\label{sec:simplicity_high_dim}
 From a symplectic viewpoint, a natural generalization of area-preserving homeomorphisms to higher dimensions is given by {\bf symplectic homeomorphisms}.  These are, by definition, those homeomorphisms which appear as $C^0$ limits of symplectic diffeomorphisms.  By the celebrated rigidity theorem \cite{Eliashberg, Gromov} of Eliashberg and Gromov, a smooth symplectic homeomorphism is a symplectic diffeomorphism. These homeomorphisms form the cornerstone of the field of {\bf $C^0$-symplectic topology} which explores continuous analogues of smooth symplectic objects; see for example \cite{muller-oh, banyaga_homeo, HLS1, HLS2, BuOp}.

The connection between $C^0$-symplectic topology and the simplicity conjecture is formed by the fact that, in dimension two, symplectic homeomorphisms are precisely the area and orientation preserving homeomorphisms of surfaces.   Indeed, as we mentioned in Section \ref{sec:intro_main_theo}, the simplicity conjeture has been one of the driving forces behind the  development of $C^0$-symplectic topology; for example, the articles \cite{muller-oh, Oh10, Viterbo-uniqueness, buhovsky-seyfaddini, humiliere,  LeRoux-6Questions, LeRoux10, EPP, Sey13, Sey-displaced} were, at least partially, motivated by this conjecture.

The connection to $C^0$-symplectic topology motivates the following
generalization of Question \ref{que:disc}.

\begin{question}\label{que:simplicity_ball}
Is $\overline{\Symp}_c(\D^{2n}, \omega)$, the group of compactly supported symplectic  homeomorphisms of the standard ball, simple?\footnote{An argument involving the Alexander isotopy  shows that $\overline{\Symp}_c(\D^{2n}, \omega)$ is connected.} 
\end{question}

As we will now explain Question \ref{que:massflow} admits a natural generalization to higher dimensions as well. To keep our discussion simple we will suppose that $(M, \omega)$ is a closed symplectic manifold.  However, this assumption is not necessary and the question below can be reformulated for non-closed manifolds too.

On a symplectic surface $(M, \omega)$, the group
$\mathrm{ker}(\mathcal{F})$ discussed in the above section is often referred to as the group of {\bf Hamiltonian
homeomorphisms}  and is denoted by $\overline{\Ham}(M, \omega)$; see for
example \cite{lecalvez}.  The reason for this terminology is that it
can be shown that $\mathrm{ker}(\mathcal{F})$ coincides with the $C^0$
closure of Hamiltonian diffeomorphisms.  Hence, in this language, Question \ref{que:massflow} may be rephrased as the question of whether or not the group of Hamiltonian homeomorphisms is simple. 
On higher dimensional symplectic manifolds, the elements of the $C^0$
closure of $\Ham(M,\omega)$ are also called Hamiltonian
homeomorphisms and have been studied extensively in $C^0$ symplectic topology; see, for example, \cite{muller-oh, BHS1, BHS2, kawamoto}.

\begin{question}\label{question:ham-bar-simple}
  Is   $\overline{\Ham}(M, \omega)$ a simple group?
\end{question}
In comparison, Banyaga's theorem states that the group of Hamiltonian diffeomorphisms is simple for closed $M$.

\subsection{Finite energy homeomorphisms}
\label{sec:fhomeoq}

The group of finite energy homeomorphisms, $\FHomeo(M, \omega)$, can be defined on arbitrary symplectic manifolds; the construction is analogous to what is done in Section \ref{sec:prop-norm-subgr}. It forms a normal subgroup of $\overline{\Ham}(M, \omega)$.  However, we do not know if it is a proper normal subgroup.   Infinite twists can also be constructed on arbitrary symplectic manifolds: the construction of $\phi_f$, described in Section \ref{sec:calabi_inf_twist}, admits a  generalization to $\D^{2n}$.    And the analogue of Equation \eqref{eq:infinite_calabi}, the condition for having ``infinite" Calabi invariant, can also be formulated in higher dimensions.

\begin{question}\label{que:inf_twist_higher_dimension}
 Is it true that infinite twists, which satisfy the higher dimensional analogue of Equation \eqref{eq:infinite_calabi},   are not finite energy homeomorphisms?
\end{question}

Clearly, a positive answer to this question would settle all of the above simplicity questions.  In the case of surfaces, one could hope to apply the machinery of PFH to approach this question.  However, a serious obstacle arises in higher dimensions where PFH, and the related Seiberg-Witten theory, have no known generalization.

We now return to the case of the disc, where we know that $\FHomeo_c(\D, \omega)$ is a proper normal subgroup of $\Homeo_c(\D, \omega)$.   This immediately gives rise to several interesting questions about $\FHomeo_c(\D, \omega)$. 

\begin{question}  
  Is $\FHomeo_c(\D, \omega)$ simple?
\end{question}

As was mentioned in Remark \ref{rem:Hameo}, the Oh-M\"uller group $\Hameo_c(\D, \omega)$, which we introduce below,  is a subgroup of $\FHomeo_c(\D, \omega)$, and it can easily be checked that it is a normal subgroup.

\begin{question}
  Is $\Hameo_c(\D, \omega)$ a proper normal subgroup of $\FHomeo_c(\D, \omega)$?
\end{question}

Another interesting direction to explore is the algebraic structure of the quotient $\Homeo_c(\D, \omega) / \FHomeo_c(\D, \omega)$.  At present
we are not able to say much beyond the fact that this quotient is abelian; see Proposition \ref{prop:commutators}.  Here are two sample questions.

\begin{question}
Is the quotient $\Homeo_c(\D, \omega) / \FHomeo_c(\D, \omega)$ torsion-free?  Is it divisible?
\end{question}

It would also be very interesting to know if  the quotient $\Homeo_c(\D, \omega) / \FHomeo_c(\D, \omega)$ admits a geometric interpretation.

\subsection{Hutchings' conjecture and extension of the Calabi invariant}\label{sec:extending_Calabi}
Ghys \cite{Ghys_ICM}, Fathi and Oh \cite[Conjecture 6.8]{muller-oh} have asked if the Calabi invariant extends to either of $\Hameo_c(\D, \omega)$ or $\Homeo_c(\D, \omega)$.  It seems natural to add $\FHomeo_c(\D, \omega)$ to the list.

\begin{question}
Does the Calabi invariant admit an extension to any of $\Hameo_c(\D, \omega)$, $\FHomeo_c(\D, \omega)$, or $\Homeo_c(\D, \omega)$?
\end{question}

The reason for expecting that such extensions could exist is that the Calabi invariant exhibits certain topological properties:  In an unpublished manuscript  \cite{Fathi-Calabi}, Fathi shows that the Calabi invariant admits an interpretation as  an average rotation number and uses this to extend the Calabi invariant to Lipschitz homeomorphisms; see also \cite{humiliere} for a different proof.  Fathi's interpretation of the Calabi invariant enables a proof of  topological invariance for Calabi's invariant \cite{Gambaudo-Ghys}, i.e.\ it is invariant under conjugation by area-preserving homeomorphisms.

As was mentioned in Remark \ref{rem:hutchings_conjecture}, Hutchings has conjectured that the asymptotics of PFH spectral invariants recover the Calabi invariant for Hamiltonian diffeomorphisms.  This conjecture, combined with Theorem \ref{theo:C0_continuity}, implies that the Calabi invariant admits an extension to $\Hameo_c(\D, \omega)$; we verify this below.  We remark that extending the Calabi invariant to $\Hameo_c(\D, \omega)$ would imply that both $\Hameo_c(\D, \omega)$ and $\FHomeo_c(\D, \omega)$ are non-simple.  Non-simplicity of $\Hameo_c(\D, \omega)$ is clear; that of $\FHomeo_c(\D, \omega)$ follows from the fact that if $\Hameo_c(\D, \omega)$ is not a proper normal subgroup then $\FHomeo_c(\D, \omega)$ would coincide with the non-simple $\Hameo_c(\D, \omega)$.

We end our article by verifying that Hutchings' conjecture implies that the Calabi invariant extends to $\Hameo_c(\D, \omega)$.  Let us begin with the definition\footnote{The definition we have given here is a slight variation of the one in \cite{muller-oh}; it can easily be checked that if $\phi$ is a hameomorphism in the sense of  \cite{muller-oh}, then it is also a hameomorphisms in the sense described here.} of $\Hameo_c(\D, \omega)$.  We say $\phi \in \Homeo_c(\D, \omega)$ is a {\bf hameomorphism} if there exists a sequence of smooth Hamiltonians $H_i \in C^{\infty}_c(\S^1 \times \D)$ and a continuous $H \in C^{0}_c(\S^1 \times \D)$ such that 

 $$\varphi^1_{H_i} \xrightarrow{C^0} \phi, \text{ and } \| H- H_i \|_{(1, \infty)} \to 0.$$   The set of all hameomorphisms is denoted by $\Hameo_c(\D, \omega)$.\footnote{Oh and M\"uller use the terminology \emph{Hamiltonian homeomorphisms} for the elements of $\Hameo_c(\D, \omega)$.  We have chosen to avoid this terminology because in the surface dynamics literature it is commonly used for referring to homeomorphisms which arise as $C^0$ limits of Hamiltonian diffeomorphisms.}  It defines a normal subgroup of $\Homeo_c(\D, \omega)$ which is clearly contained in $\FHomeo_c(\D, \omega)$.

Take $\phi \in \Hameo_c(\D, \omega)$ and $H \in C^{0}_c(\S^1 \times \D)$ 	as in the previous paragraph;  we denote $ \varphi_H := \phi$. Define 
\begin{equation}\label{eqn:Calabi_Hameo}
\Cal(\varphi_H) := \int_{\S^1} \int_\D H \, \omega \, dt.
\end{equation}
We first show that $\Cal$ is well-defined.  First, note that because $\Cal$ is a homomorphism on $\Diff_c(\D,\omega)$, 
 to show this it suffices to show that if $\varphi_H = \id$, then 
\begin{equation}
\label{eqn:somanyequations}
\int_{\S^1} \int_\D H \, \omega \, dt = 0.
\end{equation}

Suppose that $\varphi_H = \id$  and fix a sequence $(H_1, H_2, \ldots,)$ for $\varphi_H$ as in the definition of $\Hameo_c(\D, \omega)$.  
\begin{claim}\label{cl:last_claim}
For all $i$, we have 
$\vert \frac{ c_d } { d } (\varphi^1_{H_i}) \vert \leq \| H- H_i \|_{(1, \infty)} .$
\end{claim}
\begin{proof}
By Hofer continuity of PFH spectral invariants, we have 
$\vert \frac{ c_d } { d } (\varphi^1_{H_j})  -  \frac{ c_d } { d } (\varphi^1_{H_i}) \vert \leq \| H_j- H_i \|_{(1, \infty)} $, for all $i, j$.  Fixing $i$ and taking the limit of this inequality, as $j \to \infty$,  yields the claim, by %
Theorem \ref{theo:C0_continuity} and item $4$ of Theorem \ref{thm:PFHspec_initial_properties},
since $\varphi^1_{H_j} \xrightarrow{C^0} \id$.  
\end{proof}

We now establish \eqref{eqn:somanyequations}. For all $i, d \in \N$ we have 
\begin{align*}
  \left\vert \int_{\S^1} \int_\D H \, \omega \,  dt \right\vert   
&\leq  \left\vert \int_{\S^1} \int_\D  H \, \omega \, dt  \; - \Cal(\varphi^1_{H_i} ) \right\vert  \\ &\qquad\qquad+   \left\vert   \Cal(\varphi^1_{H_i}) -  \frac{ c_d } { d } (\varphi^1_{H_i})  \right\vert +  \left\vert \frac{ c_d } { d } (\varphi^1_{H_i}) \right\vert  \\
 &\leq  \| H- H_i \|_{(1, \infty)}  +  \left\vert   \Cal(\varphi^1_{H_i}) -  \frac{ c_d } { d } (\varphi^1_{H_i}) \right\vert + \left\vert \frac{ c_d } { d } (\varphi^1_{H_i}) \right\vert \\
 &\leq 2 \| H- H_i \|_{(1, \infty)}  +  \left\vert   \Cal(\varphi^1_{H_i}) - \frac{ c_d } { d }(\varphi^1_{H_i}) \right\vert,
\end{align*} 
where the last inequality follows from Claim \ref{cl:last_claim}.  Now take any $\varepsilon > 0$.   There exists $N \in \N$ such that if $i \geq N$, then $ \| H- H_i \|_{(1, \infty)} \leq \frac12 \varepsilon$.  Hence, for $i \geq N$, we have 
\begin{align*}
\left\vert \int_{\S^1} \int_\D H \, \omega \, dt \right\vert \leq \varepsilon + \left\vert   \Cal(\varphi^1_{H_i}) - \frac{ c_d } { d }(\varphi^1_{H_i}) \right\vert.
\end{align*}
Assuming Hutchings's conjecture, we can now conclude that $
\vert \int_{\S^1} \int_\D H \, \omega \, dt \vert \leq \varepsilon,$
hence \eqref{eqn:somanyequations}, hence %
the claimed extension of $\Cal$.

It remains to show that 
\[ \Cal : \Hameo_c(\D, \omega) \rightarrow \R\] 
is indeed a homomorphism.  Having shown that \eqref{eqn:Calabi_Hameo} is well-defined, this has in fact already been shown in \cite{Oh10} and so we will only provide a sketch of the argument. Take $\varphi_H, \varphi_G \in \Hameo_c(\D, \omega)$.  Without loss of generality, we may suppose that $H(t,x)$ and $G(t,x)$  vanish for $t$ near $0 \in \S^1$; this can be achieved by replacing $H$ with the reparametrization $\rho'(t) H(\rho(t), x)$, where $\rho : [0,1] \to [0,1]$ coincides with $0$ near $0$ and with $1$ near $1$; see \cite[p.\ 31]{Polterovich2001} for more details on the reparametrization argument.

It can be checked that $\varphi_H \circ \varphi_G = \varphi_{K}$ where $K$ is the concatenation of $H$ and $G$ given by the formula

   \[K(t,x) =
     \begin{cases}
       2 H(2t,x), &\quadif  t\in [0, \frac{1}{2}]\\
       2 G(2t-1,x),&\quadif t\in [\frac{1}{2}, 1]
     \end{cases}.
   \]
It follows immediately from the above formula, and Equation \eqref{eqn:Calabi_Hameo}, that $\Cal(\varphi_H \circ \varphi_G)  = \Cal(\varphi_H) + \Cal(\varphi_G)$.  Verification of the fact that $\Cal(\varphi_H^{-1} )= - \Cal(\varphi_H)$ is similar and so we omit it.

 \bibliographystyle{alpha}
 \bibliography{c0PFH}


 \newpage
{\small

\medskip
\noindent Dan Cristofaro-Gardiner\\
Mathematics Department \\ 
University of California, Santa Cruz \\
1156 High Street, Santa Cruz, California, USA\\
School of Mathematics \\
Institute for Advanced Study\\
1 Einstein Drive, Princeton, NJ, USA \\
{\it e-mail}: dcristof@ucsc.edu
\medskip

\medskip
\noindent Vincent Humili\`ere \\
\noindent CMLS, CNRS, Ecole Polytechnique, Institut Polytechnique de Paris, 91128
Palaiseau Cedex, France.\\
{\it e-mail:} vincent.humiliere@polytechnique.edu
\medskip

\medskip
 \noindent Sobhan Seyfaddini\\
\noindent Sorbonne Universit\'e, Universit\'e de Paris, CNRS, Institut de Math\'ematiques de Jussieu-Paris Rive Gauche, F-75005 Paris, France.\\
 {\it e-mail:} sobhan.seyfaddini@imj-prg.fr

}

\end{document}